\documentclass{article}
\usepackage[utf8]{inputenc}
\usepackage[english]{babel}
\usepackage{amsmath, amsfonts, amssymb}
\usepackage{enumitem}
\usepackage{stmaryrd}
\usepackage{graphicx}
\usepackage{color}
\usepackage{dsfont}
\usepackage[all]{xy}
\usepackage{csquotes}

\usepackage{geometry}
\usepackage{amsthm}
\usepackage{fullpage}
\theoremstyle{plain}
\usepackage{subcaption}
\usepackage{hyperref}

\usepackage{appendix}

\usepackage{nccmath}

\title{Cutting directed ribbon graphs and recursion for volumes of the combinatorial moduli spaces.}

\begin{document}
\newtheorem{conj}{Conjecture}
\newtheorem{Def}{Definition}
\newtheorem{prop}{Proposition}
\newtheorem{thm}{Theorem}
\newtheorem{lem}{Lemma}
\newtheorem{cor}{Corollary}
\newtheorem{rem}{Remark}
\newtheorem{question}{Question}

\newcommand\MMS{\text{Mirzakhani-McShane}}
\newcommand\vol{\text{vol}}
\newcommand\Lc{\mathcal{L}}
\newcommand\Br{\text{Br}}

\newcommand\tb{\mathbf{t}}
\newcommand\qb{\mathbf{q}}
\newcommand\rk{\text{rk}}
\newcommand\End{\text{End}}
\newcommand\Res{\text{Res}}
\newcommand\C{\mathbb{C}}
\newcommand\Sf{\mathfrak{S}}
\newcommand\R{\mathbb{R}}
\newcommand\Rp{\mathbb{R}_{\ge 0}}
\newcommand\Rpp{\mathbb{R}_{>0}}
\newcommand\Q{\mathbb{Q}}
\newcommand\Z{\mathbb{Z}}
\newcommand\N{\mathbb{N}}
\newcommand\Npp{\mathbb{N}_{>0}}
\newcommand\T{\mathcal{T}}
\newcommand\sign{\text{signe}}
\newcommand\Aut{\text{Aut}}
\newcommand\Autt{\widetilde{\text{Aut}}}
\newcommand\Sym{\text{Sym}}
\newcommand\Stab{\text{Stab}}
\newcommand\Ker{\text{Ker}}
\newcommand\Proj{\text{Proj}}
\newcommand{\Pc}{\mathcal{P}}
\newcommand\Hom{\text{Hom}}
\newcommand\re{\text{Re}}
\newcommand\im{\text{Im}}
\newcommand\db{\overline{d}}

\newcommand\Tc{\mathcal{T}^{comb}}
\newcommand\Tcs{\mathcal{T}^{comb,*}}
\newcommand\Tcb{\overline{\mathcal{T}}^{comb}}
\newcommand\M{\mathcal{M}}
\newcommand\Mb{\overline{\mathcal{M}}}
\newcommand\Mc{\mathcal{M}^{comb}}
\newcommand\CM{\mathcal{CM}}
\newcommand\CMb{\overline{\mathcal{CM}}}
\newcommand\Mcs{\mathcal{M}^{comb,*}}
\newcommand\Mcb{\overline{\mathcal{M}}^{comb}}
\newcommand\HT{\mathcal{HT}}
\newcommand\Hc{\mathcal{H}}

\newcommand\Zo{Z^{\circ}}
\newcommand\lambdao{\lambda^{\circ}}
\newcommand\pr{\text{pr}}
\newcommand\Kb{\mathbf{K}}
\newcommand\phid{\overline{\phi}}
\newcommand\phib{\phi^{\bullet}}
\newcommand\Kd{\overline{K}}
\newcommand\Kbull{K^{\bullet}}
\newcommand\Kbd{\overline{\mathbf{K}}}
\newcommand\uf{\mathfrak{u}}
\newcommand\s{\mathfrak{s}}
\newcommand\lf{\mathfrak{l}}
\newcommand\cf{\mathfrak{c}}
\newcommand\alphab{\boldsymbol{\alpha}}
\newcommand\Gb{\mathbf{G}}
\newcommand\Ada{\overline{\mathcal{A}}}
\newcommand\Aba{\mathcal{A}^{\bullet}}
\newcommand\Mba{\mathcal{M}^{\bullet}}
\newcommand\Mda{\overline{\mathcal{M}}}
\newcommand\Dh{\hat{\mathcal{D}}}
\newcommand\eb{\mathbf{e}}
\newcommand\ez{\eb_{\emptyset}}
\newcommand\Gbull{G^{\bullet}}
\newcommand\Zbull{Z^{\bullet}}
\newcommand\Zd{\overline{Z}}
\newcommand\Rot{\tilde{R}^{\circ}}
\newcommand\W{\mathcal{W}}
\newcommand\Hb{\mathbf{H}}

\newcommand\topo{\textbf{top}}
\newcommand\topod{\overline{\textbf{top}}}

\newcommand\topob{\textbf{top}^{\bullet}}
\newcommand\bord{\textbf{bord}}
\newcommand\bords{\textbf{bord}^{\s}}
\newcommand\bordc{\textbf{bord}^{\cf}}
\newcommand\bordbc{\textbf{bord}^{\bullet,\cf}}
\newcommand\bordsc{\textbf{bord}^{\s,\cf}}
\newcommand\bordl{\textbf{bord}_{\lf}}
\newcommand\bordb{\textbf{bord}^{\bullet}}
\newcommand\bordbl{\textbf{bord}^{\bullet}_{\lf}}
\newcommand\bordo{\textbf{bord}^{\circ}}
\newcommand\borddo{\overline{\textbf{bord}}^{\circ}}
\newcommand\bordd{\overline{\textbf{bord}}}
\newcommand\bordoc{\textbf{bord}^{\circ,\cf}}
\newcommand\bordos{\textbf{bord}^{\circ,\s}}
\newcommand\bordosc{\textbf{bord}^{\circ,\s,\cf}}
\newcommand\bordol{\textbf{bord}^{\circ}_{\lf}}
\newcommand\bordobl{\textbf{bord}^{\circ,\bullet}_{\lf}}
\newcommand\bordob{\textbf{bord}^{\circ,\bullet}}
\newcommand\borddob{\overline{\textbf{bord}}^{\circ,\bullet}}
\newcommand\borddb{\overline{\textbf{bord}}^{\bullet}}
\newcommand\stab{\textbf{stab}}
\newcommand\stabo{\textbf{stab}^{\circ}}
\newcommand\stabol{\textbf{stab}^{\circ}_{\lf}}
\newcommand\acyclod{\overline{\textbf{acycl}}}
\newcommand\acyclodmax{\overline{\textbf{acycl}}^*}
\newcommand\acyclol{\textbf{acycl}_{\lf}}
\newcommand\acyclogen{\textbf{acycl}^{*}}
\newcommand\acyclolp{\textbf{acycl}_{\lf,+}}
\newcommand\acyclo{\textbf{acycl}}
\newcommand\acyclos{\textbf{acycl}^*}
\newcommand\acyclomax{\textbf{acycl}}

\newcommand\Cyl{\mathfrak{C}}
\newcommand\Mo{M^{\circ}}
\newcommand\Mol{M^{\circ}_{\lf}}
\newcommand\Mdo{\overline{M}^{\circ}}
\newcommand\Md{\overline{M}}
\newcommand\Go{\mathcal{G}^{\circ}}
\newcommand\Gol{\mathcal{G}^{\circ}_{\lf}}
\newcommand\Gdo{\overline{\mathcal{G}}^{\circ}}
\newcommand\To{\mathcal{T}^{\circ}}
\newcommand\G{\mathcal{G}}
\newcommand\Gao{\Gamma^{\circ}}
\newcommand\Gado{\overline{\Gamma}^{\circ}}

\newcommand\Mod{\text{Mod}}
\newcommand\Irr{\text{Irr}}
\newcommand\irr{\text{irr}}


\newcommand\A{\mathcal{A}}
\newcommand\MA{\mathcal{MA}}
\newcommand\MAr{\mathcal{MA}_{\R}}
\newcommand\MAor{\mathcal{MA}_{\R}^0}
\newcommand\MAo{\mathcal{MA}^0}
\newcommand\MAz{\mathcal{MA}_{\Z}}
\newcommand\AC{\mathcal{AC}}
\newcommand\MS{\mathcal{MS}}
\newcommand\MSt{\widetilde{\mathcal{MS}}}
\newcommand\MStr{\widetilde{\mathcal{MS}}_{\R}}
\newcommand\MSr{\mathcal{MS}_{\R}}
\newcommand\MSz{\mathcal{MS}_{\Z}}
\newcommand\MF{\mathcal{MF}}
\newcommand\MFt{\widetilde{\mathcal{MF}}}
\newcommand\PMF{\mathcal{PMF}}
\newcommand\Si{\mathcal{S}}
\newcommand\Sit{\widetilde{\mathcal{S}}}
\newcommand\QT{\mathcal{QT}}
\newcommand\Zcal{\mathcal{Z}}
\newcommand\QM{\mathcal{QM}}
\newcommand\Qq{\mathcal{Q}}
\newcommand\St{\text{St}}
\newcommand\st{\text{st}}
\newcommand\h{\mathfrak{h}}


\newcommand\Rib{\text{Rib}}
\newcommand\Ribb{\overline{\text{Rib}}}
\newcommand\Ro{R^{\circ}}
\newcommand\So{S^{\circ}}
\newcommand\Rto{\tilde{R}^{\circ}}
\newcommand\Rt{\tilde{R}}
\newcommand\Sc{\mathbb{S}}
\newcommand\Prj{\mathbb{P}}
\newcommand\Met{\text{Met}}
\newcommand\Metb{\overline{\text{Met}}}
\newcommand\cut{\text{cut}}
\newcommand\rib{\text{rib}}
\newcommand\ribb{\overline{\text{rib}}}
\newcommand\Gad{\overline{\Gamma}}

\newcommand\Lbord{L_{\partial}}
\newcommand\Sub{\text{Sub}}
\newcommand\sub{\text{sub}}
\newcommand\stabdo{\overline{\text{stab}}^{\circ}}
\newcommand\subd{\overline{\text{sub}}}
\newcommand\Subd{\overline{\text{Sub}}}
\newcommand\D{\mathcal{D}}
\newcommand\Do{\mathcal{D}^{\circ}}
\newcommand\Dd{\overline{\mathcal{D}}}
\newcommand\Ddo{\overline{\mathcal{D}}^{\circ}}
\newcommand\Nc{\mathcal{N}}
\newcommand\No{\mathcal{N}}
\newcommand\Ndo{\overline{\mathcal{N}}^\circ}
\newcommand\Loop{\text{loop}}

\newcommand\Wall{\mathbf{Wall}}
\newcommand\Lambdab{\overline{\Lambda}}
\newcommand\Vmv{\vartheta}
\newcommand\Vmvo{\vartheta^{\circ}}
\newcommand\Vmvbo{\bar{\vartheta}^{\circ}}
\newcommand\Det{\text{Det}}
\newcommand\Td{\overline{\mathcal{T}}}
\newcommand\muMV{d\mu^{MV}}

\newcommand\Vo{V^{\circ}}
\newcommand\Qc{\mathcal{Q}}
\newcommand\Bil{\text{Bi}}
\newcommand\Vc{V^{comb}}

\newcommand\sle{\text{SLE}}

\newcommand{\keywords}[1]
{
  \small	
  \textbf{\textit{Keywords---}} #1
}
\author{Simon Barazer}
\maketitle

\begin{figure}[h!]
    \centering
    \includegraphics[width=0.3\linewidth]{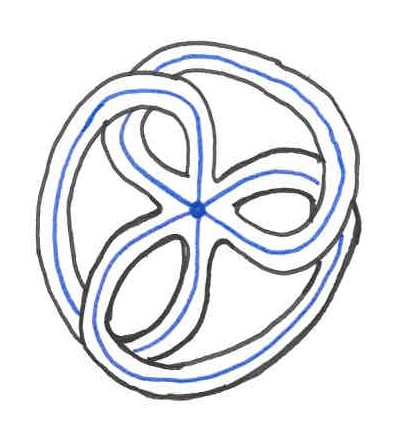}
    
\end{figure}
\begin{abstract}
In this paper we study metric ribbon graphs, in particular, directed metric ribbon graphs. These ribbon graphs are dual to bipartite maps and appear in the context of Abelian differentials. We prove that it is possible to decompose a directed ribbon graph into a family of ribbon graphs with one vertex, by performing surgeries along appropriate multi-curves. The decomposition is canonical and we call it \textit{acyclic decomposition} due to a condition on the stable graphs that encode the surgeries. This result provides a recursion scheme for volumes of moduli space of directed metric ribbon graphs, we give explicitly the recursion in the case of $4-$valent metric ribbon graphs. In a particular case, we give applications to count of Grothendieck \textit{Dessins d'enfant}.
 \end{abstract}


\newpage
\tableofcontents

\newpage

\section{Introduction:}
The study of ribbon graphs started more than thirty years ago, these objects received more attention with their use in the proof of Witten conjecture \cite{witten1990two} by M. Kontsevich in \cite{kontsevich1992intersection}. Ribbon graphs are graphs embedded in surfaces with boundaries, and the surface retracts on the graph. They have different names in different fields, they are also called fatgraphs and maps in combinatorics. We give a self-contained definition in Section \ref{section_ribbon}; but the reader could also uses \cite{kontsevich1992intersection}, \cite{lando2004graphs} or \cite{zvonkine2002strebel}. For each $(g,n)$, we denote $\Mc_{g,n}$ the combinatorial moduli space of metric ribbon graphs of genus $g$ and with $n$ boundary components. Since the works of K. Strebel \cite{strebel1984quadratic}, we know that $\Mc_{g,n}$  is a cell complex that is homeomorphic to $\M_{g,n}\times \Rpp^n$, with $\M_{g,n}$ the moduli space of marked Riemann surfaces. There is a function $\Lbord : \Mc_{g,n}\to \Rp^n$ that associates to a metric ribbon graph the length of its boundary components. In \cite{kontsevich1992intersection}, M. Kontsevich computes the volume $V_{g,n}(L)$ of the level set $\M_{g,n}(L)=\Lbord^{-1}(L)$ for its natural volume form, proving that these functions are polynomials in $L$ with coefficients given by the intersection theory of $\Mb_{g,n}$. The polynomials $V_{g,n}(L)$ appear in different contexts; in \cite{delecroix2021masur}, the authors use them to express the Masur-Veech volumes of the principal strata in moduli spaces of quadratic differentials. In \cite{andersen2023topological} a recursion was obtained for Masur-Veech volumes by using Masur-Veech polynomials. To compute volumes of other strata in moduli spaces of quadratic and Abelian differentials, we could use volumes of substrata in moduli spaces of metric ribbon graphs. This paper is a step in this direction; we study directed ribbon graphs\footnote{We remark that we prefer the terminology directed versus oriented because oriented refers to the orientation of topological space, and unorientable ribbon graphs can exist, which leads to confusion.}. Volumes of moduli spaces directed metric ribbon graphs can be used to express Masur-Veech volumes of various strata in moduli spaces of Abelian differentials. We give in this paper a procedure to compute volumes of moduli space of directed ribbon graphs inductively by finding nice decompositions of the ribbon graphs. Completed with the  cases of directed ribbon graphs with only one vertex (which have been studied in detail in \cite{yakovlev2022contribution} by I. Yakovlev, who obtained impressive results), our results allow us to compute all the possible volumes of moduli spaces of directed metric ribbon graphs.

\paragraph{Directed ribbon graphs, moduli spaces, and their volumes:}
 Using surgeries on trivalent ribbon graphs, it was possible for the authors in \cite{andersen2020kontsevich} to obtain a Mirzakhani Mac-Shane formula for trivalent ribbon graphs and then derive a recursion for the polynomials $V_{g,n}$. It gives another proof of results in \cite{kontsevich1992intersection} and \cite{delecroix2021masur}. Nevertheless, outside this case, very little was known about volumes of sub-strata in the combinatorial moduli space. We focus on the case of directed ribbon graphs; these are ribbon graphs with a coherent orientation of the edges (see definition \ref{def_orientation_ribbon} or figure \ref{intro_fig_ribbon} for a picture); by duality, these graphs are related to bipartite maps. We use the notation $R$ for ribbon graphs and $\Ro$ for ribbon graphs with a direction (and we use the upper script $\circ$ for everything that is directed). The direction of the edges induces a direction of the boundary components; some of them are oriented according to the orientation induced by the underlying surface, and they are labeled by $+$, the others have the opposite orientation and are labeled by $-$; we call such surfaces with signs on the boundary a directed surface (we denote it $\Mo$ instead of $M$). We consider the moduli space $\Mc_{g,n^+,n^-}$ of directed metric ribbon of genus $g$ with $n^+$ positive and $n^-$ negative boundary components. This space is stratified by the order of the vertices of the ribbon graphs; we mainly focus on the principal strata that correspond to four-valents directed ribbon graphs, in these cases, our results are easier to state. For directed ribbon graphs the boundary lengths are not independent; the sum of the positive boundaries equals the sum of the negative ones. We consider the level set $\Mc_{g,n^+,n^-}(L^+|L^-)$ of directed ribbon graphs, where $L^{\pm}$ is the length of the positive/negative boundaries \footnote{These volumes correspond to principal strata, which are dense open set in the moduli spaces of directed ribbon graphs. We also study cases of sub-strata's but in general the volumes are not continuous.}. There is a natural volume form on this space, and we consider the volumes as a function $V_{g,n^+,n^-}(L^+|L^-)$ of two sets variables; we prove the following result (Theorem  \ref{thm_polynomial_Vgnn} in the text).

\begin{thm}
\label{into_thm_polynomial_Vgnn}
    The functions $V_{g,n^+,n^-}(L^+|L^-)$ are continuous, homogeneous, piecewise polynomials of degree $4g-3+n^++n^-$ defined on $\Lambda_{n^+,n^-}$.
\end{thm}
The set $\Lambda_{n^+,n^-}$ is given by the condition $\sum_i L_i^+=\sum_i L_i^-$. It is possible to be more precise on the wall structure of the functions $V_{g,n^+,n^-}(L^+|L^-)$, it is similar to the ones of double Hurwitz numbers studied in \cite{hahn2018wallcrossing}. 
\begin{figure}
    \centering
    \includegraphics[width=0.3\linewidth]{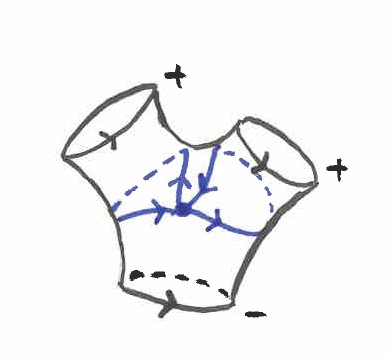}
    \caption{A directed ribbon graph on a directed surface of type $(0,2,1)$ }
    \label{intro_fig_ribbon}
\end{figure}

\paragraph{Surgeries on directed ribbon graphs and acyclic decompositions:}
We investigate, in some sense, the geometry
of metric ribbon graphs and will use ideas that are close to the ones of \cite{andersen2020kontsevich} and study curves on ribbon graphs in
order to decompose them in simple pieces. To perform surgeries on ribbon graphs that are not trivalent, we  introduce in Section \ref{subsection_admissible}, Definition \ref{def_admissible_fol}, the notion of admissible multi-curves
on a ribbon graph. This is a particular subset of homotopy classes of multi-curves on which it is possible to define the
twist flow and perform surgeries. Admissible multi-curves are the ones that do not split a vertex when we cut the ribbon graph along them. We can complete these spaces of admissible curves by considering admissible foliations. We study the spaces of admissible foliations on any ribbon graph in Section \ref{subsection_admissible} and we provide coordinates in Proposition \ref{prop_coord_x}. When we cut a directed ribbon graph $\Ro$ along an admissible multi-curve $\Gamma$ the result is a family of directed ribbon graphs $\Ro_{\Gamma}$, each of these graphs has signs on its boundary components, and two boundary components glued along a curve must have opposite signs. Then, when we perform surgeries on directed ribbon graphs, stable graphs that encode the surgeries have an additional structure; they are directed stable graphs (see Definition \ref{def_directed_stable_graph} or figure \ref{intro_directed_stab}) in the sense that the edges of the graphs are directed. In the sequel, acyclic directed stable graphs will play an important role.
\begin{figure}[h!]
    \centering
    \includegraphics[width=0.3\linewidth]{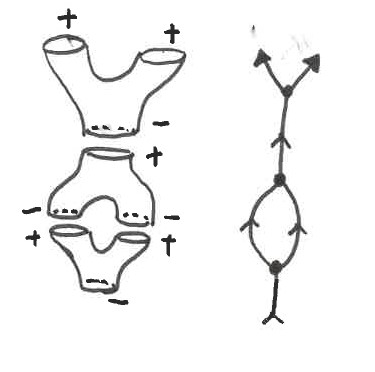}
    \caption{A directed acyclic stable graph.}
    \label{intro_directed_stab}
\end{figure}
By studying admissible multi-curves, we prove that: if we fix a directed ribbon graph $\Ro$ (with at least two vertices) and a vertex $v$ of $\Ro$, we can find an admissible curve that extracts this vertex. In other words, after surgeries along the multi-curve there is a component of the graph that contains $v$ and no other vertices. Moreover, with some assumptions, the multi-curve is unique. The detailed statement is given below and corresponds to Theorem \ref{thm_acycl_curve} in the text: 

\begin{thm}
\label{intro_thm_acycl_curve}
Let $\Ro$ be a directed ribbon graph with at least two vertices. For each vertex $v$, there exists a unique admissible primitive multi-curve $\Gamma_v^{+}$ such that:
\begin{itemize}
\item The stable graph $\Go_v$ of $\Gamma_v^+$ contains a component $c_v$ that spares $v$ from the rest of the graph.
\item All the curves in $\Gamma_v^+$ are boundaries of $c_v$.
\item $c_v$ is glued along negative boundaries only.
\end{itemize}
\end{thm}

\begin{rem}
Theorem \ref{intro_thm_acycl_curve} is quite surprising; it means that the local structure of the graph around an even vertex admits some model that corresponds to a minimal ribbon graph. Moreover, it is possible to cut the graph around the vertex in a canonical way. This process allows us to recover the structure of the graph inductively on the number of vertices. The recursion scheme is very similar to a directed version of the topological recursion, as we will see later.
\end{rem}

To find the curve, we use the two-form on each stratum of metric ribbon graphs; it comes from the intersection
pairing in cohomology and was first used by M. Kontsevich in \cite{kontsevich1992intersection} (see Section \ref{subsection_symplectic}). When there is a vertex of even degree, this form can degenerate. We identify in Proposition \ref{prop_coord_x} the space of admissible foliation to the tangent space of a stratum of the moduli space. The admissible curves that are in the kernel of the two-form enjoy many interesting properties and are the ones that are used to decompose the surface in Theorem \ref{intro_thm_acycl_curve}. We prove Theorem \ref{intro_thm_acycl_curve} in Section \ref{subsection_acyclic_decomposition} and we use it to obtain a recursion for the volumes $V_{g,n^+,n^-}$. We can also iterate Theorem \ref{intro_thm_acycl_curve}, if we fix an enumeration of the vertices in a directed ribbon graph, we can remove these vertices one by one and we obtain a multi-curve. This multi-curve decomposes the ribbon graph into graphs with only one vertex. We call such graphs minimal graphs. As we mention, this multi-curve is associated to a directed stable graph, and the structure of the gluing's implies that this directed stable graph is acyclic. Moreover, the enumeration of the vertices induces an enumeration of the components of the stable graph which is compatible with the order relation induced by the acyclic structure. Such enumeration is called a linear order in the context of acyclic directed graphs.

\begin{thm}
\label{intro_thm_acycl_decomp}
Let $\Ro$ be a directed ribbon graph with a linear order; then there is a unique admissible primitive multi-curve $\Gamma$ such that
\begin{enumerate}
\item The components of $\Ro_{\Gamma}$ are minimal.
\item The directed stable graph $\Go$ associated to $\Gamma$ is acyclic.
\item The linear order on the vertices induces a linear order on the acyclic stable graph.
\end{enumerate}
\end{thm}

We call such decomposition an acyclic decomposition, what is surprising is the uniqueness. The acyclicity is essential, otherwise, if we remove this condition, there is an infinite number of decomposition  in Theorems \ref{intro_thm_acycl_curve} and \ref{intro_thm_acycl_decomp}. We give in figure \ref{intro_fig_acyclic} an example of acyclic decomposition. It appears that similar structures can be observed in other context of enumerative geometry, such as double Hurwitz numbers. We see in a forthcoming paper that acyclic decomposition is related to Cut-and-Join equations. We use the acyclic decomposition to prove Theorem \ref{into_thm_polynomial_Vgnn} by studying the contribution of each acyclic stable graph, by using Theorem \ref{thm_polynomial_ZGo} and Proposition \ref{prop_graph_expension}.

\begin{figure}[h!]
    \centering
    \includegraphics[width=0.8\linewidth]{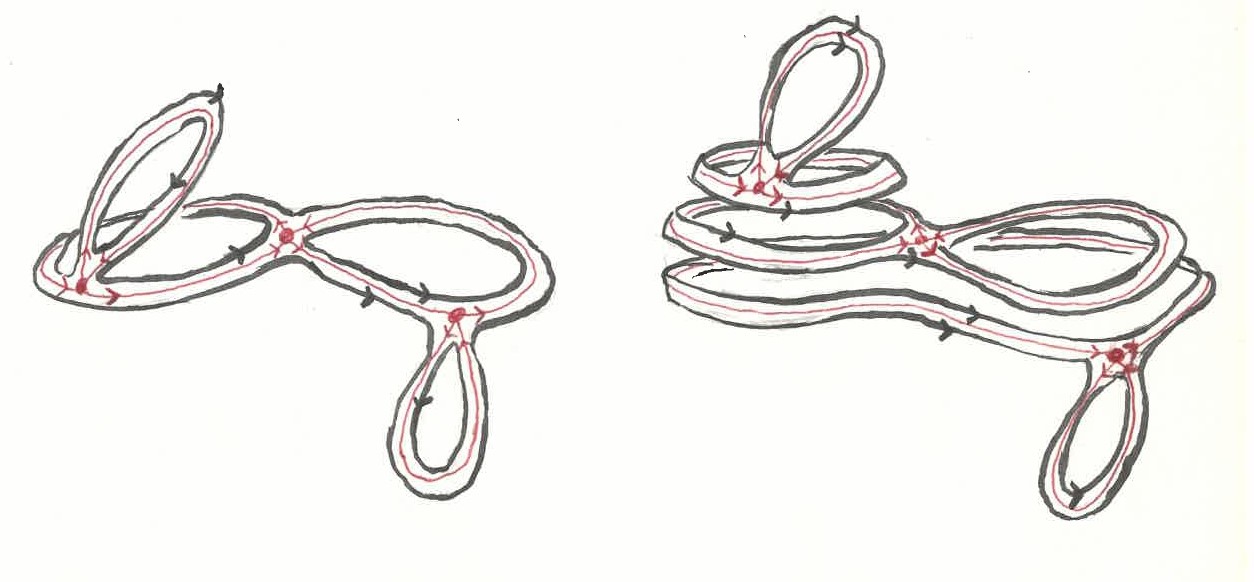}
    \caption{An acyclic decomposition.}
    \label{intro_fig_acyclic}
\end{figure}
\paragraph{Recursion for the volumes:}  By using Theorem \ref{intro_thm_acycl_curve} and techniques that were introduced by M. Mirzakhani
\cite{mirzakhani2007simple} and adapted in Proposition \ref{prop_integral_formula_2}, we can perform surgeries over moduli spaces of directed metric ribbon graphs, and this gives a recursion for the
volumes $V_{g,n^+,n^-}(L^+|L^-)$ of the principal stratum of directed metric ribbon graphs. The structure of the recursion is in some sense similar to the recursion of \cite{mirzakhani2007simple},\cite{andersen2017geometric} and \cite{andersen2020kontsevich} but we need
to take into account signs of boundary components. In the case of four-valents directed ribbon graphs, the recursion corresponds to removing a pair of pants that is glued along
its negative boundary components. In this case, the gluing's that we consider are the following (see figure \ref{intro_fig_pant_decomp})
\begin{enumerate}
    \item We  remove a pants that contains two positive boundaries $(i, j)$,
    \item We remove a pants that contains a positive boundary $i$ and a negative boundary $j$.
    \item We remove a pants with one positive boundary $i$ that is connected to the surface by two negative boundaries and does not separate the surface.
\item We remove a pants with one positive boundary $i$ that is connected to the surface by two negative boundaries
and separates the surface into two components.
\end{enumerate}
Then we obtain the following theorem, $[]_+$ means the positive part.

\begin{thm}
\label{intro_thm_recurrence_1}
For all values of the boundary lengths $L=(L^+,L^-)$, the volumes satisfy the recursion:
\begin{eqnarray*}
(2g-2+n)V_{g,n^+,n^-}(L^+|L^-)&=& 
\frac{1}{2}\sum_{i\neq j} (L_i^+~+L_j^+)~V_{g,n^+-1,n^-}(L_i^+ +L_j^+,L^+_{\{i,j\}^c}|L^-)\\
&+&\sum_{i}\sum_{j} [L_i^+-L_j^-]_+~V_{g,n^+,n^--1}([L_i^+-L_j^-]_+,L^+_{\{i\}^c}|L^-_{\{j\}^c})\\
&+&\frac{1}{2}\sum_{i} \int_0^{L_i^+} V_{g-1,n^++1,n^-}(x,L_i^+-x,L^+_{\{i\}^c}|L^-)~x(L_i^+-x)~ dx\\
&+&\frac{1}{2}\sum_{i}\sum_{\underset{I_1^{\pm}\sqcup I_2^{\pm}=I^{\pm}}{g_1+g_2=g}}^{'} x_1 x_2 V_{g_1,n_1^++1,n^-_1}(x_1,L^+_{I_1^+}|L^-_{I_1^-})~ V_{g_2,n_2^++1,n^-_2}(x_2,L^+_{I_2^+}|L^-_{I_2^-});
\end{eqnarray*}

where we use the notation\footnote{For $L\in \R^{n}$ and $I\subset \{1,...,n\}$ we denote $L_I=(L_i)_{i\in I}\in \R^I$.}
\begin{equation*}
x_l = \sum_{i\in I^{-}_l}L_i^{-}-\sum_{i\in I^{+}_l}L_i^{+}.
\end{equation*}
\end{thm}
Theorem \ref{intro_thm_recurrence_1} allow to
compute the volumes inductively by a recursion scheme by using the initial conditions 
\begin{equation*}
    V_{0,2,1}(L_1^+,L_2^+|L_1^-)=1 ~~~~~~~V_{0,1,2}(L_+^1|L_1^-,L_2^-)=1.
\end{equation*}
We do not need to enumerate the ribbon graphs and it is an
efficient way to compute the volumes for low values of $g$ and $n^\pm$. 

\begin{figure}
    \centering
    \includegraphics[width=0.5\linewidth]{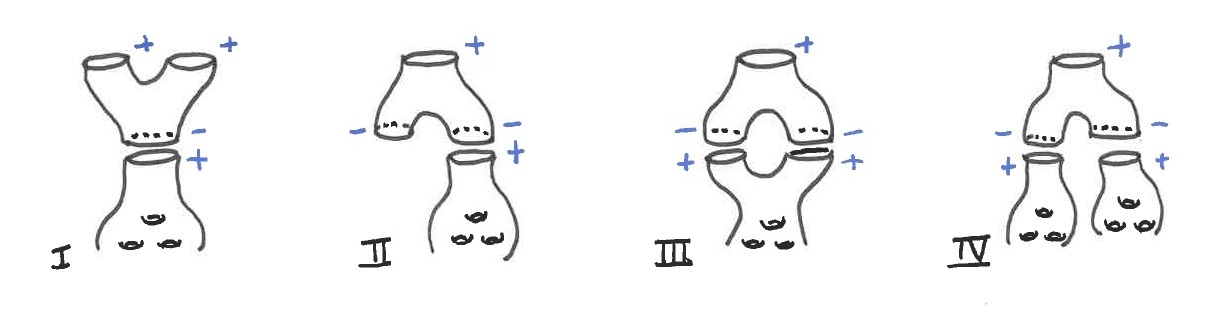}
    \caption{The four terms of the recursion in theorem \ref{intro_thm_acycl_curve}}
    \label{intro_fig_pant_decomp}
\end{figure}


\paragraph{Special case, Cut-and-Join equation and Grothendieck dessins d’enfants:} There is one case where
the piecewise polynomials $V_{g,n^+,n^-}$ and the recurrence of Theorem \ref{intro_thm_recurrence_1} are particularly simple. It corresponds to $n^-=1$, i.e., directed ribbon graphs with only one negative boundary component. In this case, we can write the volumes as functions $\Vo_{g,n}(L)$ of only the positive boundaries. The subfamily of surfaces with one negative boundary is preserved by the recursion, and the only allowed gluings are the ones of type $I$ and $III$. Moreover, the fact that there is only one negative boundary doesn’t provide enough ”space” for piece-wise polynomiality (see Theorem \ref{thm_recursion_1_bord}).

\begin{thm}
\label{intro_thm_recursion_1_bord}
Functions $\Vo_{g,n}$ are homogeneous polynomials of degree $4g-2+n$ and satisfy the following recursion:
\begin{eqnarray*}
(2g+n-1)\Vo_{g,n}(L)&=& \frac{1}{2}\sum_{i\neq j} (L_i~+L_j)~\Vo_{g,n-1}(L_i +L_j,L_{\{i,j\}^c})\\
&+&\frac{1}{2}\sum_{i} \int_0^{l_i} \Vo_{g-1,n+1}(x,L_i-x,L_{\{i\}^c})~x(L_i-x)~ dx
\end{eqnarray*}
\end{thm}

This recursion leads to a recursion for the coefficients of the polynomials. These are symmetric, and then can
be indexed by partitions $\mu = (0^{\mu(0)},1^{\mu(1)},...)$, we denote them $(c(\mu))_{\mu}$. We can form a generating function $\Zo$ given by
\begin{equation*}
\Zo(q,\tb)=\sum_\mu \frac{q^{\frac{d(\mu)+n(\mu)}{2}}\prod_i (i!) ^{\mu(i)}t_i^{\mu(i)}}{\prod_i \mu(i)!} c(\mu).
\end{equation*}
From Theorem \ref{intro_thm_recursion_1_bord}, we obtain the following result (see Theorem \ref{cor_cutandjoin_1_bord} in the text):
\begin{cor}
\label{intro_cor_cutandjoin_1_bord}
The series $\Zo(q,\tb)$ satisfies the following equation:
\begin{equation*}
\frac{\partial \Zo}{\partial q} = \frac{1}{2}\sum_{i,j}(i+j)t_it_j \partial_{i+j-1} \Zo+ \frac{1}{2}\sum_{i,j}(i+1)(j+1)t_{i+j+3} \partial_i\partial_j \Zo + \frac{t_0^2}{2},
\end{equation*}
with $\Zo(0,\tb)=0 $.
\end{cor}
The last equation is a Cut-and-Join equation; these equations appear at different places in enumerative geometry, such as enumeration of ramified covering over the sphere. We relate coefficients $c(\mu)$ to the count of a special kind of Grothendieck "dessins d'enfants".

\paragraph{Acknowledgment:} The author thanks several colleagues and friends for their support during this work. First of all I am grateful to my advisors, M. Kontsevich and A. Zorich for their support during my thesis and also the IHES for the hospitality. I am also gratefull to M. Liu, E. Goujard, V. Delecroix and Y. Yakovlev for their interest and their help. I want also thanks the jury of my thesis, G. Borot, J.E. Anderson, B. Eynard, D. Zvonkine, E. Garcia Faldes and S. Lelievre, for their encouragement and criticism. I am also grateful to many researchers from the IHES, IMO, and other members of the ERC ReNewQuantum, such as V. Fantini, I. Ren, A. Takeda, A. Giacchetto, R. Santaroubane, D. Lewansky...

\paragraph{Plan of the paper:}

\begin{enumerate}
    \item In section $1$, we recall definitions and results on multi-curves, foliations and quadratic differentials. We extend these results to surfaces with boundaries and discuss the relations with multi-arcs. We also introduce directed surfaces, directed multi-curves, and directed stable graphs; they play a central role in the text. Most of the material can be found in \cite{fathi2012thurston}, \cite{farb2011primer} or \cite{andersen2020kontsevich}; but we adapt it for our purposes.
    \item In section $2$, we study directed stable graphs in more detail and, more precisely, acyclic stable graphs. They are essential in Theorem \ref{intro_thm_acycl_decomp}. We also prove piece-wise polynomiality results that will be used to prove Theorem \ref{into_thm_polynomial_Vgnn}.
    \item In section $3$, we discuss ribbon graphs, we define directed ribbon graphs, and we study their properties. We also introduce cohomology of ribbon graphs, which allows us to symplifies some results. We recall the construction of the combinatorial moduli spaces, their stratifications and we define their volumes. We also recall the basic relationships between ribbon graph multi-arcs and foliations that can also be found in \cite{andersen2020kontsevich} or \cite{arbarello2011geometry}.
    \item In section $4$, we study curves on ribbon graphs, surgeries, and define admissible curves. In Proposition \ref{prop_coord_x}, we give coordinates for admissible foliations and admissible integral multi-curves. We finally study the properties of the twist flow allong an admissible foliation and relate it to the horocyclic flow for quadratic differentials.
    \item  In the section $5$, we prove the vertex surgery theorem (Theorem \ref{intro_thm_acycl_curve}) and the acyclic decomposition (Theorem \ref{intro_thm_acycl_decomp}). To do this, we define the Kontsevich two-form using the cohomology of ribbon graphs and we study degenerations of this form. We also define irreducible and minimal ribbon graphs and study their properties.
    \item In section $6$, we derive Theorems \ref{intro_thm_recurrence_1} and \ref{into_thm_polynomial_Vgnn} by using the results of the preceding sections and an integration formula over the combinatorial moduli space.
    \item In the last section, we study the special cases of graphs with one negative boundary component, relate it to Grothendieck dessins d'enfants, and we derive the Cut-and-Join equation (Theorem \ref{intro_cor_cutandjoin_1_bord}).
\end{enumerate}

\newpage

\section{Topological background and notations:}
\label{section_background}

\subsection{Surfaces and curves}

\paragraph{Surfaces, directed surfaces and decorations:}
\label{paragraph_surface_background}
We consider compact, oriented, topological surfaces $M$, with boundary. The boundary of $M$ is a finite union of circles; by abuse of notation, we denote $\partial M$ the set of boundary components $\pi_0(\partial M)$. A connected, oriented, compact surface is  characterized, up to homeomorphism by a pair 
$(g,n)$. With $g$, his genus and $n$ is the number of boundary components. We denote, $d_{g,n}=2g-2+n$
the opposite of the Euler characteristic. The mapping class group $\Mod(M)$ is the group of isotopy classes of homeomorphisms of $M$ that fix each boundary component. We denote $\bordc$ the category of connected, stable surfaces with boundaries (i.e., $n>0$),  $\bord$ corresponds to the categories of non-connected surfaces such that each connected component is in $\bordc$. A surface of first importance in $\bordc$ is the pair of pants; it corresponds to a sphere with three boundary components, i.e., of type $(0,3)$. If $M\in \bord$ and $c\in \pi_0(M)$, we denote $M(c)\in \bordc$ the corresponding surface and write 
\begin{equation*}
    M=\sqcup_{c}~M(c).
\end{equation*}
We give the following definition:
\begin{Def}
\label{def_dir_surfaces}
   A directed surface $\Mo$ is a pair $(M,\epsilon)$, with $M\in \bord$ and
\begin{equation*}
    \epsilon : \partial M \longrightarrow \{\pm 1\}
\end{equation*}
is a map, non-constant on each connected component of $M$.
\end{Def}
As before, we use the notation $\bordo,\bordoc...$ for different categories of stable directed surfaces. A directed surface $\Mo$ defines in an obvious way an usual surface $M$ by forgetting the direction. By assumption, a direction divides the boundary into two non-empty sets, $\partial^\pm\Mo$. The positive boundary components $\partial^+\Mo$ are in some sense the outputs, and the negative boundary components $\partial^-\Mo$ are the inputs of the surface. If we denote $n^{\pm}=\#\partial^\pm\Mo\in \N^*$, a connected directed surface is characterized by a triple $(g,n^+,n^-)$. In the case of directed surfaces, there are two surfaces of first importance. They are the two pairs of pants $P_+$ and $P_-$, which are respectively of type $(0,2,1)$ and $(0,1,2)$.

\begin{rem}[Orientation of the boundary components]
\label{rem_orientation_boundaries}
    A direction $\Mo$ on $M$ defines an orientation of the boundary components of $M$, as $M$ is oriented the boundary components are naturally oriented. We choose to orient the positive boundaries according to the orientation induced by $M$; and the negative boundaries, with the opposite orientation.
\end{rem}

If $\Mo$ is connected, we define $\Lambda_{\Mo}$ by
\begin{equation}
\label{formula_LambdaMo}
   \Lambda_{\Mo}=T_{\Mo}\cap \Rp^{\partial M},~~~\text{and}~~~T_{\Mo}=\left\{L=(l_\beta)\in \R^{\partial M}~|~{\sum}_\beta \epsilon(\beta)l_\beta=0\right \}.
\end{equation}
In general, if $\Mo=\sqcup_c~ \Mo(c)$ we set $\Lambda_{\Mo}=\prod_c \Lambda_{\Mo(c)}$ and define $T_{\Mo}$ similarly. 
\begin{figure}
    \centering
    \includegraphics[width=0.4\linewidth]{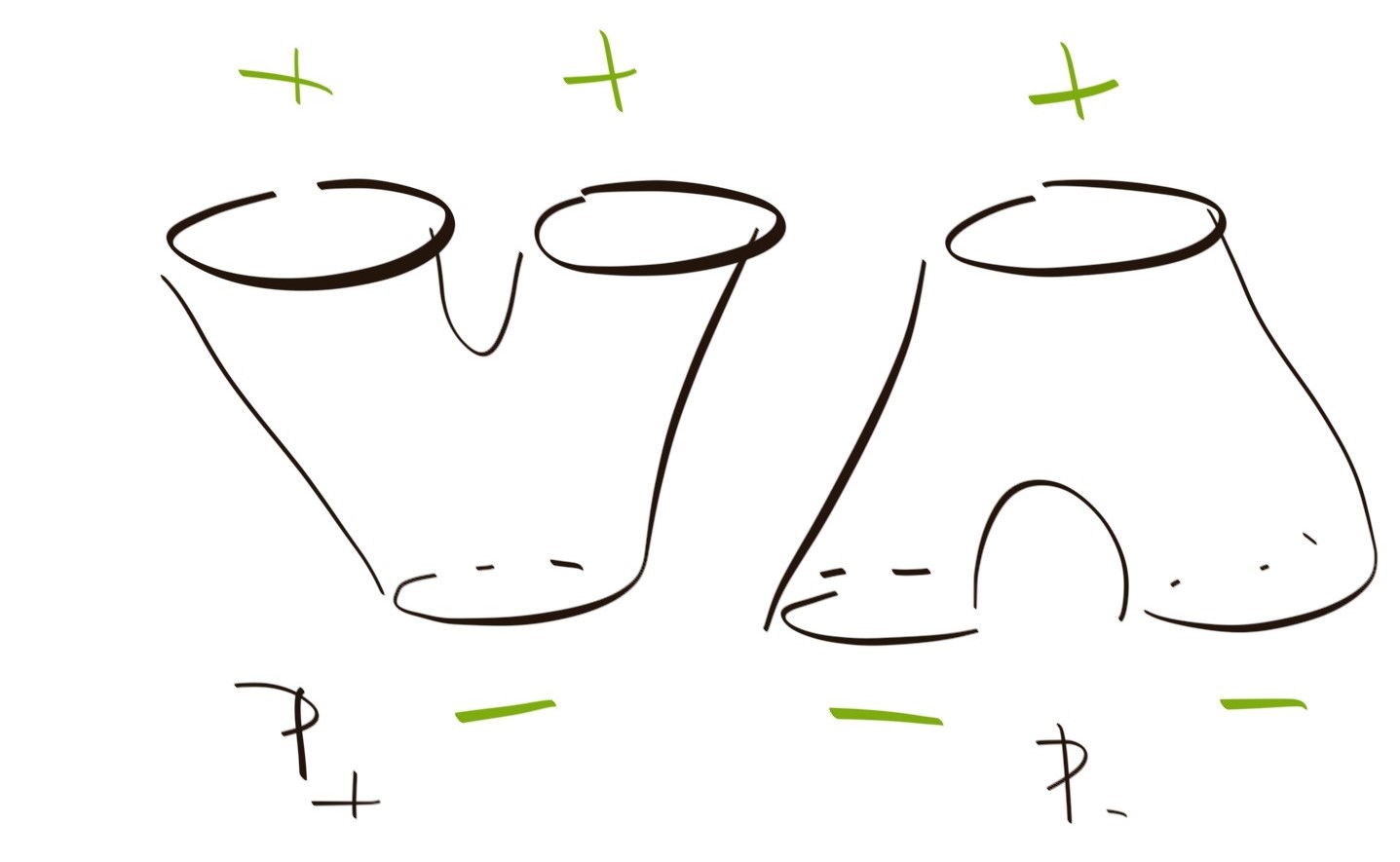}
    \caption{Directed pants of type $(0,2,1)$ and $(0,1,2)$.}
    \label{figure_directed_pants}
\end{figure}
\begin{rem}[marked points and labelling:]
\label{rem_marked_points_label}
\begin{itemize}
    \item It is sometimes convenient to add marked points to surfaces; in this case, we assume that an element of $\Mod(M)$ can permute marked points. If $M\in \bord$, we can consider the surface $M^{\bullet}$ obtained by pinching the boundary components to marked points, for $\beta\in \partial M$ we denote $\beta^{\bullet}$ the corresponding marked point in $M^{\bullet}$. We also denote $M^{cap}$  the surface obtained by gluing a disc on each boundary component of $M$. We use the upper-script $^\bullet$ for categories of surfaces with marked points; a connected surface is characterized by the triple $(g,n,m)$ where $m$ is the number of marked points; we denote $d(M)=d_{g,n,m}=2g-2+n+m$ and extend it by linearity to non-connected surfaces.
   \item A labeled surface is a surface $M\in \bord$ with a label $\lf$, which is a bijection $\lf: \partial M\to \llbracket 1 , n\rrbracket$. If the surface is directed, we choose two labels $\lf^+,\lf^-$ for the positive and negative boundary components; in this case, if the surface is connected of type $(g,n^+,n^-)$, $\Lambda_{\Mo}$ is identified  with $\Lambda_{n^{+},n^{-}}$ given in Appendix \ref{paragraph_notation_indices}.
\end{itemize}
\end{rem}

 We introduce a notation useful to define stratifications of moduli spaces. We consider "partition" $\nu$ which is a function 
 \begin{equation*}
     \nu: \frac{1}{2}\N_{>0} \to \N
 \end{equation*}
 with finite support, we use notations:
 \begin{equation*}
     d(\nu)=\sum_i i~\nu(i)~~~\text{and}~~~n(\nu)=\sum_i\nu(i).
 \end{equation*}
 There is a natural monoid structure on the set of partitions given by $(\nu_1+\nu_2)(i)=\nu_1(i)+\nu_2(i)$ and $d,n$ are both additive for this structure. A decorated surface $\Md$ is a pair $(M,\nu)$ where $M\in \bord$ and $\nu=(\nu_c)_{c\in\pi_0(M)}$ is the data of a partition for each connected component of $M$. Moreover, we impose the following constraints:
 \begin{equation*}
d(\nu_c)=d(M(c)),~~~\forall~c\in \pi_0(M).
 \end{equation*}
If $\Mo=(M,\epsilon)$ is directed, then $\Mdo=(\Mo,\nu)$ is a decorated surface if we assume that $\nu$ is supported by $\N_{>0}$. We denote $\bordd,\borddo...$ the different categories of decorated surfaces.
 \begin{rem}
 \label{rem_decoration_collapsing_et_marked_point}
 \begin{itemize}
     \item In the case of marked surfaces, for each marked point $x$ we assume that we have the additional data of an half integer $\kappa_x$. In general, we assume that $\kappa_x\ge -\frac{1}{2}$ (resp $\kappa_x\ge 0$ in the directed case), we have the constraint $d(\nu_c)+\sum_{x\in M(c)}k_x=d(M(c))$. As before, we denote $\borddb,\borddob,...$ the different categories of decorated surfaces. 
     \item We can collapse two blocs of a partition by the transformation $\nu'=\nu-(i)-(j)+(i+j)$ (if $\nu(i),\nu(j)\ge 1$), $(a)$ is the partition with a single block of size $a$, we have $d(\nu)=d(\nu')$ and $n(\nu)=n(\nu')+1$. We denote $\nu_1<\nu_2$ if it is possible to obtain $\nu_1$ by a succession of collapsing on $\nu_2$. This relation defines a partial order relation on the space of partitions.
 \end{itemize}
 \end{rem}
 
\paragraph{Curves and multi-curves:}

\label{paragraph_multi_curves}
\begin{figure}
    \centering
    \includegraphics[width=0.3\linewidth]{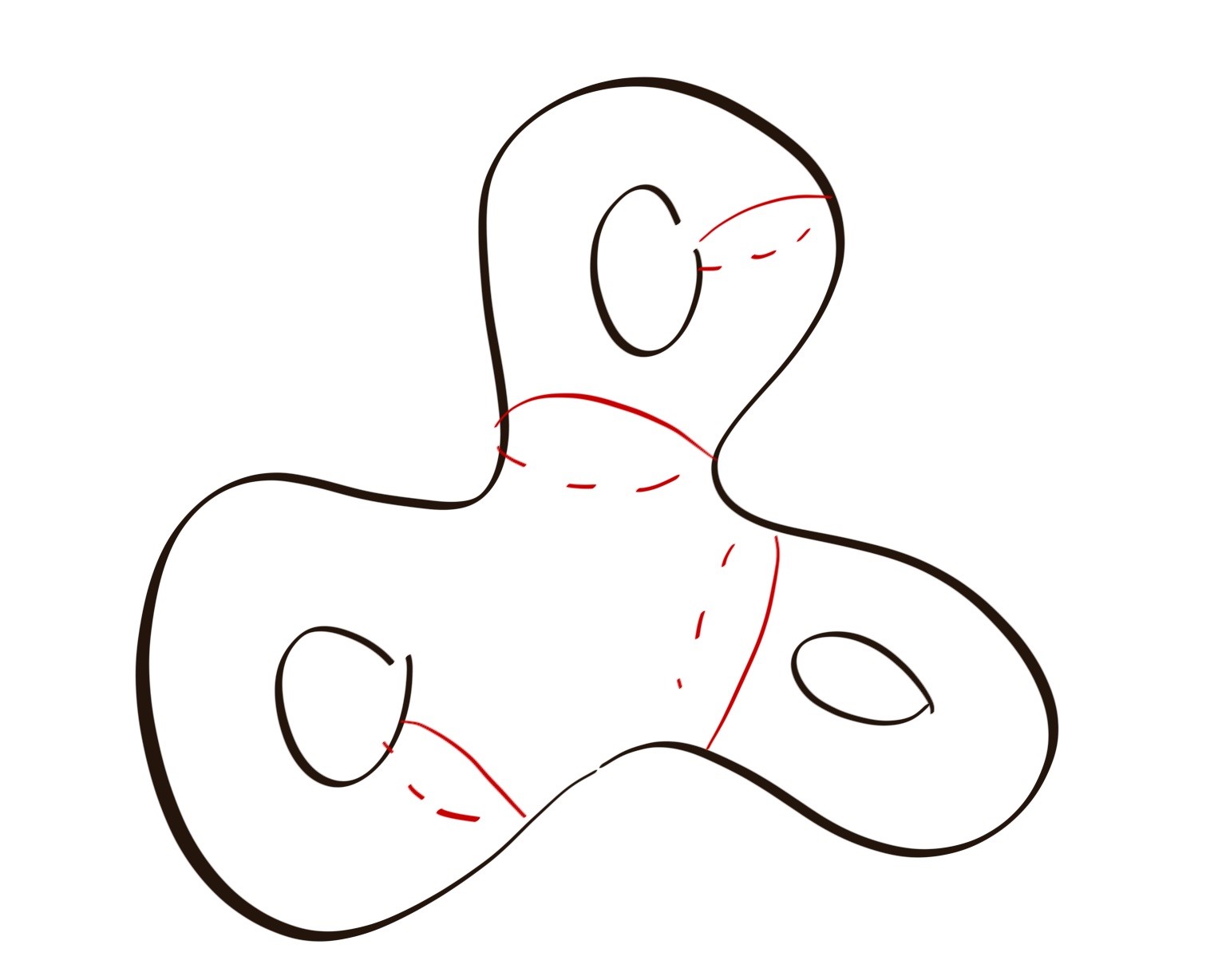}
    \caption{A multi-curve.}
    \label{figure_multi_curves}
\end{figure}

Let $M\in \bord$, an essential simple curve, is an embedded closed curve, that is non-contractible, non-oriented and it is not homotopic to a point or a boundary component of $M$. When $M$ has marked points, we assume that the curve does not pass through these points and does not retract to a marked point. We denote $\Si(M)$ the set of isotopy classes of essential simple curves \cite{fathi2012thurston}, isotopies are taken relatively to the marked points if any. If $M$ has boundary, we also denote $\Sit(M)=\Si(M)\sqcup \partial M$, the union of $\Si(M)$ and the set of boundary curves.\\

A multi-curve $\Gamma \in \MSz(M)$ is a closed family of disjoint, essential curves (see figure \ref{figure_multi_curves}). A multi-curve $\Gamma$ can be represented formally by a sum $\sum_\gamma m_\gamma \gamma$ over $\Si(M)$. Weights $(m_\gamma)$ are positive integers, and two curves in the support are non-intersecting; this last condition implies that the sum is finite and the number of curves in a multi-curve is bounded by $3g-3+n+m$ for a surface of type $(g,n,m)$. If the weights are in $\{0,1\}$, we say that the multi-curve is primitive; we denote $\MS(M)$ the subset of primitive multi-curves. In the same way, we can consider sets $\MSt(M)$ and $\MSt_{\Z}(M)$.\\

If $\Gamma\in \MS(M)$ is a primitive multi-curve on $M$, it is always possible to cut $M$ along $\Gamma$. We obtain a stable topological surface $M_\Gamma$, the procedure depends on the choice of the representative, but all the possible surfaces obtained by surgeries are canonically identified up to isotopies. Surgeries along multi-curves are encoded by stable graphs; we give a possible definition:

\begin{Def}
\label{def_stable_graphs}
A stable graph is given by $\G=(X_0\G,(\G(c))_{c\in X_0\G},s_1)$ where:
\begin{itemize}
\item $X_0\G$ is the set of vertices.
\item $(\G(c))_{c\in X_0\G}$ is a family of topologically stable surfaces with boundary.
\item If $X\G= \sqcup_c \pi_0(\partial \G(c))$ is the set of boundary components of the surfaces, then $s_1: X\G\longrightarrow X\G$ is an involution.
\end{itemize}
\end{Def}
If $X\G(c)=\partial \G(c)$, then a stable graph defines a graph $(X_0\G,X\G, (\partial \G(c))_{c\in X_0\G},s_1)$ (see Appendix \ref{paragraph_graph}). But it also contains information about the topology of the ''vertices''; it can be encoded by a genus map $g: X_0\G\rightarrow \N$ with the stability condition for each vertex of the graph. We say that two graphs are isomorphic if there is a bijection $\phi : X\G \rightarrow X\G$, which is an isomorphism of graphs and preserves the genus map. The involution $s_1$ divides $X_1\G$ into two subsets: $X^{int}_1\G$ is the set of elements of order two; $\partial\G$ is the set of fixed points. We denote by $\Aut(\G)$ the automorphism group of the stable graph, as before we assume that all the elements of $\Aut(\G)$ act trivially on the set $\partial\G$.\\

A stable graph $\G$ encodes how to glue a family of stable surfaces together. $\G$ defines a surface $M_\G$ by gluing the pairs of boundary components in $X^{int}\G$ that are in the same orbit of $s_1$, these boundaries define a multi-curve $\Gamma_\G$ on $M_\G$. Conversely, a multi-curve defines a stable graph, and two multi-curves are in the same orbit under the action of $\Mod(M)$ iff their stable graphs are isomorphic. Then, if $\st(M)$ is the set of isomorphism classes of stable graphs with an homeomorphism $M\simeq M_\G$, we have $\st(M)= \MS(M)/\Mod(M)$. Two multi-curves are in the same orbit iff they have the same topology, which is exactly the information contained in a stable graph.\\

If $\Gamma\in \MS(M)$, we can consider $\Stab(\Gamma)\subset \Mod(M)$, the stabilizer of $\Gamma$; for each $\gamma\in \Si(M)$, we denote $\delta_\gamma\in \Mod(M)$ the Dehn twist along $\gamma$ (see \cite{farb2011primer}). Let $D_\Gamma=\langle \delta_\gamma,\gamma\in \Gamma\rangle$, it is an Abelian normal subgroup of $\Stab(\Gamma)$. Moreover, we have the following exact sequence of surgery:
\begin{equation*}
\label{formula_exact_sequence_stable_graph}
\{0\}\to \Mod(M_\Gamma)\to \Stab(\Gamma)/D_\Gamma\to \Aut(\G_\Gamma)\to \{0\}.
\end{equation*}

\begin{rem}[Quotient and product]
\label{rem_quotient_product_stab}
\begin{itemize}
    \item As for usual graphs, it is possible to take the quotient of a stable graph. Let $E\subset X^{int}_1\G$, we can define the stable graph $\G_{\langle E\rangle}$ by gluing the boundary components in $\tilde{E}$ that are identified by $s_1$, where $\tilde{E}$ is the preimage of $E$ in $X^{int}\G$. We can also define $\G_E$ the graph obtained after removing the edges in $E$.
    \item Let $\G$ be a stable graph and
\begin{equation*}
T_\G=\{(l_\beta)\in \R^{X\G}~|~l_{\beta}=l_{s_1(\beta)}\}.
\end{equation*}
 For each edge $\gamma\in X_1\G$, we denote:
\begin{equation*}
l_\gamma : \Lambda_\G \longrightarrow \R
\end{equation*}
the projection. For a family of sets $(E(c))_{c\in X_0\G}$, with maps $L_c: E_c \rightarrow \R^{\partial\G(c)}$. We can consider the fiber product ${\prod}_\G E_c$. It corresponds to the elements of the product that equalize the length of two elements of $X\G$ in the same orbit under the involution.:
\begin{equation*}
    {\prod}_\G E_c=\{(x_c)\in \prod_c E_c~|~l_{\beta}(x)=l_{s_1(\beta)}(x)~~\forall \beta \in X\G\}.
\end{equation*}
If $\gamma\in X_1\G$, we have a projection ${\prod}_\G E_c \to T_\G$ and then $l_\gamma: {\prod}_\G E_c \rightarrow \R$.
\end{itemize}
\end{rem}

\paragraph{Directed multi-curves:}
We use in this text directed multi-curves; they appear in the context of directed ribbon graphs. 
\begin{Def}
\label{def_oriented_multicurve}
Let $\Mo=(M,\epsilon)$, a directed multi-curve is a pair $\Gao = (\Gamma,\epsilon_\Gamma)$ such that:
\begin{enumerate}
\item $\Gamma$ is a primitive multi-curve on $M$.
\item $\epsilon_\Gamma$ is a map, $\epsilon_\Gamma : \partial M_\Gamma \rightarrow \{\pm 1\}$ such that:
\begin{equation*}
\epsilon_\Gamma(\beta)=\epsilon(\beta)~~~~\forall \beta\in \partial M~~\text{and}~~\epsilon_\Gamma(s_1(\beta))=-\epsilon_\Gamma(\beta)~~~~\forall \beta\in X^{int}\Gamma.
\end{equation*}
\item For each component $M_\Gamma(c)$, let $\epsilon_{\Gamma,c}$ be the restriction of $\epsilon_{\Gamma}$ to $\partial M_{\Gamma}(c)$, then we assume that $\Mo_{\Gao}(c)=(M_\Gamma(c),\epsilon_{\Gamma,c})$ is a directed surface.
\end{enumerate}
A directed multi-curve is non-degenerate iff
\begin{enumerate}[resume]
\item For each sub-multi-curve $\Gamma'\subset \Gamma$, if $\epsilon_{\Gamma'}$ is the restriction of $\epsilon_{\Gamma}$ to $\partial M_{\Gamma'}$, then $(\Gamma',\epsilon_{\Gamma'})$ is a directed multi-curve.
\end{enumerate}
We denote $\MS(\Mo)$ the set of non-degenerate directed multi-curves on $\Mo$.
\end{Def}
 
\begin{figure}
    \centering
    \includegraphics[width=0.5\linewidth]{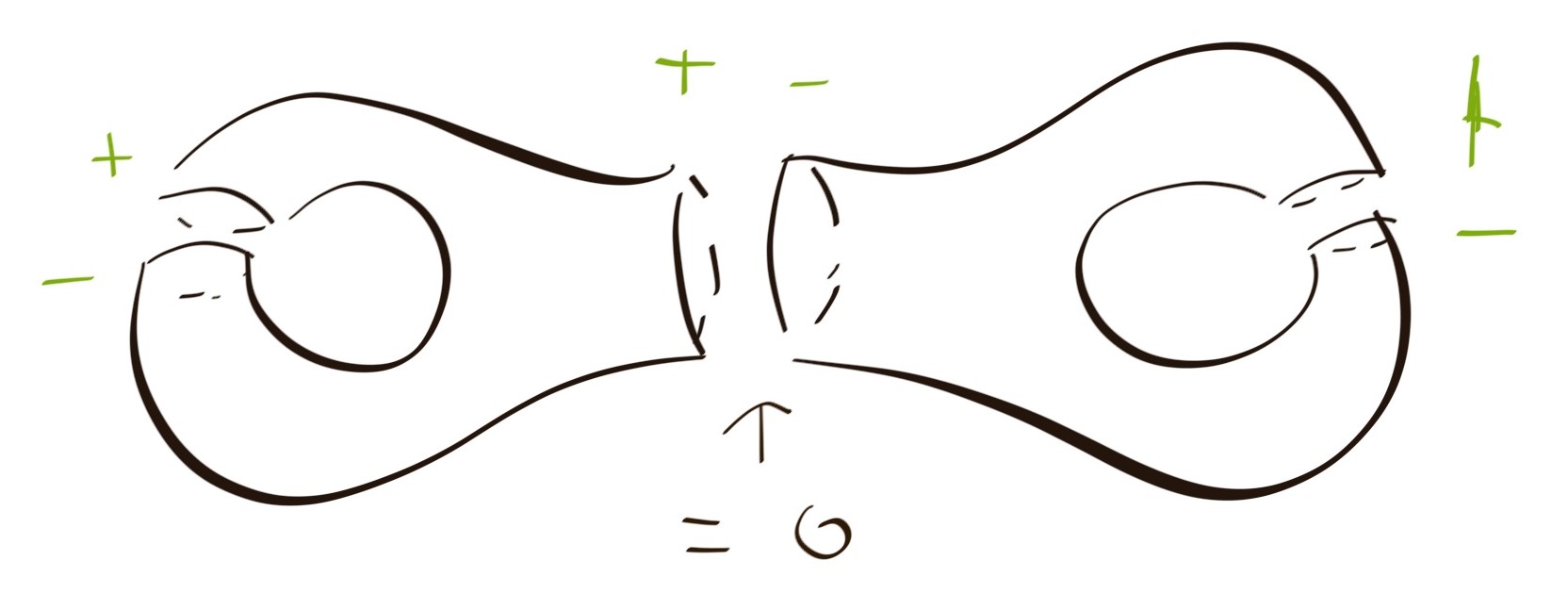}
    \caption{A degenerate directed multi-curve}
    \label{figure_degenerate_graph}
\end{figure}
\begin{rem}[Degenerations]
The point $(4)$ in the definition is more technical but necessary; otherwise, it can happen that the multi-curve contains a sub-multi-curve $\Gamma'$ such that the direction is constant on some connected components of $M_{\Gamma'}$ and such multi-curves are not realizable in practice (see figure \ref{figure_degenerate_graph}).
\end{rem}
\begin{rem}[Orientation of the curves]
\label{rem_orient_curves}
    As we see in remark \ref{rem_orientation_boundaries}, a direction on $\Mo$ defines an orientation of the boundary components. Let $\Gao$ a directed multi-curve on $\Mo$, the direction on $\Mo_{\Gao}$ induces an orientation of the boundary components  $\partial M_\Gamma$. Using the point $(2)$ in the Definition \ref{def_oriented_multicurve}, we see that two boundary components glued together along a curve in $\Gamma$ have the same orientation; then a direction $\Gao$ on $\Gamma$ induces an orientation of the curves in $\Gamma$. 
\end{rem}

To a directed multi-curve, it is natural to associate a directed stable graph, we give a natural definition.

\begin{Def}
\label{def_directed_stable_graph}
A directed, stable graph $\Go$ is a pair $(\G,\epsilon_\G)$ such that:
\begin{enumerate}
\item $\G$ is a stable graph, and $\epsilon : X\G \rightarrow \{\pm 1\}$ is a map.
\item For each $c\in X_0\G$, the restriction $\epsilon_{\G,c}$ of $\epsilon$ to $X\G(c)$ defines a directed surface $\Go(c)$.
\item The involution reverses the direction of the elements of $X^{int}\G$.
\end{enumerate}
A directed stable graph is non-degenerate iff:
\begin{enumerate}[resume]
\item For all $E\in X^{int}_1\G$, the direction restricted to $\G_{\langle E\rangle}$ defines a directed stable graph (see remark \ref{rem_quotient_product_stab} for notation).
\end{enumerate}
\end{Def}

In other words, all the components are directed surfaces and two boundary components glued along a curve have opposite signs. An important point is that the edges and the half edges of a directed, stable graph are directed. We assume that the edges are directed from the $+$ to the $-$ (see figure \ref{figure_acyclic_pant_decomp}). Then a directed stable graph defines a directed graph by forgetting the topology of the components and an usual stable graph by forgetting the direction.\\

As in the case of stable graphs, a non-degenerate directed stable graph defines a directed surface $\Mo_{\Go}$ and a directed, primitive, non degenerate multi-curve $\Gao_{\Go}$ on $\Mo_{\Go}$. If $\Mo$ is a fixed directed surface, we denote $\st(\Mo)$ the subset of non-degenerate directed stable graphs up to homeomorphisms. Then, as in the usual case, we have the identification:
\begin{equation*}
\st(\Mo)\simeq  \MS(\Mo)/\Mod(M).
\end{equation*}
The group of automorphism's $\Aut(\Go)$ of a directed stable graph $\Go$ is then the subgroup of $\Aut(\G)$ that preserves the direction.

\paragraph{Decoration:}
We can define decorated multi-curve. Let $\Md=(M,\nu)$ a decorated surface; a decorated multi-curve $\overline{\Gamma}$ is a multi-curve $\Gamma$ with a decoration $\nu_{\Gamma,c}$ for each connected component of $M_\Gamma$, with the constraint $\sum_c \nu_{\Gamma,c}=\nu$. We can generalize this in a straightforward way for directed surfaces. Similarly, we can also consider decorated, stable graphs. We remark that a decoration $\nu$ on $M$ imposes restrictions on the topology of the possible decorated multi-curves $\Gamma$, for instance, we must have $\#\pi_0(M_\Gamma)\le n(\nu)$.

\paragraph{Multi-arcs:}
\label{paragraph_multiarc}
Let $M\in \bordb$ a surface with boundary, we call ``essential arc'', isotopy classes of unoriented simple arcs, with extremities in $\partial M$. Such arcs are assumed to be non-trivial, in the sense that they do not retract to a portion of a boundary. A family of simple arcs that are disjoint and pairwise non-isotopic defines a multi-arc (see figure \ref{figure_multi_arc_torus}).
Similarly to weighted multi-curves, a weighted multi-arc is defined as a formal sum $\sum_a m_a a$. The arcs in the sum are pairwise distinct and non-intersecting. This implies that the number of arcs is bounded by $6g-6+3n+2m$ whether the surface is of type $(g,n,m)$. We denote  $\A(M)$ the set of arcs, $\MA(M)$ the set of multi-arcs and $\MA_\R(M)$ the space of weighted multi-arcs.\\

In the case of a directed surface $\Mo$, we impose an additional constraint. A multi-arc $A\in \MA(M)$ is directed on $\Mo$ iff each arc connects a positive and a negative boundary of $\Mo$. We denote $\A(\Mo),\MA(\Mo),...$ the subsets of directed multi-arcs.

\begin{figure}
\centering
\includegraphics[height=3cm]{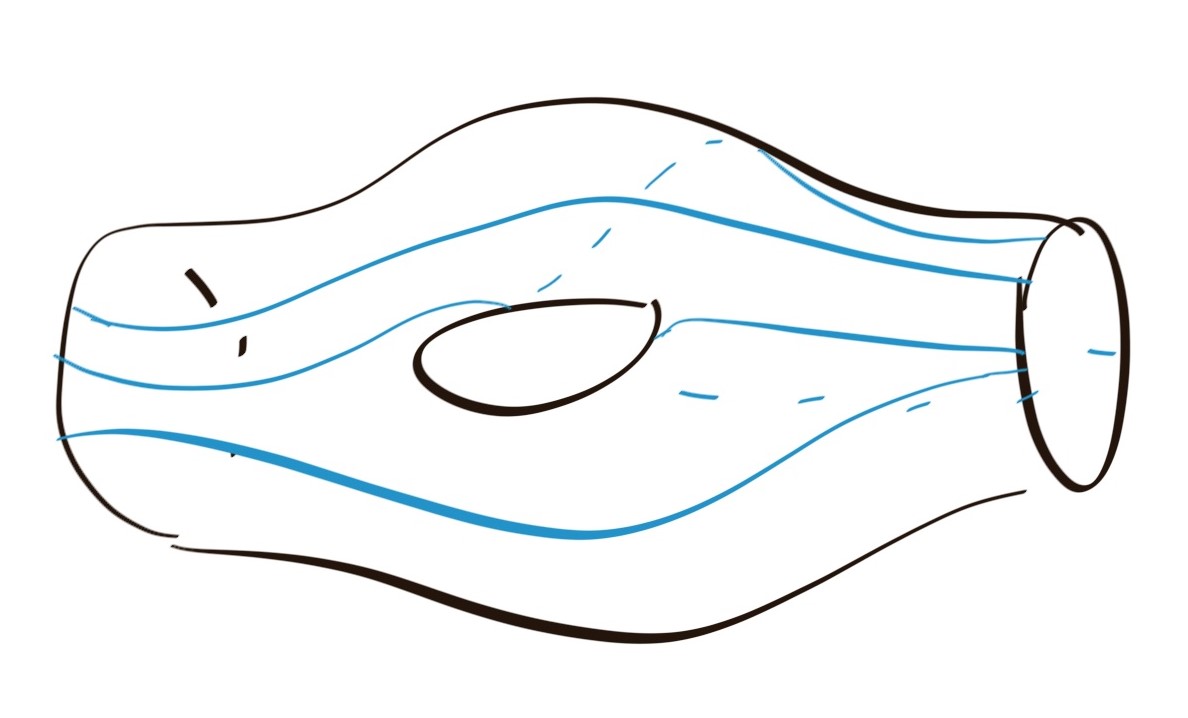}
   \caption{Multi-arc on a torus with one boundary.}
\label{figure_multi_arc_torus}
\end{figure}

\subsection{Foliations, quadratic differentials and Abelian differentials:}
\label{subsection_foliation}
\paragraph{Intersection pairing:}
Let $M$ be a stable surface, W.Thurston introduces the geometric intersection pairing $\iota(\gamma_1,\gamma_2)$ between two isotopy classes of essential curves $\gamma_1,\gamma_2$ (see \cite{fathi2012thurston} or \cite{farb2011primer}). Thurston considers the infinimum of the number of intersection points over the set of all representatives of the two curves. The definition is extended to multi-curves by linearity. If $M$ has a non-empty boundary. We denote $I(\Gamma)=(\iota(\Gamma,\gamma))_{\gamma\in \Si(M)}\in \Rp^{\Si(M)}$ a fundamental result is the following 
\begin{prop}
\label{prop_injectivity_iota}
    The map $I: \MS_\R(M)\to \Rp^{\Si(M)}$ is injective.
\end{prop}

Let $\Gamma_1,\Gamma_2$ two multi-curves, and let $\Gamma_1',\Gamma_2'$ two representatives. We say that $\Gamma_1',\Gamma_2'$ are in minimimal position if they minimize their intersection in the isotopy classes, i.e., $\iota(\Gamma_1',\Gamma_2')=\# (\Gamma_1'\cap \Gamma_2')$. The bigon criterion is a practical way to check if two representatives are in minimal position. Let $M_{\Gamma_1'\cup \Gamma_2'}$ be the closure of the surface obtained by cutting $M$ along $\Gamma_1'\sqcup \Gamma_2'$. The set $\Gamma_1\cap \Gamma_2$ defines marked points on the boundary of $M_{\Gamma_1\cup \Gamma_2}$. We call $n-$gone a topological disc with $n$ marked points on the boundary. We can see that, under our assumptions, each boundary component must contain an even number of marked points. 

\begin{lem}[Bigon criterion 1]
\label{lem_bigon_criterium_1}
$\Gamma_1',\Gamma_2'$ are in minimal position iff $M_{\Gamma_1'\cup \Gamma_2'}$ does not contain components that are $2-$gones.
\end{lem}

The intersection pairing between a multi-arcs and an element of $\widetilde{S}(M)$ is also well defined. It provides a map
\begin{equation*}
   I:  \MA(M) \longrightarrow \Rp^{\tilde{S}(M)}\backslash\{0\}.
\end{equation*}
Moreover, a combinatorial argument allow to prove the following result.
\begin{prop}
\label{prop_injectivite_iota_arc}
The map $I: \MA_\R(M)\to \Rp^{\Sit(M)}$ is injective and continuous\footnote{As we see later, the space $\MA_\R(M)$ is a cell complex and can be endowed with the topology given by this structure.}.
\end{prop}

\begin{rem}[Bigon criterion 2]
\label{rem_bigon_criterion_2}
    As in the case of two multi-curves. By a doubling argument, we can see that: a multi-arc and a multi-curve are in minimal position iff $M_{A\cup \Gamma}$ does not contain $2-$gones.
\end{rem}

\paragraph{Measured foliations:}
Let $M$ a surface with an empty boundary; the intersection pairing allows to embed $\Si(M)\times\Rpp$ to $\Rp^{\Si(M)}\backslash \{0\}$. The closure of the image of $\Si(M)\times\Rpp$ for the product topology is identified with the space of measured foliations (see \cite{masur1982interval} or \cite{fathi2012thurston}). A measured foliation $\lambda$ on $M$ is a foliation with a transverse measure. It can be defined by an atlas of charts $\{(U_i,\phi_i),i\}$ of $M\backslash Z$, where $Z$ is a finite set of singularities. We assume that the transition functions are of the form
\begin{equation}
\label{formula_change_coord_fol}
\phi_i\circ\phi_j^{-1}(x,y)=(f_{i,j}(x,y),\pm y + c_{i,j}).
\end{equation}
These functions preserve the foliation by horizontal lines. The transverse measure of an arc $\alpha~:~[0,1]\longrightarrow M\backslash \Sigma$ is defined locally by the absolute variation of the function $y_i\circ \alpha$ in coordinates. Alternatively, a measured foliation can be defined locally by an exact one form $dy_i$ on each chart; at the overlap of two charts, we have $dy_i=\pm dy_j$. The structure of the singularity at the neighborhood of a point $x\in Z$ is given by the multi-valued one form:
\begin{equation*}
\text{Im} (z^{k}dz)~~~\text{with}~~~k\in \frac{1}{2}\N^*.
\end{equation*}
Where $k$ is the order of the singularity. A saddle connection is a leaf of a foliation that connects two singularities. Two foliations are equivalent if they are related by isotopies and Whitehead moves; these moves are obtained by collapsing a saddle connection that connects two different singularities (see \cite{fathi2012thurston}). We denote $\MF(M)$ the set of equivalence classes of measured foliations on $M$. The intersection pairing between an essential curve in $\S(M)$ and an equivalence class of measured foliation in $\MF(M)$ is well defined; Thurston uses it to characterize the isotopy classes of foliations up to Whitehead equivalences. He considers the map
\begin{eqnarray*}
 I : \MF(M) &\longrightarrow& ~~(\Rp)^{\Si(M)}\\
         \lambda~~~~ &\longrightarrow& (\iota(\lambda,\gamma))_{\gamma\in \Si(M)}
\end{eqnarray*}
and proves that it is injective, two isotopy classes of foliations are Whitehead equivalent iff they have the same intersection pairing.
\begin{rem}[Marked points and foliations]
\label{rem_puncture_foliations}
When the surface has marked points, we allow marked singularities of order $k\ge-\frac{1}{2}$ at these points. In this case, we assume that isotopies are relative to the marked points, and Whitehead moves between two marked singularities are not allowed. The results of Thurston remain true under these assumptions.
\end{rem}

\paragraph{Quadratic differentials:}
\label{paragraph_stratum_ab_quad}
If $M$ is a compact topological surface with empty boundary, a structure of Riemann surface $X$ is an atlas of charts with values in $\C$ and that such transition functions are holomorphic maps. A quadratic differential $q$ on $X$ is then a global section of the square of the cotangent bundle of $X$. In other words, $q$ is given by
\begin{equation*}
q=f(z)(dz)^2,
\end{equation*}
in a local coordinate $z$, where $f$ is an holomorphic function. A quadratic differential might have zeros or poles; let $k_x\in \N$ be the order of the zeros at $x$ we must have
\begin{equation*}
\sum_x k_x= d(M).
\end{equation*}
A quadratic differential also defines a structure of a half-translation surface. Away from the zeros, it is possible to find coordinates in which the differential is given by $(dz)^2$. The change of flat coordinates is of the form $z\rightarrow \pm z + c$, where $c\in \C$ is a constant; using this, a quadratic differential $q$ defines two foliations: the vertical and horizontal foliations. They are given by
\begin{equation*}
    \lambda_v(q)=\text{Re} \sqrt{q}~~~\text{and}~~~\lambda_h(q)=\text{Im} \sqrt{q}.
\end{equation*}
Where $\sqrt{q}$ is a local square root of $q$, which is well defined up to a sign and is closed. We denote $\QT(M)$ the Teichmüller space of quadratic differentials; it is the space of triples $(\phi,X,q)$
where:
\begin{itemize}
\item $X$ is a Riemann surface,
\item $\phi: M \rightarrow X$ is a homeomorphism,
\item $q$ is a holomorphic quadratic differential on $X$.
\end{itemize}
Two triples are equivalent if there is $h: X'\rightarrow X$ biholomorphic such that $h^*q=q'$ and $\phi^{-1} \circ h\circ \phi' $ is isotopic to the identity.
\begin{rem}[Marked points]
When $M$ has marked points, we consider quadratic differentials that are holomorphic outside the set of marked points and with $k_x\ge -1$ for all marked points.
\end{rem}
\paragraph{Pair of transverse foliations:}
\label{paragraph_hubbard_masur}
Two measurable foliations $\lambda_1,\lambda_2$ on a surface $M$ without boundaries are transverse iff for all essential simple curves $\gamma\in \Si(M)$ they satisfy
\begin{equation*}
\label{formula_transverse}
\iota(\lambda_1,\gamma)+\iota(\lambda_2,\gamma)>0.
\end{equation*}
A quadratic differential $q$ defines a pair of foliations $(\lambda_v(q),\lambda_h(q))$, and it is possible to see that these two foliations are necessarily transverse \cite{hubbard1979quadratic}. In fact, the converse statement is true and is part of the Hubard-Masur Theorem.

\begin{thm}[Hubbard-Masur \cite{hubbard1979quadratic}]
\label{thm_hubbard_masur}
Let $M$ a compact, stable surface with an empty boundary. The map $q\to (\lambda_v(q),\lambda_h(q))$ induces a homeomorphism:
\begin{equation*}
\QT(M)\to  \MF(M)^2\backslash \Delta.
\end{equation*}
Where $\Delta$ is the set of pairs of non-transverse foliations.
\end{thm}
\paragraph{Jenkin-Strebel differentials:}
\label{paragraph_jenkin_strebel}
A leaf of a foliation is periodic if it is closed and does not contain singularities. Such a leaf cannot be isolated and is contained in a maximal annulus domain foliated by parallel periodic leaves called cylinder. A cylinder $C$ admits a unique isotopy class of essential curves $\gamma_C$, which is its core curve; moreover, the cylinder has a well-defined height, which we denote as $h_C$. It is possible to see that the core curves are necessarily essential curves; moreover, if two cylinders are distinct, then the core curves are non-isotopic. The multi-curve associated with $\lambda$ is
\begin{equation*}
    \Gamma(\lambda)=\sum_{C}h_C \gamma_C,
\end{equation*}
where we sum over all the cylinders. A foliation is periodic if all the leaves are either periodic or are saddle connections; in this case, the closure of the union of the cylinders covers the surface. In this case, for all $\gamma\in \Si(M)$, we have the equality $I(\Gamma(\lambda))=I(\lambda)$.
Conversely, for each multi-curve $\Gamma\in \MS_\R(M)$, it is possible to associate a periodic foliation $\lambda(\Gamma)$ that is unique up to Whitehead moves. Then the intersection pairing allows us to identify periodic foliations with the space of weighted multi-curves $\MS_\R(M)$.\\

In this text, a Jenkin-Strebel differential is a quadratic differential with a periodic horizontal foliation (see \cite{strebel1984quadratic}). Then, by Theorem \ref{thm_hubbard_masur}, the Hubbard-Masur map identifies the subspace of Jenkin-Strebel differentials with
\begin{equation*}
(\MF(M)\times \MS_\R(M))\backslash\Delta.
\end{equation*}
If $\Gamma\in \MS(M)$ is a primitive multi-curve, we denote $\MF_\Gamma(M)$ the set of foliations transverse to $\Gamma$. The Hubbard-Masur map defines a map:
\begin{equation*}
q_\Gamma: \MF_\Gamma(M)\longrightarrow \QT(M).
\end{equation*}
The quadratic differential $q_\Gamma(\lambda)$ is Jenkin-Strebel and satisfies
\begin{equation*}
\lambda_h(q_\Gamma(\lambda))= \lambda(\Gamma),~~~\text{and}~~~\lambda_v(q_\Gamma(\lambda))= \lambda.
\end{equation*}
Then image of $q_\Gamma$ is the space
\begin{equation*}
\QT_\Gamma(M)= \{q\in \QT_0(M)|\lambda_h(q)= \lambda(\Gamma)\}.
\end{equation*}

\paragraph{Foliations with boundaries and double poles:}
\label{paragraph_fol_double}
Let $M\in \bord$; we need to consider measured foliations on $M$. A possible way is to take foliations on $M$ such that, for all boundary $\beta$, all the leaves of the foliation that cross $\beta$ are either transverse to $\beta$, or $\beta$ is a non-singular periodic leaf of the foliation. As before, we consider measured foliation up to isotopies and Whitehead moves; we denote $\MFt(M)$ the corresponding space.\\

Nevertheless, this is not completely satisfactory; it is somewhat better to consider foliations with poles rather than foliations with boundaries. We introduce the space $\MF(M)$ of measured foliations on the punctured surface $M^\bullet$, with singularities of order $-1$ (double poles) at  punctures that correspond to boundary components of $M$. It means that each $\beta^\bullet$ admits a neighborhood $U_\beta$ that is isomorphic to a disc with a local chart:
\begin{equation*}
U_\beta \longrightarrow \R/l_{\beta}\Z \times \Rpp.
\end{equation*}
The leaves of the foliation are either the vertical or the horizontal lines in the half-infinite cylinder $\R/l_{\beta}\Z\times\Rpp $. When the leaves are horizontal, the value of $l_\beta$ is not relevant, and we take $l_\beta=1$. Otherwise, the number $l_\beta$ corresponds to the absolute value of the residue at the pole $\beta^\bullet$. As below, these foliations are considered up to isotopies and Whitehead moves. When the surface $M$ has marked points, we allow  singularities of order $\ge -\frac{1}{2}$ at these points ($k_x=0$ is a regular marked point). It is possible to extend the Thurston intersection pairing to the case of foliations with poles:
\begin{equation*}
\iota~:~\Sit(M)\times \MF(M) \longrightarrow \Rp.
\end{equation*}
The absolute value of the residue of an element in $\MF(M)$ at a pole $\beta^\bullet$ defines a map $l_\beta: \MF(M) \longrightarrow \Rp$\footnote{It is the intersection pairing with the circle $\beta$ around this pole (which corresponds to a boundary of $M$)}, and we denote $\Lbord=(l_\beta(\lambda))_{\beta\in \partial M}$.

\begin{rem}
An element of $\MFt(M)$ can be extended uniquely to a foliation on $M^\bullet$, and at the punctures, the foliation has a simple pole. Then, there is a surjective, but not injective, map: $\MFt(M)\longrightarrow \MF(M)$. If we restrict the map to the subset of foliations with non-vanishing residues, i.e, when the leaves are transverse to the boundary components, then the map is injective.
\end{rem}

\paragraph{Multi-arcs and foliations:}
\label{paragraph_multiarc_fol}
As for multi-curves, a weighted multi-arc defines a partial foliation of the surface, it can be extended by the Thurston enlargement procedure \cite{fathi2012thurston}. Then it is possible to construct a unique map that preserves the intersection product:
\begin{equation*}
\MAr(M)\longrightarrow \MF(M).
\end{equation*}
Moreover, this map is continuous. We give this map explicitly in the case of filling multi-arcs in paragraph \ref{paragraph_ribbon_filling_arc}. A fact that is important in the sequel is that the converse map exists. A foliation on a surface with boundary defines, in a natural way, a weighted multi-arc by considering the leaves that connect two poles. This procedure is the exploration of the surface from the boundary components.
\begin{prop}
\label{prop_zippered_fol_general}
There is a continuous map:
\begin{equation*}
 A: \MF(M)\longrightarrow \MAr(M)\cup \{0\}.
\end{equation*}
Which coincides with the identity on $\MAr(M)$.
\end{prop}

\begin{proof}
For the construction of the map, consider for each pole with non-vanishing residue a circle around it. Assume that these circles are pairwise disjoint, and each of them bounds a disc that contains a double pole and no other singularity. By using local coordinates around the pole, we can assume that the foliation intersects transversely these circles. Each circle defines a contractible neighborhood $U_\beta$ of the pole $\beta^\bullet$ and if a leaf enters such a neighborhood, it cannot escape. Then the intersection of the singular leaves of the foliation and the circles $C_\beta$ is a finite set $X_0\lambda$ and each circle $C_\beta$ is divided into a finite number of intervals. We denote $X\lambda$ the set of intervals. If $x\in C_\beta$ is a point in one of these intervals, it is possible to consider the half leaf starting at $x$ in the direction opposite to $U_\beta$. By assumption, this leaf does not hit any singularities. By the Thurston recurrence Lemma \cite{fathi2012thurston} such a leaf must intersect another circle $C_{\beta'}$ at a point $T(x)$. The map $T$ is well defined on the union of the intervals and induces a map:
\begin{equation*}
s_1: X\lambda \longrightarrow X\lambda,
\end{equation*}
such that $T$ maps $I$ to $s_1(I)$. The map $s_1$ is an involution. A leaf of the foliation that joins $I$ and $s_1(I)$ defines an arc $a_I$ on the surface. By using the bigon criterion (remark  \ref{rem_bigon_criterion_2}) it is possible to see that two arcs are necessarily non-homotopic. And then the foliation defines a multi-arc. Moreover, the transverse measure on $\lambda$ induces a measure on each interval $I$, and the total mass gives a weight $m_I(\lambda)$. The map $T$ preserves these measures and then $m_I(\lambda)$ defines weight on the arcs $a_I$, and this construction gives the desired map. The continuity is due to the fact that the weights on the edges can be computed by using intersection pairings with appropriate curves.
\end{proof}
\begin{figure}
\centering
\includegraphics[width=0.5\textwidth]{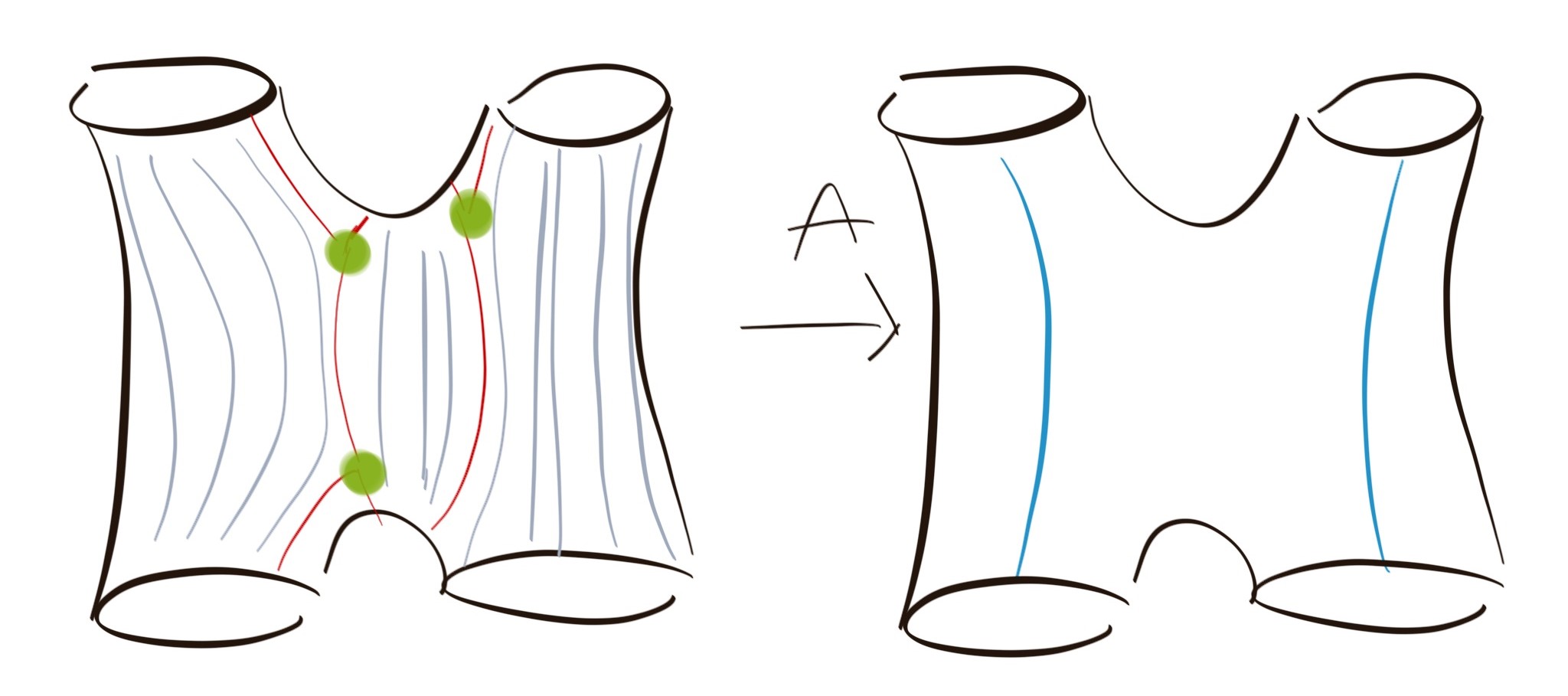}
\caption{Illustration of the map $A$.}
\label{fugure_zippered_1}
\end{figure}

 We can also make the statement of Proposition \ref{prop_zippered_fol_general} more precise and describe the entire foliation. From this, it is possible to derive the following Theorem, which is similar to the results of Thurston.

\begin{thm}
\label{thm_injetivite_iota_fol2}
Let $M$ a surface with boundary; the map:
\begin{equation*}
I: \MF(M)\longrightarrow \Rp^{\Sit(M)},
\end{equation*}
is injective. Moreover, the space $\MF(M)$ is the closure of $\MAr(M)$ for the topology given by the intersection pairing.
\end{thm}

\paragraph{Foliations with vanishing residues:}
\label{paragraph_zeros_residu}
In this paragraph, we consider $\MF_0(M)$ the subset of $\MF(M)$ formed by foliations with vanishing residues. $\MF_0(M)$ contains a trivial element, which is the trivial Jenkin-Strebel foliation; it is periodic, and all the non-singular trajectories retract to a pole. We denote this foliation as $\lambda(0)$; it is unique up to Whitehead moves and isotopies. A nontrivial foliation contains leaves that are not homotopic to boundary curves; it is then possible to pinch the boundary components and obtain a foliation on the punctured surface $M^{\bullet}$, with no double poles and marked conical singularities on each puncture of $M^{\bullet}$.
\begin{prop}
\label{prop_contraction}
The contraction defines a bijection:
\begin{equation*}
C: \MF_0(M)\backslash \{\lambda(0)\}\longrightarrow \MF(M^{\bullet}).
\end{equation*}
Moreover, the bijection is characterized by:
\begin{equation*}
\iota(\lambda,\gamma)=\iota(C(\lambda),\gamma)
\end{equation*}
for all curve $\gamma\in \Si(M^{\bullet})$.
\end{prop}
In a similar way, we can consider the space $\MFt_0(M)$ of foliations in $\MFt(M)$ with vanishing residues. It corresponds to elements of $\MF_0(M)$ marked by a choice of non-singular trajectories around each pole. Then we have an identification
\begin{equation*}
    \MFt_0(M) \simeq \MF_0(M) \times \Rpp^{\partial M}.
\end{equation*}

\paragraph{Quadratic differentials with double poles:}
For a surface $M\in\bord$, we consider quadratic differentials on the punctured surface $M^{\bullet}$ with double poles at the marked points $\beta^\bullet$. As before, we can consider the Teichmüller space $\QT(M)$; this does not depend on the choice of $M^{\bullet}$. It is possible to take the residue around a double pole of $q$ by taking a local square root; the result is an element of $\C/\{\pm 1 \}$, but we can fix a representative by assuming that the real part is positive (or the imaginary part when the real part is zero). We denote $\QT_0(M)$ the subspace of quadratic differentials with real residues at the poles.
 
\paragraph{Pairs of transverse foliations:}
Transversality has a straightforward generalization for foliations with double poles, but we need to replace $\Si(M)$ by $\Sit(M)$ in the characterization, Formula \ref{formula_transverse}. As before, a quadratic differential with double poles defines a pair of transverse foliations with poles. We have a map:
\begin{equation*}
\QT(M)\to \MF(M)^2\backslash \Delta.
\end{equation*}
We do not give the most general statement of the Theorem, but only the part that is useful in this text. A doubling argument allows us to prove the following result by using the Hubbard-Masur Theorem \ref{thm_hubbard_masur}.\\

\begin{thm}
\label{thm_hubbard_masur_poles}
Let $M$ a surface with boundary, and let $\Delta$ be the subspace of pairs of non-transverse foliations in $\MF(M)\times \MF_0(M)$. Then the map:
\begin{equation*}
\QT_0(M)\longrightarrow (\MF(M)\times \MF_0(M))\backslash \Delta,
\end{equation*}
is a bijection.
\end{thm}

\paragraph{Periodic foliations:}
We can generalize the notion of a cylinder in a straightforward way. There is only one difference: a cylinder can have an infinite height. It happens for a foliation that has a pole with vanishing residue. In this case, there is a maximal neighborhood of the poles of the form $\mathbb{S}_1\times \Rpp$ in which $\lambda$ is given by $dy$. We can also generalize the notion of periodic foliations, but we must add the Jenkin-Strebel foliation given by the empty multi-curve $0$. The space of periodic foliations is then identified with
\begin{equation*}
    (\MS_\R(M)\sqcup \{0\})\times \MF(M)\backslash \Delta.
\end{equation*}

\paragraph{Directed foliations and Abelian differentials:}
A measured foliation $\lambda$ is dirigible if it can be globally defined by a closed one form. It is equivalent to the fact that the changes of local charts take the form
\begin{equation*}
    (x,y)\longrightarrow (f(x,y),y+c).
\end{equation*}
When the surface is connected, there are at most two possible directions, and a directed foliation will be denoted by $\lambdao$.
Directed foliations are naturally related to Abelian differentials. An Abelian differential is the data of a structure of Riemann surface $X$, with an holomorphic one form $\alpha$ on $X$ (i.e., a section of $K_X$), $\alpha$ might have zeros; and in this case, we have:
\begin{equation*}
    \sum_x k_x= 2g-2.
\end{equation*}
The square of an Abelian differential is a quadratic differential. According to this, an Abelian differential defines a pair of transverse measured foliations. In this case, outside the set of zeros, we can find local coordinates $z$ in which $\alpha$ is equal to $dz$. Moreover, the transition functions are translations; there is no sign ambiguities. Then the two foliations $\lambda_h(\alpha^{\otimes 2}),\lambda_v(\alpha^{\otimes 2})$ are naturally directed, and we denote them $\lambda^{\circ}_h(\alpha),\lambda^{\circ}_v(\alpha)$. As before, we can consider the Teichmüller $\HT(M)$ of Abelian differentials on $M$; this space is naturally stratified.

\begin{rem}
If $M$ has marked points, we assume that these points correspond to marked regular points or marked zeros of the Abelian differentials.
\end{rem}

For a surface with boundary, the definition is similar. We consider directed foliation $\lambdao$ with poles. In this case, when we compute the residues, there are no sign ambiguities. The residues define real numbers by taking the period of the closed one form. The signs of these numbers define a sign for each boundary component. If the foliation is defined by a closed one form, then, by virtue of the Stokes Theorem, the sum of the residues is zero. If the residues are non-vanishing, there is at least one positive and one negative boundary components. According to this, we see that a directed foliation with non-vanishing residues on a surface $M$ defines a direction on this surface. Then, for each directed surface $\Mo$ we can consider the space $\MF(\Mo)$ of directed foliations that give a direction compatible with $\Mo$. Moreover, by Stockes theorem, the sum of the residues of the positive boundaries is equal to the sum of the residues of the negative boundaries. And then the map $\Lbord$ induces a map:
\begin{equation*}
\Lbord : \MF(\Mo)\longrightarrow \Lambda_{\Mo}.
\end{equation*}
There is an equivalent version of Proposition \ref{prop_zippered_fol_general} in this case.
\begin{prop}
Let $\Mo$ be a directed surface.
\begin{itemize}
    \item The restriction of $A: \MF(M)\longrightarrow \MA_\R(M)$ induces a map:
\begin{equation*}
A: \MF(\Mo)\longrightarrow \MA_\R(\Mo).
\end{equation*}
\item Each element $A\in \MA_\R(\Mo)$ defines a directed foliation by enlarging procedure; there is a unique map:
\begin{equation*}
\MA_{\R}(\Mo)\longrightarrow \MF(\Mo).
\end{equation*}
That preserves the intersection pairings.
\end{itemize}
\end{prop}
\paragraph{Cylinders of a directed foliation:}
If $\Mo$ is a directed surface and $\lambda^{\circ}$ is a directed foliation, it is possible to orient the core curve of a cylinder in a natural way. Indeed, for each cylinder, we can find a map:
\begin{equation*}
\varphi : \R/\Z\times ]0,h[\rightarrow M,
\end{equation*}
such that the foliation is given by $dy$ in these coordinates. We chose to label the upper boundary of the cylinder by $+$ and the lower boundary by $-$. This defines an orientation of the curves in $\Gamma(\lambda)$, then according to remark \ref{rem_orient_curves}, it defines a directed multi-curve $\Gao(\lambdao)$. Then, by the Stockes Theorem, we can prove the following:

\begin{lem}
\begin{itemize}
\item If $\lambdao\in \MF(\Mo)$, the multi-curve $\Gao(\lambdao)\in \MF_\R(\Mo)$ is non-degenerate.
\item For each non-degenerate multi-curve $\Gao\in \MS_\R(\Mo)$, there is a directed foliation $\lambda^\circ=\lambdao(\Gao)\in \MF(\Mo)$ such that
\begin{equation*}
\Gao=\Gao(\lambdao),
\end{equation*}
and this foliation is unique up to Whitehead moves.
\end{itemize}
\end{lem}

\newpage
\section{Directed stable graphs}
\label{section_directed}
\subsection{Cone associated to directed stable graphs}

\paragraph{Cone of relative cycles:}
Let $\Go=(\G,\epsilon)$ be a directed stable graph, we can construct the cone $\Lambda_{\Go}$ of directed cycles on $\Go$; it is the subset of $\Lambda_\G$ defined by:
\begin{equation}
\label{conedirected_graph}
\Lambda_{\Go} = {\prod}_\G \Lambda_{\Go(c)} = \left\{L\in \prod_c \Lambda_{\Go(c)}~|~l_\beta=l_{s_1(\beta)}~~~~\forall~\beta\in X\G\right\}.
\end{equation}
At each node of the graph, the sum of the lengths of the positive boundaries is equal to the sum of the negative boundaries.
For each edge $\gamma\in X_1\G$ we can define the projection:
\begin{equation*}
l_\gamma : \Lambda_{\Go}\longrightarrow \Rp,
\end{equation*}
we call it the length of $\gamma$. If $\Go\in \st(\Mo)$, we can identify $\partial \Go$ and $\partial \Mo$. We denote $\Lbord=(l_\beta)_{\beta\in \partial M}$, we can see that the image of $\Lbord$ is included in $\Lambda_{\Mo}$:
\begin{equation*}
\Lbord : \Lambda_{\Go} \longrightarrow \Lambda_{\Mo}.
\end{equation*}
For all $L\in \Lambda_{\Mo}$, we denote $\Lambda_{\Go}(L)$ the level set $\Lbord^{-1}(\{L\})$.\\

A walk in a graph is a sequence of positive half edges $(\beta_1,...\beta_r)$ with $[s_1\beta_k]_0= [\beta_{k+1}]_0$ for all $k\in \{1,...,r-1\}$ (the target of $\beta_k$ is the source of $\beta_{k+1}$). In a directed graph, we assume that the half edges are positive (except if $\beta_1\in \partial \Go$, in this case, we assume it is negative). An absolute cycle is a closed walk, and a relative cycle is a walk that goes from a negative boundary to a positive boundary. An absolute or relative cycle $c$ defines an element $[c]$ of $\Lambda_{\Go}$, which motivates the terminology, $l_\gamma([c])$ is the number of times that $c$ passes through $\gamma$. We call a primitive cycle, a cycle that can't be written as a union of two distinct cycles.
\begin{prop}
\label{prop_ext_ray_lambda}
Each element of $\Lambda_{\Go}$ is a linear combination of primitive cycles with positive coefficients.
\end{prop}

\begin{proof}
    If $u$ is an element of $\Lambda_{\Go}$ and $\gamma$ is an edge in the support of $u$, i.e., $u_\gamma>0$ . The fact that the sum of the inputs and the outputs are equal at each vertex of the graph implies that we can find a primitive cycle $c$ that passes through this edge with support contained in the support of $u$. This is a consequence of an exploration process in the directed graph. Moreover, by multiplying $c$ by a strictly positive real number, we can assume that $u\ge c$ and there is an edge $\gamma'$ with $u_{\gamma'}=c_{\gamma'}>0$. Then the support of $u-c$ is strictly contained in the one of $u$. By induction on the cardinal of the support of $u$, we obtain the claim.
\end{proof}

\paragraph{Tangent space and cohomology:}
    Let $T_{\Go}$ the tangent space of $\Lambda_{\Go}$, $T_{\Go}$ can be identified with the homology $H_1(\G,\partial \G,\R)$ of the graph relatively to the boundary by using the complex
\begin{equation*}
0\rightarrow \R^{X_1\G}\rightarrow \R^{X_0\G}\rightarrow 0.
\end{equation*}
A directed edge $\gamma$ has a positive $\gamma_+$ and a negative $\gamma_-$ extremity, and we set $\partial \gamma= \gamma_- -\gamma_+$.
For each $L$ in the image of $\Lbord$, the tangent space $K_{\Go}$ of the polytope $\Lambda_{\Go}(L)$ is the kernel of $T\Lbord$. We can see that $K_{\Go}$ corresponds to the homology $H_1(\Go,\R)$. Then the exact sequence:
\begin{equation*}
\label{formula_KGo}
0\longrightarrow K_{\Go} \longrightarrow T_{\Go}\overset{T\Lbord}{\longrightarrow}T_{\Mo} \longrightarrow 0,
\end{equation*}
is the long exact sequence of relative homology.
\paragraph{Behavior of $l_\gamma$ on $\Lambda_{\Go}$:}
\label{paragraph_l_gamma}
Let $\Go$ be a directed stable graph, and $\gamma\in X_1^{int}\Go$ an edge; Proposition \ref{prop_bounded_unbounded} gives information's on the behavior of $l_\gamma$.
\begin{prop}
\label{prop_bounded_unbounded}
According to the topology of the graph, we have the following dichotomy:
\begin{itemize}
\item The length function $l_\gamma$ is unbounded on $\Lambda_{\Go}(L)$ for each $L$ in the image of $\Lbord$;
\item The length satisfies
\begin{equation*}
 l_\gamma \le \sum_{\beta \in \partial^+\Go} l_\beta. \end{equation*}
\end{itemize}
\end{prop}
We call the first edges the unbounded edges and the second the bounded edges. To prove Proposition \ref{prop_bounded_unbounded}, we use Lemma \ref{lem_topo_edges} given below, and the proof also gives the bound.
\begin{lem}
\label{lem_topo_edges}
An edge is bounded iff it is not contained in the support of any absolute cycle.
\end{lem}
\begin{proof}
We use Proposition \ref{prop_ext_ray_lambda}. Let $x$ an element in $\Lambda_{\Go}$, which is a linear combination of primitive relative cycles $x=\sum_{i}x_i \gamma_i$. Each cycle $\gamma_i$ connects a positive and a negative boundary, then we have $\sum_{i}x_i= \sum_{\beta\in \partial^+\Go}l_{\beta}(x)$. Each cycle crosses the edge $\gamma$ at most once because they are primitive, then $l_\gamma(x)\le \sum_i x_i$. We can conclude that the edge is bounded if it is not contained in the support of any absolute cycle and $l_\gamma \le \sum_{\beta\in \partial^+\Go}l_{\beta}$. Conversely, if an edge $\gamma$ is crossed by an absolute cycle $c$ in the graph, we have $l_\gamma(c)>0$. Let $L$ be in the image of $\Lbord$ and $x\in \Lambda_{\Go}(L)$. For each $t>0$, we can see that $x+tc$ defines an element of $\Lambda_{\Go}(L)$, and then the edge is unbounded because $l_\gamma(x+tc)\ge tl_\gamma(c)$.
\end{proof}
\paragraph{Degenerate graphs:}
\label{paragraph_deg_graphs}
As we see, a directed stable graph or a directed multi-curve can be degenerate. We give the following definition:
\begin{itemize}
\item We say that an edge $\gamma\in X_1\Go$ is degenerate if $l_\gamma$ is constant equal to zeros on $\Lambda_{\Go}$.
\end{itemize}
Proposition \ref{prop_degenerate_graph_degenerate_edge}relies on degenerate directed graphs and degenerate edges.
\begin{prop}
\label{prop_degenerate_graph_degenerate_edge}
A directed stable graph is degenerate iff it has degenerate edges.
\end{prop}

\subsection{Acyclic stable graphs}
\paragraph{Definition:}
We use a particular kind of directed stable graphs, which are acyclic stable graphs. A directed stable graph induces a relation on the set of vertices. We say that $x\ge y$ iff there is a path oriented positively from $y$ to $x$ (see figure \ref{figure_acyclic_pant_decomp}).
\begin{Def}
\label{def_acyclic}
A directed graph is acyclic iff it satisfies one of the following equivalent conditions:
\begin{itemize}
\item The graph contains no directed absolute cycle.
\item The relation on the vertices of the graph is a strict partial order.
\end{itemize}
\end{Def}

\begin{proof}
It is straightforward to see that the graph admits a cycle iff the relation is not anti-symmetric.
\end{proof}
In what follows, we need to label the vertices of an acyclic graph, and then we use the following definition:
\begin{Def}
\label{def_linear_order}
A linear order on a directed acyclic graph is an enumeration of the vertices (elements of $X_0\Go$), which is increasing for the order relation.
\end{Def}
We can see that an acyclic stable graph is necessarily non-degenerate. Moreover, an important property of acyclic stable graphs is the following:
\begin{cor}
\label{cor_acyclic_bounded}
A directed stable graph is acyclic iff all the edges are bounded (Proposition
\ref{prop_bounded_unbounded}).
\end{cor}

\begin{proof}
Using Proposition \ref{prop_bounded_unbounded}, if a graph is bounded, an edge can't lie in the support of an absolute cycle, and then the graph is acyclic. Conversely, if the graph is acyclic, there is no absolute cycle, and then all the edges are bounded according to Proposition \ref{prop_bounded_unbounded}.
\end{proof}

Proposition \ref{prop_gluing_acycl} will be useful to prove Lemma \ref{lem_maximal_minimal}. If $\Go$ is a directed stable graph, and for each $c\in X_0\Go$, let $\G^{\circ,c}$ be a directed stable graph on $\Go(c)$. We can consider the directed stable graph $\widetilde{\Go}$ obtained by gluing the stable graphs $(\G^{\circ,c})_{c\in X_0}$ according to $\Go$.
\begin{prop}
\label{prop_gluing_acycl}
If the stable graphs $\Go$ and $(\G^{\circ,c})_{c\in X_0}$ are acyclic, then $\widetilde{\Go}$ is also acyclic.
\end{prop}

\begin{proof}
A cycle on $\tilde{\Go}$ induces a possibly trivial cycle on $\Go$. If $\tilde{\Go}$ is not acyclic, we can consider an absolute primitive cycle in $\tilde{\Go}$. Then either this cycle induces a non-trivial cycle in $\Go$ and then $\Go$ is not acyclic; or the cycle is trivial and it is contained in a component $\G^{\circ,c}$ for some $c$, then this component is not acyclic. By contraposition, we obtain Proposition \ref{prop_gluing_acycl}.
\end{proof}
\paragraph{Directed stable trees:}
\label{paragraph_trees}
Another particular kind of directed graphs are directed trees; we give the following characterization: An edge $\gamma\in X_1\Go$ is constant if the function $l_\gamma$ is constant on $\Lambda_{\Go}(L)$ for each $L$ in the image of $\Lbord$.

\begin{lem}
\label{lem_tree_constant}
A directed, stable graph $\Go$ is a tree iff all the edges in $X_1^{int}\Go$ are constant.
\end{lem}

\begin{proof}
As $l_\gamma$ is linear, it is constant on $\Lambda_{\Go}(L)$ iff $dl_\gamma$ is zero on $K_{\Go}$. Moreover, $\Go$ is a tree iff $H_1(\G)=0$. We see before that $H_1(\G)\simeq K_{\Go}$, which is the kernel of $T\Lbord$, and then all the edges are constant if $\Go$ is a tree. Conversely, if all the edges are constant, then the map $T_{\Go}\rightarrow \R^{X_1\G}$ is zero on $K_{\Go}$. But this map is injective, and then $K_{\Go}=\{0\}$.
\end{proof}
We also give the following lemma:
\begin{lem}
\label{lem_constant_edges}
An edge is constant iff it splits a connected component of the graph into two connected components.
\end{lem}
\begin{proof}
In an equivalent way, an edge is constant if the tangent map $dl_\gamma$ vanishes on the space $K_{\Go}$. It is possible to identify this space with the homology group $H_1(\G)$. Let $\G_\gamma$ (see \ref{paragraph_graph}) be the stable graph obtained after removing $\gamma$. We have a natural map:
\begin{equation*}
H_1(\G_\gamma)\longrightarrow H_1(\G).
\end{equation*}
The LHS is also the kernel of $dl_\gamma$. Computing the dimensions we obtain,
\begin{equation*}
\dim H_1(\G)-\dim H_1(\G_\gamma)= 1 + \dim H_0(\G)-\dim H_0(\G_\gamma).
\end{equation*}
Then the kernel is equal to the full space iff the edge splits a component into two components.
\end{proof}
If $\gamma$ spares $\Go$ in two connected components, $\Go_\gamma=\Go_+\sqcup \Go_-$, with $\gamma$ oriented from $ \Go_-$ to $\Go_+$. The length $l_\gamma$ of $\gamma$ factors through $\Lbord$ and is the restriction of a linear function:
\begin{equation*}
l_{\Go,\gamma}: \Lambda_{\Mo}\longrightarrow \R.
\end{equation*}
If $I_\pm=\partial \Go_\pm \backslash \{\gamma\}$ then we have the expression
\begin{equation*}
 l_{\Go,\gamma}=\sum_{\beta\in I_+}\epsilon(\beta)l_\beta=-\sum_{\beta\in I_-}\epsilon(\beta)l_\beta.
\end{equation*}
which is due to relation between the boundary components on $\Go_\pm$.

\begin{figure}
    \centering
    \includegraphics[width=0.5\linewidth]{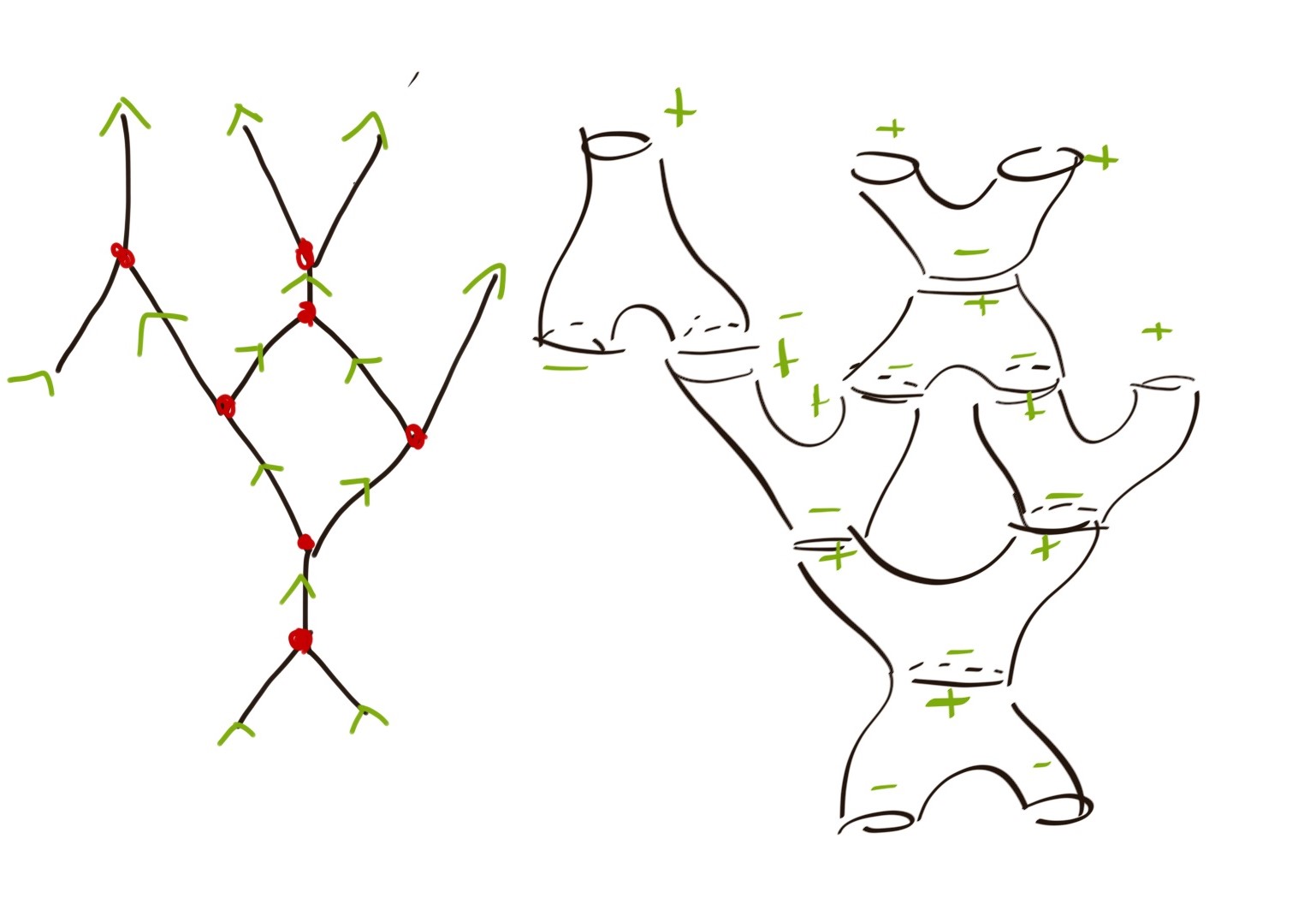}
    \caption{An acyclic directed stable graph.}
    \label{figure_acyclic_pant_decomp}
\end{figure}

\paragraph{Measures and volumes:}
Let $\Go$ be a directed stable graph marked by $\Mo$. For each $L$ in the image of $\Lbord$ the space $\Lambda_{\Go}(L)$ is a polytope with tangent space $K_{\Go}$ defined by the exact sequence in Formula \ref{formula_KGo}. The three spaces in this sequence contain a natural lattice of integer points $K_{\Go}(\Z),T_{\Go}(\Z),T_{\Mo}(\Z)$. In an affine space, a lattice in the tangent space defines a measure by assuming that the covolume of the lattice is one (see paragraph \ref{paragraph_measure_convex} for more details). We denote $d\sigma_{\Go}$, $d\sigma_{\Go}(L)$ and $d\sigma_{\Mo}$ these measures on $\Lambda_{\Go},\Lambda_{\Go}(L)$ and
$\Lambda_{\Mo}$. Then, Propositions \ref{prop_decomp_mesure_dir_stab_graph}  and \ref{lem_decompo_measure} alow to decompose the measures.

\begin{prop}
\label{prop_decomp_mesure_dir_stab_graph}
The sequence:
\begin{equation*}
0\longrightarrow K_{\Go}(\Z) \longrightarrow T_{\Go}(\Z) \longrightarrow T_{\Mo}(\Z) \longrightarrow 0,
\end{equation*}
is also exact. Then, for each $L\in \Lambda_{\Mo}$, the measure $d\sigma_{\Go}(L)$ is the conditional measure of $d\sigma_{\Go}$ with respect to $d\sigma_{\Mo}$.
\end{prop}
\begin{proof}
We can identify $K_{\Go}$ with the cohomology $H_1(\Go,\R)$, the exact sequence of Formula \ref{formula_KGo} is given by a long exact sequence of relative homology of graphs. The same sequence remains true for integral homology, which allows us to prove the first part of Proposition \ref{prop_decomp_mesure_dir_stab_graph}. For the second part, we can use Proposition \ref{lem_decompo_measure}.
\end{proof}

By using Corollary \ref{cor_acyclic_bounded}, if the graph $\Go$ is acyclic, the edges are bounded, then the subset $\Lambda_{\Go}(L)$ is bounded and has a finite volume. If $F=(F_c)_{c\in X_0\Go}$ is a family of functions such that $F_c$ is continuous on $\Lambda_{\Go(c)}$, it makes sense to compute the integral:
\begin{equation*}
V_{\Go}(F)(L)= \frac{1}{\#\Aut(\Go)|} \int_{x\in \Lambda_{\Go}(L)} \prod_{c\in X_0\G} F_c(L_c(x)) d\sigma_{\Go}(L),
\end{equation*}
where $L_c:\Lambda_{\Go}\to \Lambda_{\Go(c)}$ is the projection. In particular, we are interested in the function:
\begin{equation*}
\label{formula_ZGo}
V_{\Go}(L)= \frac{1}{\#\Aut(\Go)} \int_{x\in \Lambda_{\Go}(L)} \prod_{\gamma\in X_1^{int}\G} l_\gamma(x)d\sigma_{\Go}(L).
\end{equation*}
\begin{rem}
   When the graph is degenerate, it makes sense to define $V_{\Go}$. But according to Proposition \ref{prop_degenerate_graph_degenerate_edge}, the integral is equal to zero.
\end{rem}

\paragraph{Boundary lengths and wall configurations:}
For each $\Mo$ connected, a wall $W\in \Wall(\Mo)$ is the data of two subsets, $I_1^\pm,I_2^\pm$ such as
\begin{itemize}
\item $\partial^{\pm}\Mo=I_1^\pm\sqcup I_2^\pm$,
\item item If $I^\pm_i$ is empty, the set $I^\mp_i$ must contain exactly one element.
\end{itemize}
We denote $\Lambda_W$ the subset.
\begin{equation*}
\Lambda_W=\{L\in \Lambda_{\Mo} | |L_{I_1^+}|_1= |L_{I_1^-}|_1\}=\{L\in \Lambda_{\Mo}~|~|L_{I_2^+}|_1= |L_{I_2^-}|_1\}.
\end{equation*}
It is a subspace of $\Lambda_{\Mo}$ of codimension one.  To this configuration of walls, we can associate space $\Pc_{\Mo}$ (or $\Pc_{n^+,n^-}$) of continuous piecewise polynomials with walls in this configuration. For each wall $W$, we can define:
\begin{equation*}
L_W=||L_{I_1^+}|_1-|L_{I_1^-}|_1|=||L_{I_2^+}|_1-|L_{I_2^-}|_1|.
\end{equation*}
Then $\Pc_{\Mo}$ is the space of functions on $\Lambda_{\Mo}$, polynomials in the variables $L_W$.
\paragraph{Polynomial behavior of $V_{\Go}(F)$:}
\label{paragraph_polynomial_ZGo}
\begin{thm}
\label{thm_polynomial_ZGo}
Assuming that $\Go$ is acyclic and $F$ is a function on $\Lambda_{\Go}$, polynomial in the the variables  $(l_\gamma)_{\gamma\in X_1\Go}$ and homogeneous of degree $d$. Then the function $V_{\Go}(F)(L)$ is an homogeneous polynomial in the variable $L_W$ of degree $d(\Go)+d$. Where $W$ belongs to the set of faces of the polytope $\Lbord(\Lambda_{\Go})$.
\end{thm}
The function $V_{\Go}(F)(L)$ is then a piecewise polynomial, which is polynomial inside $\Lbord(\Lambda_{\Go})$ and vanishes on the boundary of the polytope.

\begin{rem}[Case of trees]
\label{rem_volume_ZGo_tree}
If $\Go$ is a tree, then the space $\Lambda_{\Go}(L)$ contains at most one element. By using the results of paragraph \ref{paragraph_trees}, we can see that:
\begin{equation*}
V_{\Go}(L)=\prod_\gamma l_{\Go,\gamma}(L).
\end{equation*}
Ended, each edge $\gamma$ spares the graph in two connected components; then it defines a wall $W_{\gamma}(\Go)$ and we have $l_{\Go,\gamma}=L_{W_{\gamma}(\Go)}$. We can see that the image of $L_{\partial}$ is the polytope defined by $\{L\in \Lambda_{\Mo}~|~l_{\Go,\gamma}(L)\ge 0~~\forall \gamma \}$. According to this, we see that we already proved Theorem \ref{thm_polynomial_ZGo} in this case. Moreover, the piecewise polynomial is explicit.
\end{rem}

\paragraph{Proof of theorem \ref{thm_polynomial_ZGo}:}
The proof of Theorem \ref{thm_polynomial_ZGo} relies on Ehrhart theory; it is an application of the following Theorem, which is a continuous version of the theorem $18.1$ given by Barvinok in \cite{barvinok2008integer}.\\

\begin{thm}
\label{thm_barvinock}
Let $(v_{i}(\alpha))_i$ vectors in $\R^n$ that depend on a parameter $\alpha$. Let $P_\alpha$ be a family of polytopes in $\R^n$ defined by
\begin{equation*}
P_\alpha=\text{conv}(v_1(\alpha),...,v_n(\alpha)).
\end{equation*}
 Assume that the cone of feasible directions of $P_\alpha$ at $v_i(\alpha)$ is independent of $\alpha$ for all $i$. Then, there is a polynomial $P$ such as
\begin{equation*}
\vol(P_\alpha)=P(v_1(\alpha),...,v_n(\alpha)).
\end{equation*}
\end{thm}
Then, to prove Theorem \ref{thm_polynomial_ZGo}, it remains to study the structure of the convex polytope $\Lambda_{\Go}(L)$. In our case, the set $V(\Lambda_{\Go}(L))$ of extremal points is related to spanning trees in $\Go$.

\begin{Def}
A spanning tree $\To$ of $\Go$ is a subset $E \in X_1^{int}\Go$ such that $\Go_E$ is a connected directed tree.
\end{Def}
If $\To$ is a spanning tree, according to Lemma \ref{lem_tree_constant}, for each edge $\gamma \in X_1\To=X_1\Go\backslash E$, there is a function
\begin{equation*}
l_{\To,\gamma}: \Lambda_{\Mo}\longrightarrow \Rp.
\end{equation*}
It expresses the length of $\gamma$ in terms of the boundary lengths $\Lbord$. There is a natural inclusion $\Lambda_{\To} \rightarrow \Lambda_{\Go}$, and this gives a map:
\begin{equation*}
x_{\To} : \Lbord(\Lambda_{\To}) \rightarrow \Lambda_{\Go}.
\end{equation*}
It is given explicitly by
\begin{equation*}
l_\gamma(x_{\To}(L))=\left\{
\begin{array}{ll}
l_{\To,\gamma}(L) & \text{if}~~ \gamma \in X_1\To \\
0 & \text{else.}
\end{array}
\right.
\end{equation*}
\begin{lem}
\label{lem_extremal_pts_LambdaGo}
For each $L\in \Lbord\left(\Lambda_{\Go}\right)$, the vectors $x_{\To}(L)$ are well defined and are the extremal points of $\Lambda_{\Go}(L)$.
\end{lem}
\begin{proof}
First of all, we have the following fact $\Lbord\left(\Lambda_{\Go}\right)= \Lbord\left(\Lambda_{\To}\right)$ for each spaning tree $\To$ in $\Go$. Then for each $L\in \Lbord\left(\Lambda_{\Go}\right)$, the element $x_{\To}(L)\in \Lambda_{\Go}$ is well defined. This is a consequence of the acyclicity. Then, we prove that the $(x_{\To}(L))$ are extremal points, for each $L\in \Lbord\left(\Lambda_{\Go}\right)$. Consider an element $u\in K_{\Go}$ such that the ray $x_{\To}(L)+tu$ belongs to $\Lambda_{\Go}(L)$, for $t\in \R$ close enough to the origin; then, we must have $l_\gamma(u)=0$ for all $\gamma\in E_{\To}$. Then $u$ is in the image of $H_1(\To)$. As $\To$ is a tree, $u=0$ and then $x_{\To}(L)$ is an extremal point. Conversely, given an extremal point $x$ let $E=\{\gamma|l_\gamma(x)=0\}$, we have a map $H_1(\Go_{E})\rightarrow H_1(\Go)$. For each $u$ in the image of this map, $x+tu$ belongs to $\Lambda_{\Go}(L)$ for $t$ small enough. Then we must have $H_1(\Go_{E})=0$, and then $\Go_{E}$ is a tree.
\end{proof}

The second ingredient is the following:

\begin{lem}
\label{lem_feasable_dir}
For each $\To$ spanning tree in $\Go$, the cone of feasible directions $\text{fcone}(\Lambda_{\Go}(L),x_{\To}(L))$ does not depend on $L$.
\end{lem}

\begin{proof}
This is straightforward, if $x_{\To}(L)+tu$ is in $\Lambda_{\Go}(L)$ for $t\in \Rp$ small enought iff $l_\gamma(u)\ge 0 , \forall \gamma\in E$ and then:
\begin{equation*}
\text{fcone}(\Lambda_{\Go}(L),x_{\To}(L))\simeq K_{\Go}\cap \{u\in \R^{X_1\Go}~|~l_\gamma(u)\ge 0~~~ \forall \gamma\in E\}
\end{equation*}
\end{proof}

\begin{proof}
With Lemmas \ref{lem_extremal_pts_LambdaGo} and \ref{lem_feasable_dir}, we can apply the Theorem \ref{thm_barvinock} and prove that there is a polynomial $P_{\Go}$ in vectorial variables $X_{\To}$ indexed by the spanning trees. The $X_{\To}$ take their arguments in $\R^{X_1^{int}\Go}$, and we have:
\begin{equation*}
V_{\Go}(1)(L)= P_{\Go}(((x_{\To}(L))_{\To}).
\end{equation*}
The RHS is a polynomial in $(l_{\To,\gamma})_{\gamma}$ and then is in $\Pc_{\Mo}$. To prove the general statement and deal with the factors $\prod_{\gamma}l_\gamma$, we consider: 
\begin{equation*}
\Lambda_{\Go}'(L)=\{(x,y)\in \Lambda_{\Go}(L)\times \R_{+}^{X_1^{int}\Go}~|~y_\gamma \le l_\gamma(x),~~~ \forall~\gamma\in X_1^{int}\Go\}.
\end{equation*}
By direct computation, the volume of $\Lambda_{\Go}'(L)$ is equal to $V_{\Go}(L)$. To treat this case, we just need to slightly change the last Lemmas. The set of extremal points in this case is bigger; each spanning three is associated with $2^{\#X_1^{int}\To}$ extremal points, and at each of these extremal points, the cone of feasible directions remains constant, which gives the proof by applying theorem \ref{thm_barvinock}. We can see that if $L\in \partial \Lbord(\Lambda_{\Go})$ and $x\in \Lambda_{\Go}(L)$, at least one of the $l_\gamma(x), x\in X_1^{int}\Go$ must vanishes, then $V_{\Go}(L)$ vanishes on $\partial \Lbord(\Lambda_{\Go})$.  The case of $F$ polynomials is similar; for a monomial $\prod_\gamma l_\gamma^{n_\gamma}$ we take the product with $\prod_\gamma [0,l_\gamma]^{n_\gamma}$ instead of $\prod_\gamma [0,l_\gamma]$.
\end{proof}

\newpage
\section{Metric ribbon graphs and moduli spaces:}
\label{section_ribbon}
\subsection{Ribbon graphs and metric ribbon graphs}
\paragraph{Combinatorial ribbon graph:}
\label{paragraph_def_combi_rib}
We give a careful definition of ribbon graphs, we use conventions similar to M. Kontsevich in \cite{kontsevich1992intersection}.
\begin{Def}
\label{def_combi_rib}
A combinatorial ribbon graph $R$ is defined by the data $(XR,s_0,s_1,s_2)$, where:
\begin{itemize}
 \item $XR$ is a set of half edges,
     \item $s_0,s_2 : XR \rightarrow XR$ are two permutations that define respectively the vertices and the boundary components (or faces) of the graph.
     \item  $s_1 : XR\rightarrow XR$ is an involution without fixed point,
     \item  These data satisfy the condition: $s_2s_1s_0=id$.
 \end{itemize}
 An isomorphism $\phi : R\rightarrow R'$ is a bijection $\phi : XR\longrightarrow XR$ that preserves these datas. The group $\Autt(R)$ is then the subgroup of $\mathfrak{S}(XR)$ that preserves $(s_2,s_1,s_0)$.
 \end{Def}
We denote $X_iR$ the i-cycles in $R$, $X_iR=XR/\langle s_i\rangle$ in $R$; and for all $e\in XR$ we denote $[e]_i\in X_iR$ the projection. A ribbon graph defines a graph in a natural way (see appendix \ref{paragraph_graph}). The orbits of $s_0$ define a partition of $XR$ indexed by $X_0R$, the blocks of the partition are the vertices of $R$. The involution $s_1$ encodes how to glue two half edges together to obtain an edge, then, $X_1R$ is the set of edges of $R$. The permutation $s_0$ gives an additional structure, a cyclic order on the half edges around each vertex (see figure \ref{figure_combi_ribbon}). As we explain later in more details, the set $X_2R$ represents the boundary components of a tubular neighborhood of $R$, and $s_2e$ is the successor of $e$ in the boundary (see figure \ref{figure_combi_ribbon} and paragraph \ref{zip_rect_RG}). We say that two boundary components $\beta,\beta'\in X_2R$ are adjascent if there is $e\in \beta$ with $s_1(e)\in \beta'$.

\begin{rem}[Decoration]
\label{rem_decoration_ribbon}
A ribbon graph $R$ defines a decoration $\nu_R$ in the sense of paragraph \ref{paragraph_surface_background}:
\begin{equation}
\nu_R(i)=\text{number of vertices of degree $2i+2$}.
\end{equation}
\end{rem}

\begin{rem}[Dual]
\label{rem_dual_ribbon_graph}
    A ribbon graph $R=(XR,s_0,s_1,s_2)$ admits a dual $R^*$ given by $XR^*=XR$ and $s_0^*=s_2^{-1}, s_2^*=s_0^{-1}, s_1^*=s_1$. In this picture, we have $X_0R^*=X_2R$ and $X_2R^*=X_0R$.
\end{rem}

\begin{figure}
     \centering
     \begin{subfigure}{0.4\textwidth}
        \includegraphics[width=\textwidth]{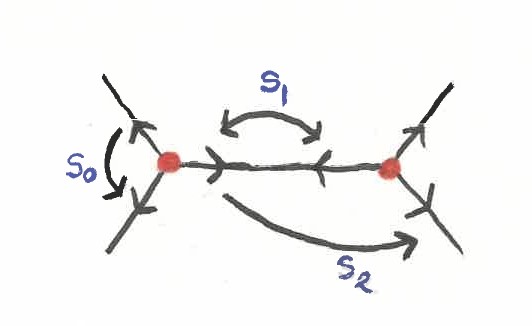}
        \caption{The three permutations\\ $(s_0,s_1,s_2)$.}
        \label{figure_combi_ribbon}
     \end{subfigure}
     \hfill
     \begin{subfigure}{0.4\textwidth}
        \includegraphics[width=\textwidth]{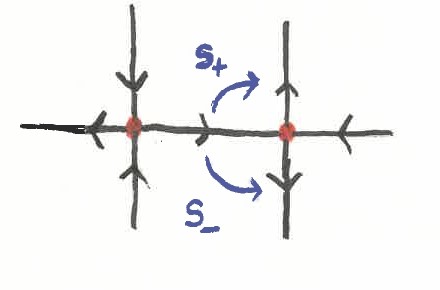}
        \caption{Combinatorix of directed\\ ribbon graphs}
        \label{figure_oriented_ribbon_schemas}
     \end{subfigure}
\end{figure}

\paragraph{Directed ribbon graph:}
\label{paragraph_def_oriented_RG}
\begin{Def}
\label{def_orientation_ribbon}
A direction $\epsilon$ on a ribbon graph $R$ is a map $\epsilon : XR\rightarrow \{\pm 1\}$, such that
\begin{equation*}
\epsilon\circ s_2=\epsilon~~~\text{and}~~~\epsilon\circ s_1=-\epsilon.
\end{equation*}
A ribbon graph is dirigible if it admits a direction, and a directed ribbon graph $\Ro$ is a couple $(R,\epsilon)$.
\end{Def}

\begin{figure}
\centering
\includegraphics[height=3cm]{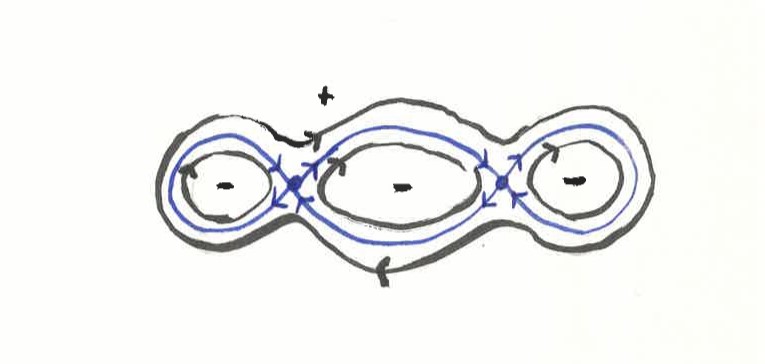}
\caption{A directed ribbon graph}
\label{fig_orientation_graph}
\end{figure}

Let $\Ro$ be a directed ribbon graph, we denote: $X^{\pm}\Ro=\{e\in XR~,~\epsilon(e)=\pm 1\}.$ It's sometimes convenient to use the two "zig-zags"\footnote {The terminology zig-zag comes from the article \cite{goncharov2021spectral}} permutations (see figure \ref{figure_oriented_ribbon_schemas} and \ref{figure_zippered_oriented}),
\begin{equation}
\label{formula_s+-}
s_+=s_2,~~~\text{and}~~~s_-=s_1s_2^{-1}s_1.
\end{equation}
We summarize some elementary properties of dirigible ribbon graphs in the following proposition.
\begin{prop}
\label{prop_orientation_list}
\begin{enumerate}
\item If $R$ is connected, it admits at most two orientations.
\item If $R$ is dirigible, then $R$ has only vertices of even degree.
\item A direction $\epsilon$ on $R$ induces a non-constant map
\begin{equation*}
\epsilon : X_2R \longrightarrow \{\pm 1\},
\end{equation*}
which defines a partition of the set of boundary components into two non-empty sets $X_2R=X_2^+R\sqcup X_2^-R$.
\item If $R$ is not dirigible, there is a canonical double cover $\tilde{R}$ that is dirigible and ramified over the vertices of odd degree.
\item A ribbon graph is dirigible iff its dual is bipartite, i.e., there is a map $\epsilon : X_0R^*\rightarrow \{\pm 1\}$ and two vertices joined by an edge have opposite signs.
\end{enumerate}
\end{prop}
\begin{proof}
    We prove $3$ and $4$. A direction $\epsilon$ on $R$ is invariant under $s_2$ and then induces a map $\epsilon : X_2R\to \{\pm 1\}$. To construct the double cover, we set $X\Rt=XR\times \{\pm 1\}$ and define $\tilde{s}_1(e,\epsilon)=(s_1e,-\epsilon),~\tilde{s}_0(e,\epsilon)=(s_0e,-\epsilon)$. These data define a ribbon graph $\tilde{R}$, the first projection defines a morphism from $\Rt$ to $R$ and the second projection defines a direction on $\Rt$. We can check that $\Rt$ satisfies the following universal property: if $R'\to R$ is a morphism of ribbon graphs and $R'$ is a dirigible graph, then the morphism factor through $\Rt$. 
\end{proof}
\paragraph{Zippered rectangles and embedded ribbon graphs}
\label{zip_rect_RG}

A ribbon graph $R$ naturally defines a surface $M_R$ in $\bord$, its \textit{topological realization}. For each $e\in XR$ we can consider a rectangle $\mathbf{R}_e=[0,1]\times [-1,1]$, we glue these rectangles in the following way: we identify  $\{0\}\times [0,1] \in \mathbf{R}_e$ to $\{0\}\times [-1,0] \in \mathbf{R}_{s_1e}$. Then we identify $\mathbf{R}_e$ and $\mathbf{R}_{s_1e}$ by a rotation of angle $\pi$ (see figure \ref{figure_zippered_non_directed}). We obtain in this way a surface $M_R$ in $\bord$, and the image of the lines $[0,1]\times \{0\}$ defines an embedded topological graph on which $M_R$ retracts. We can identify $X_0R$ with the set of vertices, $X_1R$ with the edges and $X_2R$ with the boundary components. We denote $g(R)$ the genus of $M_R$ and $n(R)=\#X_0R$ the number of boundary components.

\begin{figure}[h!]
    \centering
    \includegraphics[width=0.5\linewidth]{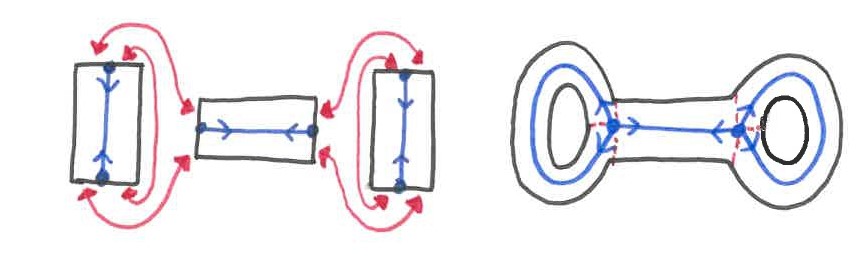}
    \caption{Zippered rectangles for a ribbon graph.}
    \label{figure_zippered_non_directed}
\end{figure}
\begin{rem}[Functoriality]
If $\phi: R_1\to R_2$ is a morphism of ribbon graphs, it induces a covering $\phi: M_{R_1}\rightarrow M_{R_2}$ possibly ramified over the vertices of $R_2$.
\end{rem}

If $\Ro$ is a directed ribbon graph, we can perform a slightly different construction. We consider rectangles $\mathbf{R}_e$ indexed by positive half-edges $e\in X^+\Ro$, and we glue these rectangles according to figure \ref{figure_zippered_oriented} by using $s_+$ and $s_-$. In this case we do not rotate the rectangles, then we can orient the edges from left to right. Moreover, some boundaries are at the top of the rectangles and others are at the bottom. The first ones are labeled by $+$ and the second by $-$, we do not flip the rectangles, so the choice is consistent. Then a directed ribbon graph $\Ro$ defines a directed surface $\Mo_{\Ro}$, we can identify $X_2R$ and $\partial M_R$ (see figure \ref{figure_zippered_oriented}). This gives an explanation of points $3$ and $5$ in Proposition \ref{prop_orientation_list}.

\begin{rem}[Double cover $\tilde{M}_R$]
\label{rem_tilde{M}_R}
We remark that if $R$ is not necessarily dirigible, we can consider its directed cover $\Rt$. The surface $\tilde{M}_R=M_{\Rt}$ is then a ramified double cover over $M_R$, by functoriality, the involution $\sigma_R$ induces an involution $\sigma_R$ in $\tilde{M}_R$ and $M_R$ is the quotient.
\end{rem}

\begin{rem}[Surface $M_R^{\bullet}$]
\label{rem_M_R^{bullet}}
We can also consider rectangles $\mathbf{R}_e^\bullet=[0,1]\times \R$ and glue them in a similar way. The resulting surface is in $\bordb$ and will be denoted $M_R^{\bullet}$
\end{rem}

If $M\in \bord$ and $R$ is a ribbon graph without univalent or bivalent vertices, an embedding of $R$ is an isotopy class of homeomorphims $\phi: M_R\to M$. We denote $\Rib(M)$ the set embedded ribbon graphs. If $\Mo$ and $\Ro$ are directed, we assume that the homeomorphisms preserve the sign of the boundary components, we denote $\Rib(\Mo)$ the set of embedded direction ribbon graphs on $\Mo$. The mapping class group of $M$ acts on $\Rib(M)$ and $\Rib(\Mo)$, we denote $\rib(M)$ and $\rib(\Mo)$ the quotients. We also use the notations $\rib_{g,n}$ (resp. $\rib_{g,n^+,n^-}$) for ribbon graphs of genus $g$ with $n$ labeled boundary components (resp. directed with $n^+$ positive and $n^-$ negative labeled boundary components).

\begin{rem}[Marked points]
\label{rem_marked_point_embedded_ribbon}
 When we consider ribbon graphs with vertices of degree one or two, to perform surgeries, we must consider the vertices as marked points on the surface $M_R$. Then, if $M$ (resp. $\Mo$) has marked points, we consider homeomorphisms $M_R\to M$ such that the preimage of marked points are vertices of $R$. We allow marked vertices to be of degree one or two. 
\end{rem}

\begin{figure}[h!]
\centering
\includegraphics[height=5cm]{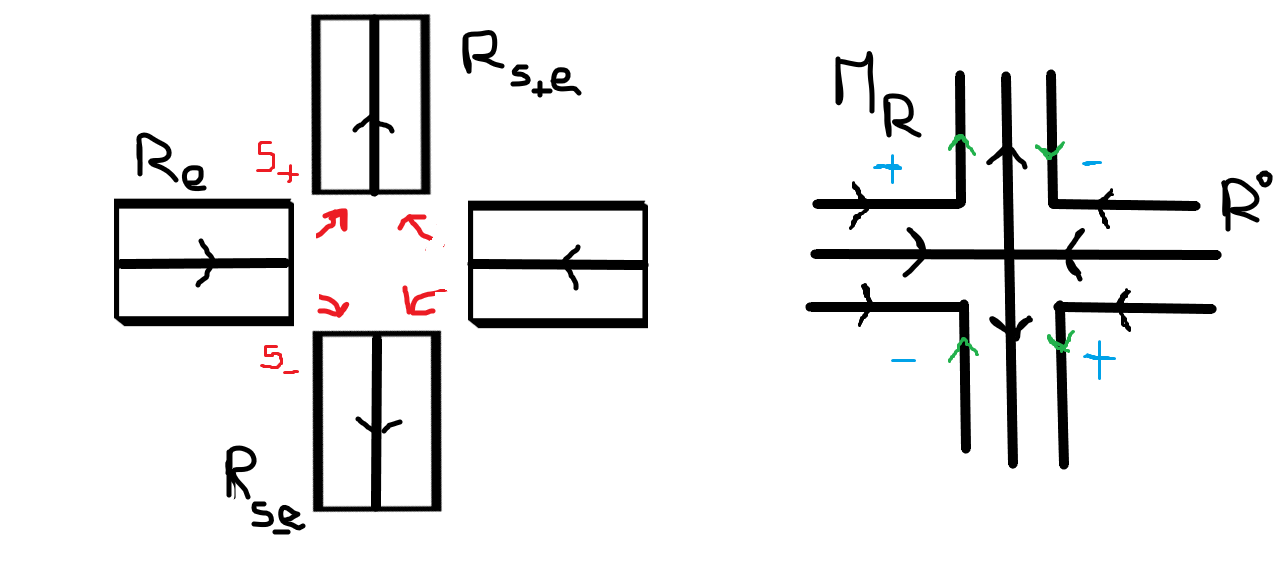}
\caption{Zippered rectangles for a directed ribbon graph $\Ro$, orientation of the boundary components induced by the orientation of the surface $M_R$ in green, direction given by $\Ro$ in black, sign's in blue.}
\label{figure_zippered_oriented}
\end{figure}
\paragraph{Metric ribbon graph and boundary lengths:}
A metric $m$ on a ribbon graph $R$ is just a map 
\begin{equation*}
m: X_1R \longrightarrow \Rpp.
\end{equation*}
We denote, respectively, $\Met(R)=\R_{>0}^{X_1R}$ and $T_R=\R^{X_1R}$, the cone of metrics on $R$ and $T_R$ the tangent space. We also use the notation $m_e: \Met(R)\rightarrow \R,$ for the canonical coordinates. The following lemma gives the dimension of a cell according to the topology of the ribbon graph.
\begin{lem}
\label{lem_dim_TR}
Let $\db(R)$ the dimension of $T_R$ we have: 
\begin{equation}
\label{formula_db(R)}
    \db(R)=2g(R)-2+n(R)+n(\nu_R).
\end{equation}
\end{lem}
\begin{proof}
    Clearly, we have $\dim T_R=\#X_1R$, a ribbon graph defines a cell decomposition of $M_R^{cap}$ and this can be used to compute the Euler characteristic we have in the connected case 
    \begin{equation*}
        2g(R)-2=-\#X_0R+\#X_1R-\#X_2R=-n(\nu_R)+\dim T_R - n(R).
    \end{equation*}
\end{proof}
Let $M$ in $\bordb$ and $\beta\in \partial M$. For each embedded metric ribbon graph $S=(R,m)$ on $M$, it's possible to compute the length of $\beta$ by summing the lengths of the edges that are adjacent to the boundary component
\begin{equation}
\label{formula_lbeta}
l_\beta(S)= \sum_{e \in XR,~[e]_2=\beta} m_{[e]_1}(S).
\end{equation}
This defines a linear function $l_\beta: \Met(R)\to\Rpp$, we denote $L_\partial=(l_\beta)_{\beta\in \partial M}$ and $T\Lbord$ the tangent map. Let $I_R=T\Lbord(T_R)$ be the image of $T\Lbord$  and $K_R$ its kernel. We have the trivial, exact sequence
\begin{equation}
\label{formula_exact_sequence_KR}
\{0\}\longrightarrow K_R \longrightarrow T_R\longrightarrow I_R \longrightarrow \{0\}.
\end{equation}
We give in Proposition \ref{prop_dimorientability} the dimension of $I_R$ and $K_R$.
\paragraph{Cohomology of a ribbon graph:}
\label{paragraph_coho_ribbongraph}
We construct the cohomology groups of a ribbon graph; it simplifies some construction and will be used in Proposition \ref{prop_dimorientability} and Lemma \ref{lem_deg_symp_pair}. We start with the case of a directed graph $\Ro$, the construction is straightforward. The complex of chains $C_*(\Ro)$ given by:
\begin{equation*}
C_0(\Ro)=\R^{X_0R},~~~C_1(\Ro)=\R^{X_1^+\Ro},~~~\text{and}~~~C_2(\Ro)=\R^{X_2R}.
\end{equation*}
We define boundary operators, for each $\beta \in X_2R$ and $e\in X_1^+R$
\begin{equation*}
\partial \beta = -\sum_{e',[e']_2=\beta}\epsilon(e')[e']_1^+,~~~\text{and}~~~\partial e = [s_1e]_0-[e]_0.
\end{equation*}
The complex of cochain's $C^*(\Ro)$ is defined by duality, and we denote $H_*(R)$ and $H^*(R)$ the homology and cohomology groups\footnote{the construction does not depend of the choice of the direction, that why we use the notation $H(R)$ instead of $H(\Ro)$}. Using the fact that $\Ro$ defines a cell decomposition, we can see that this complex computes the homology and cohomology of $M_R^{cap}$, which is the surface obtained after gluing a disc on each boundary component of $M_R$. 

Similarly, we can consider relative homology and cohomology group. Let $C_*(X_0R)$ be the complex with only one non trivial element given by $C_0(X_0R)=C_0(\Ro)$ and $C_*(\Ro,X_0R)=C_*(\Ro)/ C_*(X_0R)$ the relative complex. This complex computes the homology $H_*(M_R^{cap},X_0R,\R)$. We can also define the complex of cochains $C^*(X_2R)$ with only one non trivial element $C^2(X_2R)=\R^{X_2R}$, and we form:
\begin{equation*}
C^*(R,X_0R,X_2R)=C^*(R,X_0R)/ C^*(X_2R),
\end{equation*}
$C_*(\Ro,X_0R,X_2R)$ is defined by duality. These complexes compute $H^*(M_R,X_0R,\R)$ and $H_*(M_R,X_0R,\R)$. To summaries we have isomorphisms
\begin{equation*}
    H^*(R,X_0R)\simeq H^*(M_R^{cap},X_0R,\R)~~~\text{and}~~~H^*(R,X_0R,X_2R)\simeq H^*(M_R,X_0R,\R)
\end{equation*}
 For each edge $e\in X_1R$ there is a unique directed edge $e_+\in X^+\Ro$ with $[e_+]_1=e$, it defines a cycle $[e_+]\in H_1(M_R,X_0R,\R)$ and these cycles form a basis of the homology. Then by duality, we can identify $T_R$ and $H^1(R,X_0R,X_1R)\simeq H^1(M_R,X_0R,\R)$, we denote $f_{\Ro}$ the map\footnote{This map depend of the choice of a direction $f_{-\Ro}=-f_{\Ro}$}:
\begin{equation*}
    f_{\Ro} : T_R \to H^1(R,X_0R,X_1R),
\end{equation*}
 which is defined by:
 \begin{equation*}
     \langle f_{\Ro}(x),[e_+]\rangle= x_e\quad \forall e\in X_1R.
 \end{equation*}
In general, to construct the cohomology of a ribbon graph, we can use the directed double cover $\Rt$. The involution $\sigma_R$ acts on the different homology group and satisfies $\sigma_R[e_+]=-[\sigma_R(e)_+]$, then it forces to use the anti-invariant cohomology. We can decompose the cohomology according to the eigenvalues of $\sigma_R$ into even and odd parts, and we denote $H^*(R)=H^*(\Rt)^-$. As before, we can still identify $H^1(R,X_0R,X_2R)$ with $T_R$. We denote
\begin{equation*}
    f_R: T_R \to H^1(R,X_0R,X_2R),
\end{equation*}
each edge $e\in XR$ defines an element $[e]$ of $H_1(R,X_0R,X_2R)$ and we have $x_e=\langle f_R(x),[e]\rangle$. We $R$ is dirigible; this construction coincides with the one given precedently.

\paragraph{Lengths of boundary components and dirigibility:}
\label{paragraph_orient_boundaries}
The fact that the dual of a directed ribbon graph is bipartite implies the following lemma:
\begin{lem}
\label{lem_orientation_relation}
Let $\Mo$ be a connected directed surface, and $\Ro=(R,\epsilon)\in \Rib(\Mo)$; we have on $\Met(R)$
\begin{equation}
\label{formula_sum_zero_oriented}
\sum_{\beta \in \partial M} \epsilon(\beta) l_\beta= 0.  
\end{equation}
\end{lem}
Then the image of the application $L_{\partial}$ lies in the hyperplane $\Lambda_{\Mo}$ (equation \ref{formula_LambdaMo}) we have $I_R\subset T_{\Mo}$.
\begin{proof}
By Proposition \ref{prop_orientation_list} each edge $e\in X_1R$ is contained in exactly one positive and one negative boundary, then for each $m\in \Met(R)$ we have
 \begin{equation*}
\sum_{\beta,\epsilon(\beta)=1}l_\beta= \sum_{e}m_e=\sum_{\beta,\epsilon(\beta)=-1}l_\beta
\end{equation*}
which give the claim.
\end{proof}

Each $\beta$ defines an element $[\beta]\in H_1(R,X_0R,X_2R)$; moreover, we can see that:
\begin{equation*}
    dl_\beta(x)=\langle f_{R}(x),[\beta]\rangle.
\end{equation*}
The next proposition characterizes dirigibility, and despite its simplicity, it is very useful for us.

\begin{prop}
\label{prop_dimorientability}
Let $R$ be a connected ribbon graph not necessarily dirigible with $n$ boundary components. The dimension of $I_R$ is
\begin{itemize}
\item $n$ if $R$ is not dirigible,
\item $n-1$ if $R$ is dirigible and the only relation is given by the direction
\begin{equation*}
\sum_\beta \epsilon(\beta)dl_\beta=0.
\end{equation*}
\end{itemize}
\end{prop}

\begin{proof}
We give a cohomological proof. From paragraph \ref{paragraph_coho_ribbongraph}, $T_R$ is identified with $ H^1(R,X_0R,X_2R)$. The short, exact sequence of complexes
\begin{equation*}
\{0\}\longrightarrow C^*(X_2R)\longrightarrow C^*(R,X_1R)\longrightarrow C^*(R,X_1R,X_2R)\longrightarrow
\{0\},
\end{equation*}
leads to a long, exact sequence:
\begin{equation*}
\{0\}\longrightarrow H^1(R,X_0R)\longrightarrow H^1(R,X_0R,X_2R) \longrightarrow H^2(X_2R) \longrightarrow H^2(R,X_0R)\longrightarrow \{0\}.
\end{equation*}
We can identify $H^2(X_2R)\simeq \R^{\partial M}, H^1(R,X_0R,X_2R)\simeq T_R, H^2(R,X_0R)\simeq H^2(R)$. The second nontrivial map in the sequence corresponds to the evaluation of $T\Lbord$. Then we can rewrite the sequence in the following way:
\begin{equation*}
\{0\}\longrightarrow K_R \longrightarrow T_R\overset{T\Lbord}{\longrightarrow}  \R^{\partial M}\longrightarrow H^2(R)\longrightarrow \{0\}.
\end{equation*}
Finally, $H^2(R)$ can be identified with $H^2(M_{\Rt}^{cap},\R)^-$; the surface $M_{\Rt}^{cap}$ is connected if $R$ is non-dirigible, otherwise, it has two connected components exchanged by the involution $\sigma_R$. Then we have
\begin{equation*}
H^2(R) = \left\{
\begin{array}{ll}
\R & \mbox{if R is dirigible } \\
0 & \mbox{else.}
\end{array}
\right.
\end{equation*}
The obstruction to being dirigible is then in $H^2(R)$. To conclude, from the exact sequence we obtain
\begin{equation*}
\dim I_R = \#\partial M - \dim H^2(R),
\end{equation*}
which proves Proposition \ref{prop_dimorientability}. According to Lemma \ref{lem_orientation_relation}, the only possible relation is given by the direction.
\end{proof}

\begin{rem}
\label{rem_KR_H(R,X0R)}
    According to the proof of Proposition \ref{prop_dimorientability}, we obtain that the map $f_R$ induces an isomorphism
\begin{equation*}
\label{formula_K_R_cohomology}
f_R: K_R\to H^1(R,X_0R).
\end{equation*}
\end{rem}

\subsection{Relations between metric ribbon graphs, weighted multi-arcs, and measured foliations}

\paragraph{Ribbon graphs and filling multi-arcs:}
\label{paragraph_ribbon_filling_arc}
In this part, we picture the relation between multi-arcs and ribbon graphs; similar results are also presented in \cite{arbarello2011geometry} and at several other places in the literature.
\begin{lem}
\label{lem_ribbon_multiarc_1}
Let $M\in \bord$ each embedded ribbon graph $R$ in $M$, we can associate a multi-arc $A(R)$.
\end{lem}
We give an illustration of this construction in figure \ref{fig_ribbon_arc}.
\begin{proof}
 Let $R$ be an embedded ribbon graph on $M$ with no vertices of degree one or two. For each edge $e\in X_1R$, there is a unique arc $e^*\in \A(M)$ in $\mathbf{R}_e$ that joins the two horizontal boundaries of the rectangle $\mathbf{R}_e$ (namely $\{1/2\}\times [0,1]$). The arc $e^*$ intersects $e$ and no other edge of the graph. Then the union of all the arcs $e^*$ forms a multi-arc $A(R)$ (see figure \ref{fig_ribbon_arc}). The fact that the arcs are non trivial and pairwise non homotopic is a consequence of the Bigon criterium (see Remark \ref{rem_bigon_criterion_2}). We can see that the construction is compatible with isotopies, then the map $ A:~\Rib(M)\longrightarrow \MA(M)$ is well defined \footnote{When there is vertices of degree one or two, $A(R)$ is still a multi-arc, but on the surface obtained by removing the corresponding points (we make this choice to be consistent with the non triviality of the arcs; see Remark \ref{rem_bigon_criterion_2}).}.
\end{proof}

A ribbon graph defines a multi-arc, but the converse is not always true; the map $A$ is not surjective. We give the following definition:

\begin{Def}
\label{def_filling_multiarc}
A multi-arc $A \in \MA(M)$ is filling\footnote{The terminology is borrowed from the theory of measured foliations \cite{lindenstrauss2008ergodic} but such multi-arcs are also called proper \cite{arbarello2011geometry}} if $\iota(A,\gamma)>0$ for all $\gamma\in \Sit(M)$. We denote $\MAo(M)$ and $\MAor(M)$ the subset of filling multi-arcs and weighted multi-arcs.
\end{Def}

The following proposition relies on filling multi-arcs and ribbon graphs:

\begin{prop}
\label{prop_filling_arc_rib}
A multi-arc $A\in \MA(M)$ is filling iff $A=A(R)$ for an embedded ribbon graph $R$, the map
\begin{equation*}
A: \Rib(M)\rightarrow \MA^0(M),
\end{equation*}
is a bijection, and we denote $R$ the inverse. If $\Mo$ is directed we also have the identification $\Rib(\Mo)\simeq \MA^0(\Mo)$.
\end{prop}

This proposition is a consequence of the following criteria to check if a multi-arc is filling or not. We recall that for all multi-arc $A$ on $M$, $M_A$ is the surface obtained after surgeries along the arcs in $A$ (see Remark \ref{rem_bigon_criterion_2}).

\begin{lem}
\label{lem_filling_disc}
A multi-arc is filling iff all the components of $M_A$ are topological discs with at most one interior marked point.
\end{lem}

\begin{figure}
    \centering
    \includegraphics[width=0.4\textwidth]{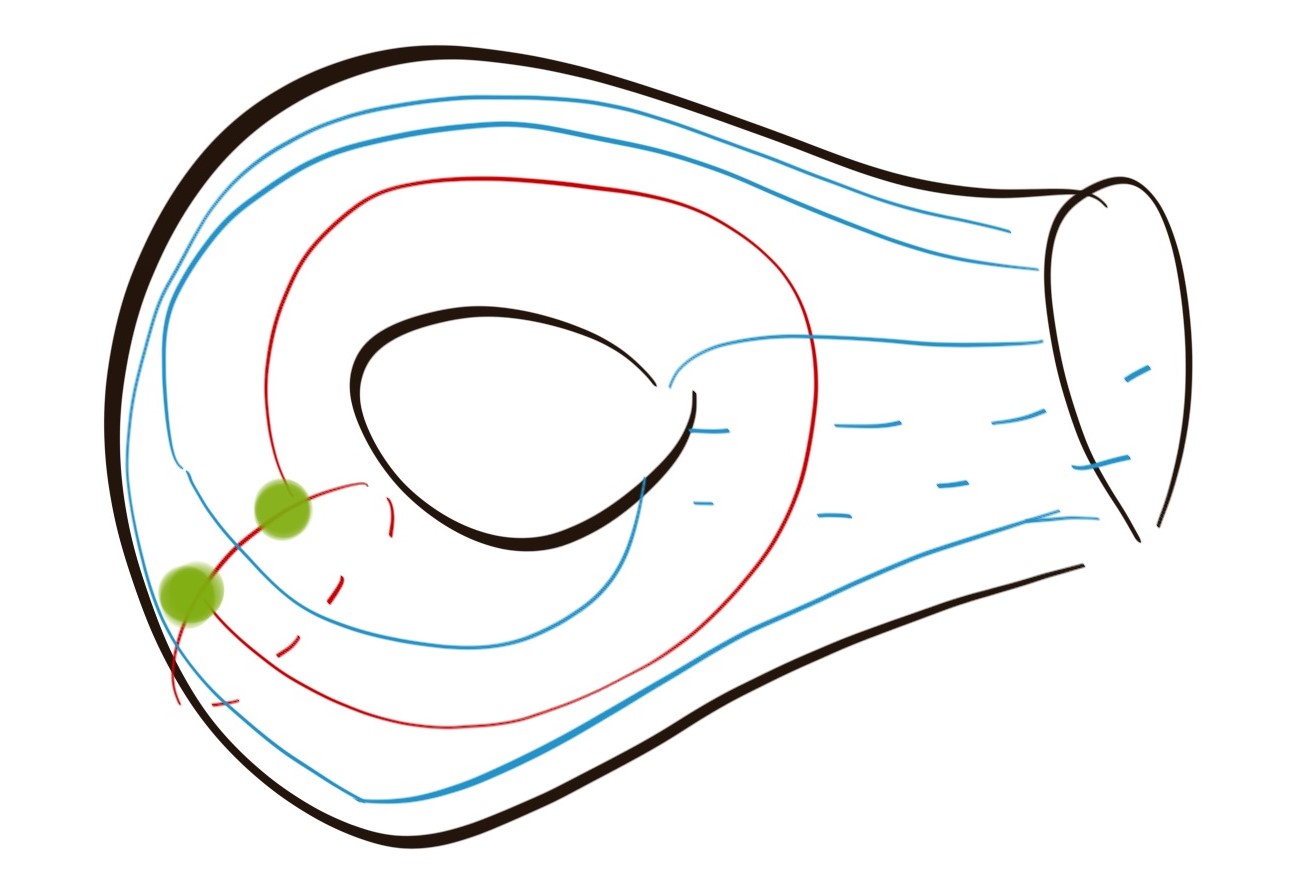}
    \caption{Ribbon graph and the multi-arc on a torus.}
    \label{fig_ribbon_arc}
\end{figure}

\paragraph{Metric ribbon graphs and filling foliations:}
\label{paragraph_ribbon_filling_fol}
According to the results of the last paragraph, a ribbon graph defines a filling multi-arc; similarly, a metric ribbon graph $S$ defines a weighted multi-arc 
\begin{equation*}
    A(S) = \underset{e\in X_1R}{\sum}m_e(S) e^*.
\end{equation*}
The map $ A:~S\rightarrow A(S)$ induces a bijection between embedded metric ribbon graphs and filling weighted multi-arcs. We can also define the notion of filling foliations by replacing $A$ by $\lambda$ in Definition \ref{def_filling_multiarc}.

\begin{prop}
\label{prop_ribbon_fol}
To each embedded metric ribbon graph $S$ in $M$, we can associate a filling foliation $\lambda(S)\in \MF^0(M)$ such that
\begin{equation*}
\iota(\lambda(S),\gamma)= \iota(A(S),\gamma),~~~\forall \gamma\in \tilde{\Si}(M).
\end{equation*}
And foliation $\lambda\in \MF(M)$ is filling iff $\lambda=\lambda(S)$ for a metric ribbon graph $S$.\footnote{This construction gives the map $\MA^{0}(M)\longrightarrow \MF(M)$ that preserves the intersection pairing. A generalization of this construction allows to prove the statement of paragraph \ref{paragraph_multiarc_fol} in general. The converse map is then given by the map $A$ defined in Proposition \ref{prop_zippered_fol_general}, which gives a bijection between filling weighted multi-arcs and filling foliations,
\begin{equation*}
A : \MF^0(M)\longrightarrow \MA^0_\R(M).
\end{equation*}}
There exist a unique quadratic differential $q_S(0)\in \QT(M)$ such that 
\begin{equation*}
    \text{Re} q_S(0)=\lambda(S),\quad \text{and} \quad \text{Im} q_S(0)=\lambda(0).
\end{equation*}
And we have 
\begin{equation*}
    \Res_{\beta^{\bullet}}q_S(0)=l_\beta(S),~~\forall \beta \in \partial M.
\end{equation*}
If $\So$ is a directed metric ribbon graphs, we can associate a directed foliation $\lambda^{\circ}(\So)$ and an abelian differential $\alpha_{\So}(0)$. Moreover, we have: 
\begin{equation*}
\Res_{\beta^{\bullet}}q_S(0)=\epsilon(\beta)l_\beta(S),~~\forall \beta \in \partial M,
\end{equation*}
we orient the boundary components according to the orientation of $M$.
\end{prop}

\begin{proof}
    We use the zippered rectangles construction of paragraph \ref{zip_rect_RG}. Let $(R,m)$ be a metric ribbon graph. For each $e\in XR$, we consider on $\mathbf{R}_e^{\bullet}$ (see Remark \ref{rem_M_R^{bullet}}) the measured foliation given by $m_e dx$, this choice is compatible with the gluings and defines a measured foliation in $\MF(M_R)$, in the sense of section \ref{subsection_foliation}. We can check that, at a vertex of degree $k$, the foliation has a singularity of order $\frac{k-2}{2}$. Similarly, we can define $q_S(0)$ locally by $q_S(0)_{|\mathbf{R}_e^{\bullet}}=m_edz^2$ (this choice is compatible with the gluings because $dz^2$ is invariant by a rotation of angle $\pi$). It has a double pole at each marked point $\beta^{\bullet}$ for $\beta\in \partial M$. We can see that $q_S(0)$ is Jenkin-Strebel; leaves of the horizontal foliation are periodic, and all non-singular trajectories are circles around double poles. Then the horizontal foliation is the trivial foliation $\lambda(0)$. The fact that a filling foliation is given by a weighted multi-arcs is a consequence of Proposition \ref{prop_zippered_fol_general}. In the case of directed ribbon graphs $\Ro$, the gluings preserve the abelian differential defined locally by $\alpha_{\So}(0)_{|\mathbf{R}_e^{\bullet}}=m_edz$. 
\end{proof}

\subsection{Moduli spaces of metric ribbon graphs and their volumes}
\label{moduli_space_volumes}
\paragraph{Teichmüller spaces:}
Let $M\in \bord$, the Teichmüller space $\Tc(M)$ of metric ribbon graphs on $M$ is the space of all embedded metric ribbon graphs in $M$. It is a disjoint union of cells:
\begin{equation}
 \label{formula_Tc(M)}
\Tc(M)=\underset{R\in \Rib(M)}{\bigsqcup}\Met(R).
\end{equation}
A ribbon graph $R$ can degenerate into another ribbon graph $R_{\langle E\rangle}$ by contracting a set of edges $E$ that form a disjoint union of sub-trees $E$ (see Figure \ref{figure_pant_degeneration}). If $M\in \bordb$ we do not allow to contract an edge that joins two marked vertices. The quotient of a ribbon graph is also a ribbon graph (see paragraph \ref{paragraph_graph}); and we have a map
\begin{equation*}
\Met(R_{\langle E\rangle})\longrightarrow \Metb(R),
\end{equation*}
with $\Metb(R)=\Rp^{X_1R}\backslash \{0\}$. Degenerations of ribbon graphs induce a structure of cell complex on $\Tc(M)$. 
The top cells in $\Tc(M)$ correspond to the subset $\Rib^*(M)$ of embedded ribbon graphs that are trivalent with univalent vertices at the marked points. If $M$ is of type $(g,n,m)$, using Lemma \ref{lem_dim_TR} and Proposition \ref{prop_dimorientability} we obtain
\begin{equation*}
    \dim(T_R)= 6g-6+3n+2m~~~\text{and}~~~\dim K_R= 6g-6+2n+2m.
\end{equation*}
if $R$ in $\Rib^*(M)$. Then, in this case, the combinatorial moduli space is a cell complex of dimension $6g-6+3n+2m$.\\
\begin{figure}
\centering
\includegraphics[width=0.5\textwidth]{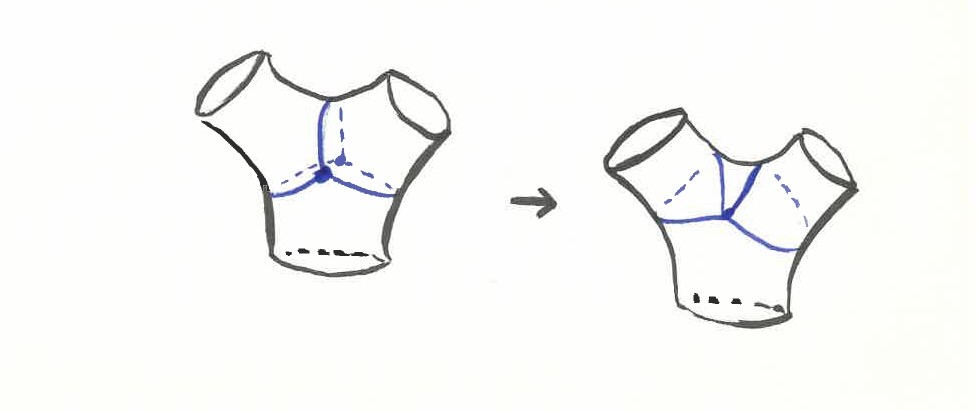}
\caption{Degeneration of a ribbon graph on a pair of pant's.}
\label{figure_pant_degeneration}
\end{figure}
If $\Mo$ is a directed surface, we denote $\Tc(\Mo)$ the Teichmüller space of directed metric ribbon graphs on $\Mo$:
\begin{equation}
\label{formula_Tc(Mo)}
\Tc(\Mo)=\underset{\Ro\in \Rib(\Mo)}{\bigsqcup}\Met(R).
\end{equation}
We can see that a degeneration of a directed graph is also directed and this defines a structure of cell complex on $\Tc(\Mo)$ and the projection
\begin{equation*}
 \Tc(\Mo)\rightarrow \Tc(M)
\end{equation*}
defines a subcomplex of $\Tc(M)$. In the case of directed ribbon graphs, the top cells correspond to quadrivalent directed graphs with bivalent vertices at the marked points. In this case, the dimension of a top cell is given by:
\begin{equation*}
\dim T_R= 4g-4+2n^++2n^-+m~~~~ \dim K_R= 4g-3+2n^++2n^-+m.
\end{equation*}
Then $\Tc(\Mo)$ is of dimension $4g-4+2n^++2n^-+m$.

\paragraph{Combinatorial Teichmüller spaces and measured foliations:}
\label{paragraph_combiteich_folliation}
This part is just a reformulation of the results of Proposition \ref{prop_filling_arc_rib} and \ref{prop_ribbon_fol}. There is two injective maps
\begin{equation*}
    \Tc(M)\longrightarrow \MA_\R(M)\longrightarrow \MF(M).
\end{equation*}
Moreover, by using Propositions \ref{prop_filling_arc_rib} and \ref{prop_ribbon_fol}, we have seen that the Teichmüller space is identified with the space of filling foliations and also with the space of filling multi-arcs
\begin{equation*}
    \Tc(M)= \MA_\R^0(M)=\MF^0(M).
\end{equation*}
In this picture, the top cells $\T^{comb,*}(M)$ of $\Tc(M)$ correspond to the subspace $\MF^*(M)$ of foliations with simple zeros and no saddle connection (and simple poles at the marked points). In a similar way, we have the inclusions
\begin{equation*}
    \Tc(\Mo)\longrightarrow \MA_\R(\Mo)\longrightarrow \MF(\Mo).
\end{equation*}
Proposition \ref{prop_filling_arc_rib} gives the identifications
\begin{equation*}
    \Tc(\Mo)= \MA_\R^0(\Mo)=\MF^0(\Mo).
\end{equation*}
In this case, $\T^{ comb,*}(\Mo)$ corresponds to the space $\MF^{*}(\Mo)$ of directed foliations with singularities of order $2$ and no saddle connection.

\paragraph{Moduli spaces:}
\label{paragraph_def_moduli_space}
Let $M\in \bordb$, the mapping class group $\Mod(M)$ acts on $\Tc(\M)$, the moduli space $\Mc(M)$ is the quotients under this action. The moduli space is an orbifold cell complexes; the set of cells is $\rib(M)=\Rib(M)/\Mod(M)$ which is the set of combinatorial ribbon graphs (see Definition \ref{def_combi_rib}). It's finite, and we have
\begin{equation*}
\Mc(M)=\bigsqcup_{R\in \rib(M)} \Met(R)/\Aut(R),
\end{equation*}
where $\Aut(R)$ is a finite group that acts freely by linear transformations. If $\Mo$ is a direction on $M$, we can also define $\Mc(\Mo)=\Tc(\Mo)/\Mod(M)$, we have
\begin{equation*}
\Mc(\Mo)=\bigsqcup_{\Ro\in \rib(\Mo)} \Met(R)/\Aut(R).
\end{equation*}
Then $\Mc(\Mo)$ is a sub-complex in $\Mc(M)$.
\paragraph{Level sets:}
\label{paragraph_level_set}

Let $M\in \bordb$, for $\beta \in \partial M$ the length\footnote{see paragraph \ref{paragraph_orient_boundaries}} $l_\beta$ of $\beta$ defines a function:
\begin{equation*}
    l_\beta : \Tc(M)\to \R.
\end{equation*}
We also denote $\Lbord=(l_\beta)_{\beta\in \partial M}$, it is invariant under the action of $\Mod(M)$, and it descends to a map on the moduli spaces $\Mc(M)$. We denote
\begin{equation*}
    \Mc(M,L)=L_\partial^{-1}(\{L\}),\quad \forall L \in \Rpp^{\partial M}
\end{equation*}
If $\Mo$ is a directed surface, by Proposition \ref{prop_dimorientability}, the map $\Lbord$ is defined on $\Tc(\Mo)$ and takes its values in $\Lambda_{\Mo}$:
\begin{equation*}
\Lbord: \Tc(\Mo) \to \Lambda_{\Mo}.
\end{equation*}
In this case it also induces a map on $\Mc(\Mo)$ and we denote 
\begin{equation*}
    \Mc(\Mo,L)=L_\partial^{-1}(\{L\}),\quad \forall L \in \Lambda_{\Mo}.
\end{equation*}
 The level sets are also cell complexes, and according to Proposition \ref{prop_dimorientability}, for each $L\in \Lambda_{\Mo}$ $\Mc(\Mo,L)$ is locally affine sub-manifolds of codimension $n-1$ in the moduli space $\Mc(\Mo)$. We have:
\begin{equation*}
\dim \Mc(\Mo,L) = 4g-3+n^++n^-+m,
\end{equation*}
if $\Mo$ is of type $(g,n^+,n^-,m)$.

\paragraph{Stratifications and decorations:}
\label{para_decorations_stratification}
As we see in remark \ref{rem_decoration_ribbon} in paragraph \ref{paragraph_def_combi_rib}, a ribbon graph $R$ defines a decoration $\nu_R$ in the sense of paragraph \ref{paragraph_surface_background}. Decorations are used to stratify the Teichmüller space $\Tc(M)$. Let $\Md=(M,\nu)$ be a decoration on $M$. We denote $\Rib^*(\Md)$ the ribbon graphs $R\in \Rib(M)$ with $\nu_R=\nu$ and $\Rib(\Md)$ the ones with $\nu_R\le \nu$ (the meaning of $\le$ is given in paragraph \ref{rem_decoration_collapsing_et_marked_point}).
\begin{equation*}
\Tc(\Md)=\bigsqcup_{R\in \Rib(\Md)}\Met(R),~~~\text{and}~~~\Tcs(\Md)=\bigsqcup_{R\in \Rib^*(\Md)}\Met(R).
\end{equation*}
We can see that if $R'\le R$ is a degeneration of $R$, then $\nu_{R'}\le \nu_R$ is a degeneration of $\nu_R$. Then $\Tc(\Md)$ is a closed sub-complex of $\Tc(M)$, and it's the closure of $\Tcs(\Md)$ in $\Tc(M)$. In general, Lemma \ref{lem_dim_TR} shows that:
\begin{equation*}
\dim(\Tc(\Md))=d(M)+n(\nu).
\end{equation*}
The decoration also induces a stratification of $\Tc(\Mo)$. A decorated, directed surface $\Mdo$ defines a stratum $\Tc(\Mdo)$. These stratifications are compatible with the action of $\Mod(M)$, and then moduli spaces $\Mc(M)$ and $\Mc(\Mo)$ are also stratified; we denote $\Mc(\Md)$ and $\Mc(\Mdo)$ the strata in the moduli spaces.
\begin{rem}
For a connected surface $\Md$ of type $(g,n,\nu)$, it's convenient to use the notations $\Tc_{g,n}(\nu)$ and $\Mc_{g,n}(\nu)$ instead of $\Tc(\Md)$ and $\Mc(\Md)$.
\end{rem}

\paragraph{Measures on the Teichmüller and moduli spaces of metric ribbon graphs:}
\label{paragraph_measurecombteich}

If $R$ is a ribbon graph, it is natural to endow $\Met(R)$ with the Lebesgue measure $\prod_e dm_e$ by using the identification $\Met(R)=\Rpp^{X_1R}$. The group $\Aut(R)$ acts by permutation on $\Met(R)$, then the measure is well defined on the quotient $\Met(R)/\Aut(R)$; we denote $d\mu_R$ the quotient measure (We remark that the action of $\Aut(R)$ preserves the measure but a priori not the orientation). Then each stratum $\Tc(\Md)$ (resp $\Tc(\Mdo)$) in the Teichmüller space carries a measure supported on $\Tcs(\Mdo)$ (resp $\Tcs(\Mdo)$), the set of cells of maximal dimension. This measure descends to a measure $d\mu_{\Md}$ on $\Mc(\Md)$ and $d\mu_{\Mdo}$ on $\Mc(\Mdo)$. In the case of a principal stratum, we simply denote $d\mu_{M}$ and $d\mu_{\Mo}$ the measures on $\Mc(M)$ and $\Mc(\Mo)$.\\

We can also define measures on the level sets of $\Lbord$. For each ribbon graph $R$ and each $L$, we can consider the Lebesgue measure on $\Met(R,L)$ (similarly to paragraph \ref{paragraph_measure_convex} ). The tangent space of $\Met(R,L)$ is the space $K_R$ defined in paragraph \ref{paragraph_orient_boundaries}; to normalize the Lebesgue measure on $K_R$, we use the lattice of integral points: $K_R(\Z)=K_R\cap \Z^{X_1R}$. The space $\Met(R,L)$ is contained in an affine subspace directed by $K_R$. A choice of origin $m\in \Met(R,L)$ allows us to identify it with an open polytope in $K_R$. Different choices of base points produce a change of coordinates given by a translation and then preserve the Lebesgue measure on $K_R$ normalized by $K_R(\Z)$. Then for each $L$ in $\Lbord(\Met(R))$ there is a measure $d\tilde{\mu}_R(L)$ on the space $\Met(R,L)$ normalized by $K_R(\Z)$. The lattice $K_R(\Z)$ is invariant by $\Aut(R)$ (because it is defined using $T\Lbord$), then the Lebesgue measure induces a measure $d\mu_R(L)$ on $\Met(R,L)/\Aut(R)$. Then, for each $\Mdo$ (or $\Md$), these measures define a measure $d\mu_{\Mdo}(L)$ on the level set $\Mc(\Mdo,L)$ supported by the top cells (resp. $d\mu_{\Mdo}(L)$ on $\Mc(\Md,L)$).
\paragraph{Volumes of the moduli space:}
\label{paragraph_volumes_moduli}
Each stratum $\Mc(\Mdo)$ (resp. $\Mc(\Md)$), possesses a natural measure supported on the top cells, but its volume is infinite. Nevertheless, it is possible to consider the measure $dV_{\Mdo}$ on $\Lambda_{\Mo}$ (resp. $dV_{\Md}$ on  $\Lambda_{M}$) defined as the push forward of $d\mu_{\Mdo}$ under the map $\Lbord$.
\begin{lem}
The measures $dV_{\Mdo}$ (resp. $dV_{\Md}$ ) are sigma finite.
\end{lem}
\begin{proof}
We have
\begin{equation*}
 dV_{\Mdo}= {\Lbord}_{*}~d\mu_{\Mdo}=\sum_{\Ro\in\rib^*(\Mdo)}dV_{\Ro}.
\end{equation*}
Where the sum is finite and for $\Ro$ we denote $dV_{\Ro}={\Lbord}_{*}d\mu_{\Ro}$. Each edge of $\Ro$ is contained in a boundary, then we have
\begin{equation}
m_e\le \|\Lbord(m)\|_\infty,~~~\forall e\in X_1R.
 \end{equation}
 Then we can see that the measures $dV_{\Mdo}$ are sigma-finite because the preimage of a bounded set by $\Lbord$ is also bounded (this is also true for $dV_{\Md}$).
\end{proof}
These measures are characterized by the relation
\begin{equation*}
\int_{\Lambda_{\Mo}} f(L)~dV_{\Mdo} = \int_{\Mc(\Mdo)} f( \Lbord(S))~d\mu_{\Mdo},
\end{equation*}
for $f$ a measurable and positive function on $\Lambda_{\Mo}$. As we see in the last paragraph, for each $L\in \Lambda_{\Mo}$, the level set $\Mc(\Mdo,L)$ is equipped with its Lebesgue measure $d\mu_{\Mdo}(L)$. In what follows, we consider the volume of $\M^{comb,*}(\Mdo,L)$ of the subset of generic directed ribbon graphs. It makes sense to compute the volume of $\Met(\Ro,L)$ by using the last inequality, $\Met(\Ro,L)$ is a relatively compact convex polytope, and then 
\begin{equation*}
V_{\Ro}(L)  =\int_{\Met(\Ro,L)/\Mod(M)}d\mu_{\Ro}(L)=\frac{1}{\#\Aut(\Ro)} \int_{\Met(\Ro,L)}d\mu_{\Ro}
\end{equation*}
is finite for all $\Ro$ and $L$ (the last equality is true because a $\Aut(\Ro)$ acts freely on $\Met(\Ro,L)$). Then the total volume of $\M^{comb,*}(\Mdo,L)$ is equal to the finite sum 
\begin{equation*}
 V_{\Mdo}(L) =\int_{\M^{comb,*}(\Mdo,L)}d\mu_{\Mdo}(L)=\sum_{\Ro\in \rib^*(\Mdo)} V_{\Ro}(L),
\end{equation*}
then, it's also finite. In the case of a connected surface, we can also use the notation $V_{g,n^+,n^-}^\nu(L^+|L^-)$. The three objects $dV_{\Mdo}$, $V_{\Mdo}$, and $d\sigma_{\Mo}$\footnote{the natural Lebesgue measure on $\Lambda_{\Mo}$} are related by the following proposition.

\begin{prop}
\label{prop_decomp_volumes_L}
The measure $dV_{\Mdo}$ is absolutely continuous with respect to $d\sigma_{\Mo}$. On $\Lambda_{\Mo}$, we have the relation
\begin{equation*}
\frac{dV_{\Mdo}}{d\sigma_{\Mo}}=V_{\Mdo}(L),~~~a.s.
\end{equation*}
Then for a measurable function
\begin{equation*}
\int_{\Mc(\Mdo)} f(L_\partial(S))d\mu_{\Mdo}= \int_{\Lambda_{\Mo}} f(L) V_{\Mdo}(L)d\sigma_{\Mo}.
\end{equation*}
\end{prop}
Proposition \ref{prop_decomp_volumes_L} is a consequence of the following lemma and the results of paragraph \ref{paragraph_measure_convex}.
\begin{lem}
\label{lem_exact_sequence_TL}
For all directed ribbon graphs $\Ro$ embedded in $\Mo$, the map $T\Lbord$ fits into the following exact sequence:
\begin{equation*}
\{0\}\to K_R(\Z)\to T_R(\Z) \overset{T\Lbord}{\to} T_{\Mo}(\Z)\to \{0\}.
\end{equation*}
\end{lem}
When the graph is non-dirigible, we have the following lemma; in this case, there is an extra factor $\frac{1}{2}$.
\begin{prop}
If $R$ is non-dirigible, embedded in $M$ and $d\sigma_M$ is the Lebesgue measure on $\Lambda_M$. We have:
\begin{equation*}
\frac{d V_R}{d\sigma_M}=\frac{V_R(L)}{2}.
\end{equation*}
\end{prop}
\begin{proof}
This is due to the fact that $T\Lbord(T_R(\Z))$ is equal to the subgroup of $\Z^{\partial M}$ of vectors $(l_\beta)$ with $\sum_\beta l_\beta \in 2\Z$. But it is not straightforward to prove that the image is exactly this lattice.
\end{proof}

\newpage

\section{Curves on ribbon graphs and surgeries:}
\subsection{Combinatorial curves and surgeries}
\label{subsection_surgeries}
\paragraph{Combinatorial representation of an homotopy class of curves:}
\label{para_combi_curve}

Let $M\in \bord$, and $\mathcal{C}(M)$ be the set of homotopy classes of curves in $M$. Let $R$ be a ribbon graph embedded in $M$, a combinatorial curve is a closed walk modulo the action of $\Z\rtimes \{\pm 1\}$ by shifting the sequence of half edges and reversing the order. Nevertheless, there is still several representations of an isotopy class of curves  $\gamma \in \mathcal{C}(M)$ as a combinatorial curve on $R$. To get around this problem, we say that $\gamma_1\ge \gamma_2$ if we can obtain $\gamma_2$ by removing subsequences of the form $(e,s_1e)$ in $\gamma_1$. We write $\gamma_1\sim\gamma_2$ iff there is $\gamma$ such that $\gamma_i\ge \gamma ,i=1,2$. We can see that this relation defines an equivalence relation; it corresponds to the combinatorial notion of homotopy, and we denote $\mathcal{C}^{comb}(R)$ the equivalence classes of combinatorial curves. We say that an element is minimal if it is minimal for the partial order relation in its equivalence class.

\begin{lem}
\label{lem_curve_comb}
    Let $M\in \bord$ and $R\in \Rib(M)$:
    \begin{itemize}
        \item Each class in $\mathcal{C}^{comb}(R)$ admits a unique minimal representation.
        \item Each homotopy class can be represented as a combinatorial curve, we have $\mathcal{C}^{comb}(R)\simeq \mathcal{C}(M)$.
    \end{itemize}
\end{lem}

\begin{rem}
\label{rem_curve_marked_point}
If $M\in \bordb$ and $R\in \Rib(M)$, marked points in $M$ correspond to marked vertices in $R$. In this case, in the definition of $\sim$ we allow a subsequence of the form $(e_i,s_1e_i)$ iff the target of $e_i$ is a marked vertex. Then, Lemma \ref{lem_curve_comb} remains true.
\end{rem}

\begin{figure}
    \centering
    \includegraphics[width=0.5\linewidth]{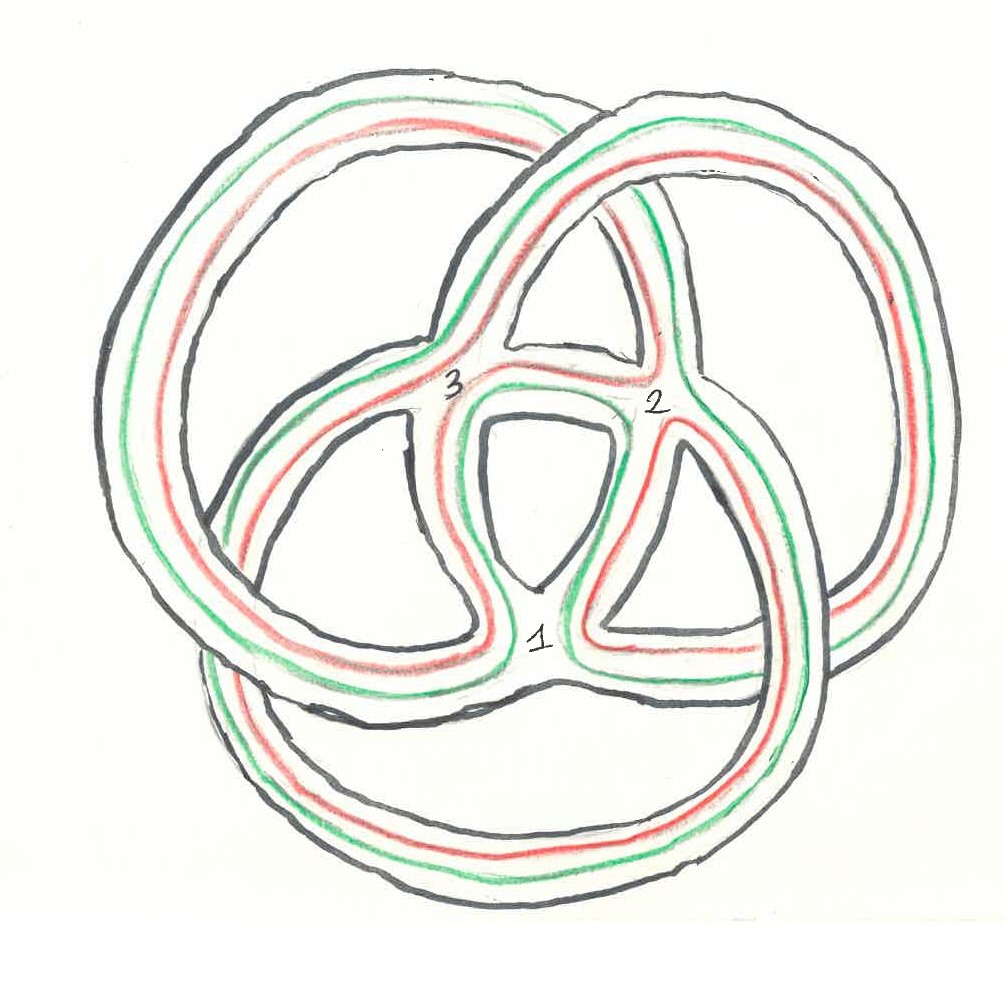}
    \caption{Curve on a ribbon graph.}
    \label{figure_curve_ribbon}
\end{figure}

\paragraph{Length of a curve on a metric ribbon graph:}
\label{lenght_curve_rib}
Let $S=(R,m)$ be a metric ribbon graph, the length $l_\gamma(S)$ of a combinatorial curve $\gamma$ is defined naively by summing the metric of each edge:
\begin{equation*}
l_\gamma(S)=\sum_{i=1}^{r}m_{[e_i]}(S).
\end{equation*}
By using Lemma \ref{lem_curve_comb}, we can define the combinatorial length of a curve $\gamma\in \mathcal{C}(M)$ as the length of its unique minimal combinatorial representation. It is the smallest length among all possible representations of $\gamma$. There is also a topological way to compute the length of a curve. As we see in Lemma \ref{lem_ribbon_multiarc_1}, each edge $e\in X_1R$ is associated with an arc $e^*$ that relies on the two boundaries $[e]_2,[s_1e]_2$. Let $ y_e(\gamma)$ be the intersection pairing of a curve $\gamma\in \Sit(M)$ with the arc $e^*$
\begin{equation}
\label{y_coord}
y_e(\gamma)=\iota(\gamma,e^*).
\end{equation}
The length of $\gamma$ is then given by the intersection pairing with $A(S)$ (see paragraph \ref{paragraph_ribbon_filling_fol}):
\begin{equation}
\label{formula_lenght_curve_2}
l_\gamma(S)=\iota(A(S),\gamma)=\sum_{e\in X_1R} m_e(S)y_e(\gamma).
\end{equation}
The numbers $(y_e(\lambda))_{e\in X_1R}$ are well defined also for a foliation $\lambda\in \MFt_0(M)$; and Formula \ref{formula_lenght_curve_2} makes sense for foliations in $\MFt_0(M)$. If $\lambda\in \MF_0(M)$, the length of $\lambda$ defines a continuous function on $l_\lambda: \Tc(M)\to \Rp$ which is linear on each cell.
\paragraph{Surgery along a curve:}
\label{paragraph_surgery}
Let $R$ be a ribbon graph embedded in $M$, and $\Gamma\in\MS(M)$ a multi-curve. We recall that $M_\Gamma$ is the surgery of $M$ along $\Gamma$. In this part, if $R\in \Rib(M)$, we show that it is always possible to cut $R$ along $\Gamma$ and obtain a ribbon graph $R_\Gamma$ on $M_\Gamma$. We define the ribbon graph $R_\Gamma$ using the multi-arc $A(R)$. We proceed in several steps. 
\begin{lem}
\label{lem_cut_arc}
If $\Gamma$ is a multi-curve, there is a map:
\begin{equation*}
\cut_\Gamma: \MA_\R(M)\longrightarrow \MA_\R(M_\Gamma),
\end{equation*}
which is linear on each cell. Moreover, $\cut_\Gamma$ is the unique map such that:
\begin{equation*}
\iota(\cut_\Gamma(A),\gamma) = \iota(A,\gamma),~~~\forall \gamma \in \Sit(M_\Gamma).
\end{equation*}
\end{lem}

\begin{proof}
If $A$ is a multi-arc and $\Gamma$ is a multi-curve, then, up to homotopy, we can assume that $A,\Gamma$ are in minimal position. This means that the intersections are transverse, and the multi-arc and the multi-curve minimize the number of intersection points up to homotopy. The fact that they intersect transversely makes it possible to cut the surface and the arcs along $\Gamma$. The result is a family $\tilde{A}_\Gamma$ of arcs on each connected component of $M_\Gamma$. If $A$ and $\Gamma$ are in minimal position, then by using the Bigon criterion (remark \ref{rem_bigon_criterion_2}), we can see that the arcs of $\tilde{A}_\Gamma$ are non trivial. There are possibly families of homotopic arcs in $\tilde{A}_\Gamma$; we identify the homotopic arcs and obtain a smaller family of arcs $A_\Gamma$ on $M_\Gamma$, which is now a multi-arc. If $m\in \Met(A)$ is a weight on $A$, it induces a weight on each arc of $\tilde{A}_\Gamma$. We sum the weights of the arcs that are identified and then define weights on the arcs of $A_\Gamma$. This construction defines a map: 
\begin{equation*}
   \cut_\Gamma:\Met(A)\longrightarrow \Met(A_\Gamma),
\end{equation*}
which is linear. Then we can see that this map preserves the intersection pairing; if $\gamma\in \Si(M_\Gamma)$, there is $c$ such that $\gamma\in \Si(M_\Gamma(c))$. Let $A_\Gamma(c)$ be the arcs of $\cut_\Gamma(A)$ that are in $M_\Gamma$. Up to homotopy, we can assume that $A$ and $\Gamma\sqcup\{\gamma\}$ are in minimal position. Then by remark \ref{rem_bigon_criterion_2}, we can see that $A_\Gamma(c)$ and $\gamma$ are also in minimal position. Then we have
\begin{equation*}
\iota(A,\gamma)=\iota(A_\Gamma(c),\gamma)=\iota(\cut_\Gamma(A),\gamma).
\end{equation*}
When $\gamma\in\Gamma\sqcup \partial M$, the situation is similar. The uniqueness of the construction follows from Proposition \ref{prop_injectivite_iota_arc}.
\end{proof}

Using Lemma \ref{lem_cut_arc} and the relation between filling multi-arcs and ribbon graphs, we can immediately deduce the Corollary \ref{cor_cut_rib}:

\begin{cor}
\label{cor_cut_rib}
If $M\in \bordb$ and $A\in \MA^0(M)$ is a filling multi-arc, then for all $\Gamma\in \MS(M)$, the multi-arc $A_\Gamma$ is also filling on $M_\Gamma$. Then, for all ribbon graph $R$, there is a ribbon graph $R_\Gamma$ obtained after cutting $R$ along $\Gamma$, and this induces a linear map:
\begin{equation*}
\cut_\Gamma: \Met(R)\longrightarrow \Met(R_\Gamma).
\end{equation*}
It defines a continuous map $\cut_\Gamma: \Tc(M)\longrightarrow \Tc(M_\Gamma)$; moreover, this is the unique map that satisfies:
\begin{equation*}
l_\gamma(\cut_\Gamma(S))= l_\gamma(S),~~~\forall \gamma \in \Sit(M_\Gamma).
\end{equation*}
\end{cor}
\begin{figure}[h!]
    \centering
    \includegraphics[width=0.4\linewidth]{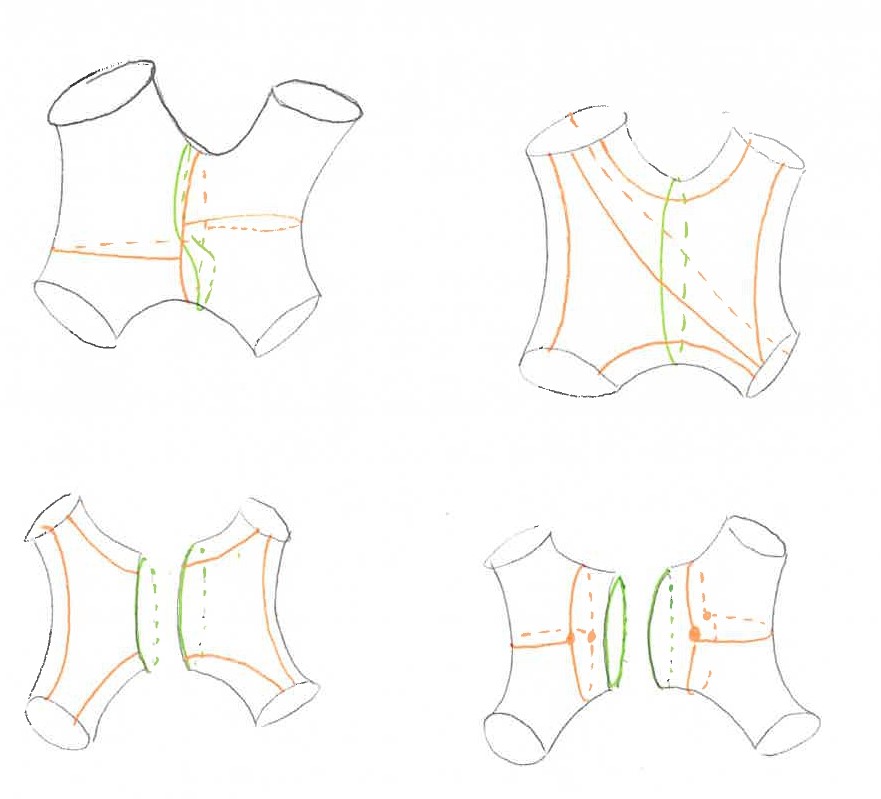}
    \caption{Cutting a ribbon graph using the dual multi-arc.}
    \label{figure_decoupage_ribbon}
\end{figure}
\paragraph{Multi-curves on a directed ribbon graph:}
\label{paragraph_oriented_multicurve}

In the case of directed ribbon graphs, it is natural to consider directed multi-curves. We say that an equivalence class of combinatorial curves is dirigible if there is a representation $(e_i)_i$ along with $\epsilon(e_i)=1, \forall i$. We can see that such representation is minimal; moreover, according to figure \ref{figure_surgeries_oriented} the direction of each curve induces a direction $\Gao$ on $\Gamma$ (see remark \ref{rem_orient_curves}). Proposition \ref{prop_curve_dir_ribbon} ensures that surgeries along dirigible curves preserve the dirigibility.  

\begin{prop}
\label{prop_curve_dir_ribbon}
\begin{itemize}
    \item Let $\Ro$ be a directed ribbon graph, if $\Gamma$ is a dirigible multi-curve, then the direction of the curves induces a non-degenerate direction $\Gao$ on $\Gamma$. Moreover, the ribbon graph $R_{\Gamma}$ possesses a natural direction $\epsilon_{\Gamma,\Ro}$, which induces the same direction as $\Gao$ on $\Gamma$.
    \item If $R$ is a metric ribbon graph such that $R_\Gamma$ is directed and this direction induces a direction on $\Gamma$, Then the graph $R$ is also directed, and this direction is compatible with the direction on $\Gamma$.
\end{itemize}
\end{prop}

\begin{rem}
    For a more general directed multi-arc $A^\circ$, a curve is dirigible if it admits a representation that crosses arcs in $A^\circ$ in a direct way. For each directed non-degenerate multi-curve $\Gao\in \MS(\Mo)$, there is a subset of weighted multi-arcs $\MA_{\Gao}(\Mo)\subset \MA(M)$ such that $\Gamma$ is dirigible on each element of $\MA_{\Gao}(\Mo)$ and the direction induced on $\Gamma$ is $\Gao$. As a corollary of the last results, we obtain the following:.

\begin{cor}
\label{lem_cut_arc_orientable}
Let $\Mo$ be a directed surface; for a directed, non-degenerate multi-curve $\Gao$, the restriction of $\cut_\Gamma$ induces a map:
\begin{equation*}
\cut_{\Gamma}~:~ \MA_{\Gao}(\Mo)\longrightarrow \MA(\Mo_{\Gao}).
\end{equation*}
\end{cor}
\end{rem}

\begin{figure}[h!]
    \centering
    \includegraphics[width=0.5\linewidth]{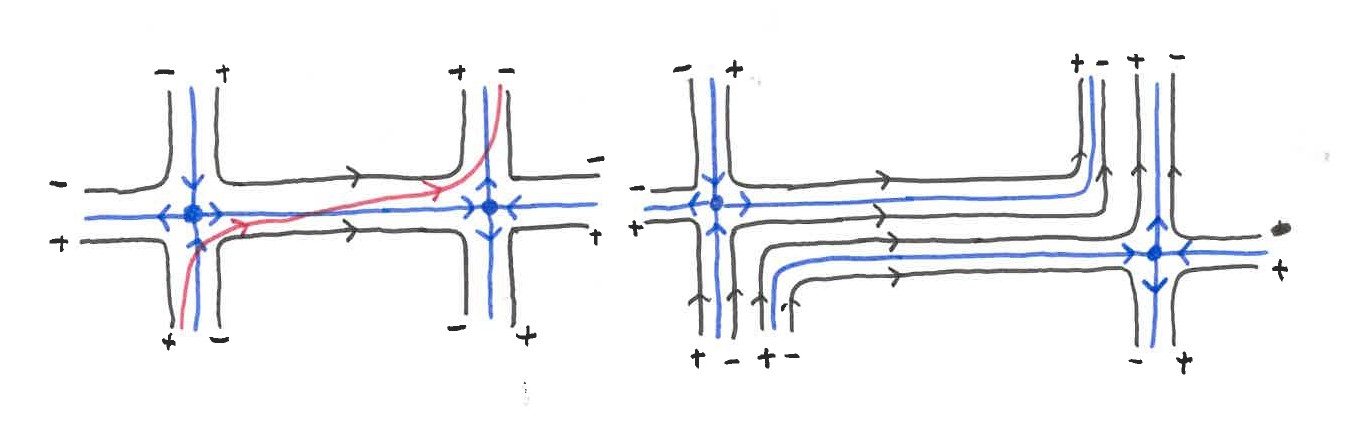}
    \caption{Surgeries on directed ribbon graphs.}
    \label{figure_surgeries_oriented}
\end{figure}

\subsection{Admissible curves}
\label{subsection_admissible}
\paragraph{Definition:}
\label{paragraph_admissible_curves}
In this part, we introduce admissible curves and foliations. We give three possible equivalent definitions of admissible curves which have their own advantages. The first use the combinatorial representation, the second the relation with surgeries and the third is related to Jenkin-Strebel differentials. Let $R$ be a ribbon graph. We introduce in paragraph \ref{paragraph_def_oriented_RG} the two "zig-zag" permutations (Formula \ref{formula_s+-} ).
\begin{Def}
\label{def_admissible_fol}
    Let $R$ a ribbon graph, an essential curve $\gamma \in \Si(M_R)$ is admissible iff it satisfies one of the following equivalent cconditions 
    \begin{enumerate}
        \item It admits a combinatorial representation $(e_i)_i$ of the form
        \begin{equation*}
        e_{i+1}= s_{\pm} e_i.
        \end{equation*}
        \item It does not split a vertex of the graph $\nu_{R_\gamma}=\nu_R$.
        \item The quadratic differential $q_\gamma(S)$ has a zero of order $d-2$ for each vertex of degree $d$, $\nu_{q_\gamma(S)}=\nu_{S}$.
    \end{enumerate}
\end{Def}

\begin{figure}
    \centering
    \includegraphics[width=0.5\linewidth]{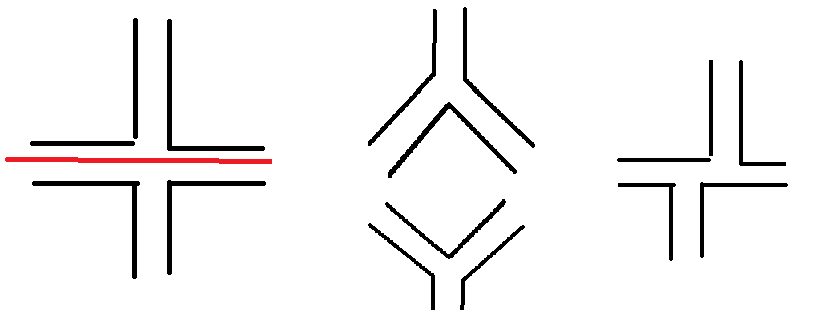}
    \caption{A non admissible curve.}
    \label{figure_non_admissible}
\end{figure}
The point $2)$ also make sense for multicurves and the point $3)$ foliations in $\MF(M)$, and this does not depend on the metric on $R$. We denote $\Si(R),\MS(R),\MS_\Z(R),\MF_0(R),...$ the various spaces of admissible curves, multi-curves and foliations. 

\begin{proof}
We start by explaining why $1) \Leftrightarrow 2)$. Let $\gamma\in \Si(M)$ already see that $\gamma_R^{comb}$ defines a curve in minimal position with $A=A(R)$. Then we can see that surgeries along the curve won't change the decoration $\nu_{A}=\nu_{A_\Gamma}$ iff the curve is admissible.\\
For $2) \Leftrightarrow 3)$, let $\gamma \in \Si(M)$, for each $S\in \Met(R)$ we prove that $\nu_{q_\gamma(S)}=\nu_{S_\Gamma}$. Using Corollary \ref{cor_cut_rib}, we can cut $S$ along $\gamma$; the result is a family of metric ribbon graphs $S_\gamma=(S_\gamma(c))_c$ on $M_\gamma$. For each connected component $c$ of $M_\gamma$, there is a Jenkin-Strebel differential $q_c=q_{\lambda_0}(S_\gamma(c))$ on $M_\gamma(c)$ with $\nu_{q_c}=\nu_{S_\gamma(c)}$. Every non-singular horizontal trajectory of $q_c$ is periodic and surrounds a pole. Let  $(\gamma^1,\gamma^2)$ be the two boundaries of $M_\gamma$ that correspond to $\gamma$. It is possible to glue a horizontal cylinder of height $1$ to the two boundaries $\gamma^1,\gamma^2$. The result is a Jenkin-Strebel differential $q$ on $M$ such that
\begin{equation*}
\text{Re} (q)=\lambda(S)~~~\text{and}~~~\text{Im} (q) =\lambda(\Gamma).
\end{equation*}
By uniqueness in Theorem \ref{thm_hubbard_masur_poles} $q=q_\Gamma(S)$, and by construction, we have:
\begin{equation*}
\nu_{S_\Gamma}=\nu_q=\nu_{q_\Gamma(S)}.
\end{equation*}
Then the points $2)$ and $3)$ of the definition coincide.
\end{proof}

\begin{rem}[Particular cases]
\label{rem_part_cases_admissible_curve}
\begin{enumerate}
    \item We have $\Si(R)=\Si(M_R)$ iff the ribbon graph $R$ is generic\footnote{it has only trivalent vertices and univalent vertices at the marked points}.
    \item If $\Ro$ is directed, all the admissible curves are dirigible. And the converse is true iff $\Ro$ is generic\footnote{it has only quadrivalent vertices and bivalent vertices at the marked points}
\end{enumerate}
\end{rem}

\paragraph{Coordinates for $\MF(R)$ and $\MF_0(R)$:}
\label{paragraph_coord_x}

Let $M\in \bord$ and $R$ be an embedded ribbon graph in $M$. We define:
\begin{equation*}
\Qq(R)=\{q\in \QT(M)~|~ \text{Re}(q)\in \Met(R)~\text{and}~\nu_q=\nu_R\}~~~\text{and}~~~\Qq_0(R)=\Qq(R)\cap \QT_0(M).
\end{equation*}
From Definition \ref{def_admissible_fol}, they correspond to subsets of admissible foliations on metric ribbon graphs in $\Met(R)$. Proposition \ref{prop_coord_x} gives period coordinates for admissible foliations.
\begin{prop}
\label{prop_coord_x}
For all $R$ ribbon graph, we have:
\begin{equation*}
\Qq(R)=\Met(R)\times \MF(R)~~~\text{and}~~~ \Qq_0(R)=\Met(R)\times \MF_0(R).
\end{equation*}
There is bijections that preserve integral points:
\begin{equation*}
\Qq(R)\longrightarrow \Met(R)\times T_R~~~\text{and}~~~\Qq_0(R)\longrightarrow \Met(R)\times K_R.
\end{equation*}
In particular, they induce bijections:
\begin{equation*}
\MF(R)\simeq T_R,~~~\text{and}~~~\MF_0(R)\simeq K_R.
\end{equation*}
\end{prop}

\begin{rem}
 We recall that $T_R$ is the tangent space of $\Met(R)$ and $K_R$ is the subspace of tangent vectors that preserve the boundary lengths (see paragraph \ref{paragraph_orient_boundaries}). Using paragraph \ref{paragraph_coho_ribbongraph} and remark \ref{rem_KR_H(R,X0R)}, we can identify these spaces with the cohomologies $H^{1}(R,X_0R,X_2R)$ and $H^{1}(R,X_0R)$.
\end{rem}

\begin{proof}
We first construct the map, which is the period map of the horizontal foliation along the edges of the embedded graph $R$
\begin{equation*}
\Qq(R)\longrightarrow \Met(R)\times T_R.
\end{equation*}
We use the zippered rectangle construction (see Proposition \ref{prop_zippered_fol_general}). For all $q$ and all $e\in XR$, there is a maximal embedded infinite rectangle $\mathbf{R}^{\bullet}_e\to M^\bullet$ such that $\text{Re} (q)$ is locally given by $|dx|$ on $\mathbf{R}^{\bullet}_e$. Moreover, we assume that the direction of $[0,1]$ corresponds with the direction of $e$; $q$ has no singularities on the interior of $\mathbf{R}^{\bullet}_e$, but the maximality of $\mathbf{R}^{\bullet}_e$ implies that there is at least one singularity on each boundary component. As $\nu_q=\nu_R$, then $q$ has no vertical saddle connections, then there is only one singularity on each boundary of $\mathbf{R}^{\bullet}_e$. It is possible to choose a square root's $\alpha_e$ of $q$ on $\mathbf{R}^{\bullet}_e$ such that $\text{Re}( \alpha_e )=dx$. After this choice, we denote $x_e^-$ the singularity on the left boundary and $x_e^+$ the one on the right. Let $I_e\subset \mathbf{R}^{\bullet}_e$ be the horizontal maximal open interval directed according to $e$ such that the left extremity is $x_e^-$. There is an isometry $I_e=[0,m_e(S_q)]$, and we can consider for all $x\in I_e$ the vertical flow $v_t$ such that $\im (\alpha) (\partial v_t)=1$. It defines a map:
\begin{eqnarray*}
\phi_e : I_e\times \R &\longrightarrow& \mathbf{R}^{\bullet}_e\\
(x,y)&\longrightarrow& v_y(x).
\end{eqnarray*}
If $(x,y)$ are complex coordinates, then the pullback of $\alpha$ under this map is equal to $dz$. In the local coordinates given by $\phi_e$, we can define
\begin{equation*}
 x_e^+=m_e(\alpha) + ix_e(\alpha).
\end{equation*}
This is the relative period of $\sqrt{q}$ along the edge $e$. There are no sign ambiguities because we assume that the real part of the period is positive, and then $x_{s_1e}=x_{e}$. Then this defines an element of the tangent space $T_R$.\\

By the zippered rectangle construction, the data $(m,x)$ are enough to recover $q$. We simply glue the rectangles $\mathbf{R}_e^{\bullet}=[0,m_e]\times \R$ by performing a shear of parameter $x_e$ on the right boundary. There is no constraint on $(m,x)$ to perform the construction. We obtain in this way a Riemann surface with an Abelian differential given by $dz$ on each rectangle. The surface is $\tilde{M}_R^{\bullet}$ the directed cover of $R$. Then the square of the Abelian differential defines a quadratic differential $q_S(x)$ on $M^{\bullet}$, which is locally given by $(dz)^2$ on each rectangle. The two constructions are the inverse of each other, and then there is a bijection:
\begin{equation*}
 \Qq(R)\longrightarrow \Met(R)\times T_R.
\end{equation*}
The imaginary part of the quadratic differential $q_S(x)$ defines a foliation $\lambda_R(x)$, which does not depend on $m$, so the space $\MF(S)$ depends only on $R$.
Moreover
\begin{equation*}
l_\beta(\lambda) = \sum_{e\in X_1R} y_e(\beta)x_e(\lambda)= dl_\beta(\sum_ex_e(\lambda)\partial_e).
\end{equation*}
Then the elements of $\MF_0(S)$ correspond exactly to the vectors in $K_R$.
\end{proof}

\begin{figure}
\centering
\includegraphics[width=8cm]{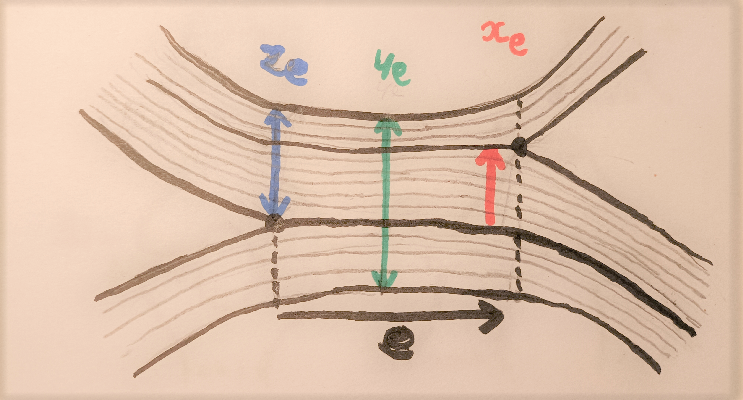}
\caption{coordinates $x,y,z$}
\label{coordfig}
\end{figure}

\paragraph{Case of directed ribbon graphs:}
If $\Ro$ is a directed ribbon graph on $\Mo$, we consider in a similar way the spaces of Abelian differentials $\Hc(\Ro),\Hc_0(\Ro)$. As a corollary of Proposition \ref{prop_coord_x}, we obtain the following result:
\begin{cor}
\label{prop_x_coord_dir}
All admissible foliations on $\Ro$ are dirigible,
\begin{equation*}
\MF(R)\subset\MF(\Mo).
\end{equation*}
The quadratic differential $q_S(x)$ is the square of an Abelian differential $\alpha_S(x)$, and we have bijections:
\begin{equation*}
\Hc(\Ro)\simeq \Met(R)\times T_R,~~~\text{and}~~~\Hc_0(\Ro)\simeq \Met(R)\times K_R.
\end{equation*}
\end{cor}

\begin{rem}[Relation with cohomology]
If $\Ro$ is embedded in $M$ we see that we have isomorphisms
\begin{equation*}
K_R=H^1(M^{cap},X_0R,\R)~~~~T_R=H^1(M^{\bullet},X_0R,\R).
\end{equation*}
Then the map corresponds to the period coordinates.
\end{rem}

\paragraph{Coordinates for integral multi-curves:}

An important consequence of Proposition \ref{prop_coord_x} is the following proposition that characterizes the set of admissible integral multi-curves:

\begin{prop}
\label{prop_x_admissible_curves}
The set of admissible integral multi-curves $\MF_\Z(R)$ is identified with the set of non-zero integral points $K_R(\Z)\backslash\{0\}$ under the period coordinates $x$.
\end{prop}

\begin{proof}
This is mainly straightforward by using the construction. For an element in $K_R(\Z)\backslash\{0\}$, we can see that the horizontal foliation is necessarily periodic with cylinders of integral height. Conversely, if we take such foliation, we can see that the periods are integers.
\end{proof}


\paragraph{Irreducible ribbon graphs and Fenchel-Nielsen decompositions:}
\label{irreducible_graph}
Irreducible ribbon graphs generalize pairs of pants outside the generic case (vertices of order one or three). They form an interesting class of ribbon graphs, but we won't use them so much here.

\begin{Def}
\label{Def_irreducible}
Let $R$ a connected ribbon graph, we say that it is irreducible iff it satisfies one of the two equivalent conditions:
\begin{enumerate}
    \item There is no non trivial admissible curves $\MS_\Z(R)=\emptyset$.
    \item $H^1(R,X_0R)=\{0\}$.
\end{enumerate}
\end{Def}

We call these graphs irreducible because we cannot reduce their topology by admissible surgeries. In some sense, they are minimal objects in the family of decorated surfaces. Moreover, as we have $K_R=H^1(R,X_0R)=\{0\}$, the strata of irreducible ribbon graphs are isolated points in the moduli space $\Mc_{G,n}(L)$.
\begin{proof}
The proof uses Proposition \ref{prop_coord_x}. We have $\MF_\Z(R)=K_R(\Z)\backslash\{0\}$ and then $\MF(R)$ is empty iff $K_R(\Z)$ is zero and then iff $K_R=\{0\}$.
\end{proof}
By computing the dimension of $H^1(R,X_0R)$ we can derive the following equivalence:

\begin{prop}
A ribbon graph is irreducible iff it is of genus zero and it satisfies one of the two following conditions:
\begin{itemize}
    \item It has only two vertices, and they are of odd degree.
    \item It is dirigible and has only one vertex.
\end{itemize}
\end{prop}

To conclude this paragraph, irreducible graphs can be used to obtain Fenchel-Nielsen decompositions of a non-generic graph. But outside the generic case, there is no hope that these coordinates are global coordinates on a given stratum of metric ribbon graphs.

\begin{cor}
\label{Fenchel_Nielsen}
For each ribbon graph $R$, there exists an admissible multi-curve $\Gamma$ such that : all the components of $R_\Gamma$ are irreducible. Such a multi-curve can be called a maximal admissible multi-curve.
\end{cor}

\begin{proof}
    We can proceed by induction, either the graph is irreducible and $K_R=\{0\}$ or $K_R\neq\{0\}$; in the secound case, using $\MS_\Z(R)\simeq K_R(\Z)\backslash \{0\}\neq \emptyset$ we can find an admissible multi-curve $\Gamma$ on $R$, we can also check that $\text{dim} K_{R_\Gamma}< \text{dim} K_{R}$ then by induction on $\text{dim} K_{R}$ we obtain the corollary.
\end{proof}

\paragraph{Dual coordinates on $\widetilde{\MF}_0(R)$ and $\MF_0(R)$:}
\label{zippered_dual}
For each $R$, we define in paragraph \ref{paragraph_zeros_residu} the space $\MFt_0(M)$. An element of this space can be represented by a foliation in $\MF_0(M)$ marked by a choice of a periodic trajectory around each pole. Let $R$ be an embedded ribbon graph. It is possible to define the subspace $\MFt_0(R)$ of foliations in $\MFt_0(M)$ that are admissible on $R$. There is a map $\MFt_0(R)\to \MF_0(R)$,  quantities $(y_e)_e$ of Equation \ref{y_coord} are well defined for foliations in $\widetilde{\MF}_0(R)$. But in general, the map $y: \widetilde{\MF}_0(R)\to \Rp^{X_1R}$ is neither injective nor surjective. We define other parameters in the following way: For each $e\in XR$, let $\gamma_e$ be the undirected arc that joins $[e]_2$ and $[e]_0$, and we denote
\begin{equation*}
|z_e| = \iota(\lambda,\gamma_e)\in \Rpp.
\end{equation*}
Which is the distance between the singularity $[e]_0$ and the boundary curve $[e]_2$. They satisfy the relation (see figure \ref{coordfig})
\begin{equation*}
 x_e(\lambda)=|z_{e}(\lambda)|-|z_{s_2e}(\lambda)|,~~~\text{and}~~~y_e(\lambda)=|z_{s_0^{-1}e}(\lambda)|+|z_{e}(\lambda)|.
\end{equation*}
And the vector $(|z_e|)$ also satisfies the constraints:
\begin{equation}
\label{constraints_z}
|z_{s_2s_1e}(\lambda)|+|z_{e}(\lambda)|=|z_{s_2e}(\lambda)|+|z_{s_1e}(\lambda)|,~~~\forall e\in XR.
\end{equation}
Let $W_R^+$ be the set of $(|z_e|)\in \Rpp^{XR}$ that satisfy the last relations, and $W_R=TW_R^+$ the tangent space.\\

When $\Ro$ is directed according to Proposition \ref{prop_x_coord_dir}, all the admissible foliations are also directed. It is possible to consider the cycles $[\gamma_e]$ directed from $[e]_0$ to $[e]_2$; they satisfy the relations
\begin{equation*}
[\gamma_e]-[\gamma_{s_0^{-1}e}]= e^*~~~\text{and}~~~[\gamma_e]-[\gamma_{s_2e}] = e. \end{equation*}
The family $([\gamma_e])_{e\in XR}$ generates the relative homology, but it is not free; the cycles satisfy the boundary condition
\begin{equation*}
[\gamma_e]-[\gamma_{s_0^{-1}e}]+[\gamma_{s_1e}]-[\gamma_{s_2e}]=0.
\end{equation*}
Then, in the directed case, $W_R$ is identified with the cohomology
\begin{equation*}
W_R =H^1(M_R^{cap},X_0R\cup X_2R,\R).
\end{equation*}
When the graph is not directed, the space $W_R$ is given by the anti-invariant cohomology of the two covers $\Rt$
\begin{equation*}
W_R =H^1(\tilde{M}_R^{cap},X_0\tilde{R}\cup X_2\Rt,\R)^-.
\end{equation*}

\begin{prop}
\label{prop_coord_z}
The map:
\begin{eqnarray*}
|z| ~:~ \MFt_0(R)&\longrightarrow& W_R^+\\
 \lambda~~&\longrightarrow& |z_\lambda|
\end{eqnarray*}
is a bijection and identifies $W_R^+(\Z)$ with $\MFt_\Z(M)$. Moreover, $R$ is dirigible, a choice of direction allow to lift this map to:
\begin{equation*}
z: \MFt_0(R) \longrightarrow H^1(M_R^{cap},X_0R\cup X_2R,\R).
\end{equation*}
Such that
\begin{equation*}
z_e(\lambda)=\epsilon(e)|z_e(\lambda)|=\int_{[\gamma_e]}z_\lambda.
\end{equation*}
\end{prop}

\begin{rem}{Relation with train tracks:}

\begin{lem}
\label{WRtraintrack}
There is a train track $\tau_R$ such that $W_R^+$ is the set of weights on $\tau_R$.
\end{lem}
We construct the train track in the following way: for each directed edge in $XR$, we associate a vertex $(v_e)$. The edges of $\tau_R$ are of two types:
\begin{itemize}
\item There is an edge $s_e$ for all $e\in X_1R$; $s_e$ joins the two vertices labeled by the two directions of $e$.
\item There is an edge $s'_e$ for all $e\in XR$; it joins $v_{[e]_1}$ and $v_{[s_0e]_1}$.
\end{itemize}
Then $W_R^+$ is the set of positive weights on the train track $\tau_R$. The train track is always non-degenerate; if $W_R$ is the same space with real weights, then $\dim W_R=\dim W_R^+=\#X_1R+1$.
\end{rem}

\subsection{Deformations of metric ribbon graphs, twists and horocyclic flows}
\label{subsection_twist}
\paragraph{Twist flow along a curve:}
We use curves and foliations to study the deformations of metric ribbon graphs. As in \cite{andersen2020kontsevich}, we consider twist flows along admissible curves and relate these flows to the horocyclic flow on the space of quadratic differentials with poles. The first intuitive definition of the twist flow along a curve $\gamma$ is the same as the definition of the twist flow along a geodesic on a hyperbolic Riemann surface. Nevertheless, in the case of metric ribbon graphs, the gluing are not always defined, then the twist flow is not well defined on the combinatorial Teichmüller space $\Tc(M)$ but on a bigger space $\MF_\gamma(M)$ of foliation in $\MF_0(M)$ transverse to $\gamma$. To define the twist flow, it is somehow easier to use the horocyclic flow for quadratic differential.  If $q\in \QT(M)$, locally we can find local coordinates $z=x+iy$ such that $q=(dz)^2$. Then $U_tq$ is defined locally by $(dx+tdy+idy)^2$; $U_tq$ then defines a new quadratic differential on a different Riemann surface (see \cite{forni2013introduction} for references ). We can see that the flow is complete; for all $t\in \R$, it induces a homeomorphism: 
\begin{equation*}
    U_t~:~\QT(M)\longrightarrow\QT(M).
\end{equation*}
 The action on residues is of the form $x+iy\longrightarrow x+ty+iy$, then the horocylic flow preserves the subspace $\QT_0(M)$ of quadratic differentials with real residues, and it also preserves the level sets $\QT_0(M,L)$ of quadratic differentials with fixed residues. By using Theorem \ref{thm_hubbard_masur_poles}; for a pair of transverse foliations $(\lambda_1,\lambda_2)$ we can see that there is a foliation $\phi_{\lambda_2}^t(\lambda_1)$ transverse to $\lambda_2$ and such that
\begin{equation*}
 U_t(\lambda_1,\lambda_2)=(\phi_{\lambda_2}^t(\lambda_1),\lambda_2).
\end{equation*}
In particular, for all multi-curve $\Gamma\in \MSr(M)$, it is possible to consider the subspace $\MF_\Gamma(M)$ in $\MF(M)$ formed by foliations transverse to $\Gamma$. We have a map 
\begin{equation*}
q_\Gamma~:~\MF_\Gamma(M)\longrightarrow \QT(M).
\end{equation*}
Moreover, the image is stable under the horocyclic flow. This defines a flow
\begin{equation*}
\phi_\Gamma^t~:~\MF_\Gamma(M)\longrightarrow, \MF_\Gamma(M),
\end{equation*}

We can see that this flow corresponds to the intuitive definition of the twist flow. To construct $q_\Gamma(S)$ we cut the graph along $\Gamma$ and glue two components by adding a cylinder of height $m_\gamma$ for each curve $\gamma$ in $\Gamma$. We can also see that we have the relationship
\begin{equation*}
    \phi_\Gamma^t=\prod \phi_\gamma^{m_\gamma t}
\end{equation*}
and the twist flows along disjoint curves commute. The twist flow admits a natural generalization for foliation, moreover, we can obtain the following result:

\begin{lem}
\label{lem_twist_rib}
\begin{itemize}
\item For all metric ribbon graphs $S\in \Met(R)$ and all curves $\gamma$ on $S$, there is $\delta>0$ small enough and $R^+\ge R$ (resp.$R^-\ge R$) such that $\phi^t_\gamma(S)\in \Met(R^+)$ (resp $\phi^{-t}_\gamma(S)\in Met(R^-)$) for $t\in ]0,\delta[$.
\item Moreover, if $S\in \Met(R)$, then $\phi^t_\gamma(S)\in \Met(R)$ for $|t|$ small enough, iff $\gamma$ is admissible on $R$.
\end{itemize}
\end{lem}

This is one of the main interest in introducing admissible curves and foliations, in this case we can see that for $\lambda\in \MF_0(R)$ the flow $\phi_\lambda^t$ is given locally by a translation on $\Met(R)$. The tangent vector will be denoted $\xi_\lambda$, an explicit computation, and the result of Proposition \ref{prop_coord_x} give the following corollary: 
\begin{cor}
\label{cor_tangent_twist}
If $\lambda$ is admissible, the twist flow is locally a translation generated by the locally constant vector field $\xi_\lambda$:
\begin{equation*}
\xi_\lambda= \sum_e x_e(\lambda)\partial_{e},
\end{equation*}
coordinates $(x_e(\lambda))_{e\in X_1R}$ are defined in Proposition \ref{prop_coord_x}. Then the map 
\begin{eqnarray*}
\MF_0(R)&\to&  K_R\\
    \lambda &\to& \xi_\lambda
\end{eqnarray*}
is bijective.
\end{cor}

\paragraph{Tangent of $\cut_\Gamma$:}
\label{para_tangent_cut}
The following proposition is useful to decompose the measures on the moduli space; we restrict it to the case of directed ribbon graphs. Let $\Ro$ be a directed ribbon graph, $\Gao$ an admissible directed multi-curve, and $\Go$ the associated directed stable graph. We denote $T_{\Ro,\Gao}$ the subspace of
\begin{equation*}
T_{\Ro_{\Gao}}= \prod_c T_{\Ro_{\Gao}(c)}
\end{equation*}
defined by (see remark \ref{rem_quotient_product_stab} for notation)
\begin{equation*}
    T_{\Ro,\Gao}={\prod}_\G T_{\Ro_{\Gao(c)}}.
\end{equation*}
And let
\begin{equation*}
 T_{\Ro,\Gao}(\Z)=T_{\Ro,\Gao}\cap T_{\Ro_{\Gao}}(\Z),
\end{equation*}
which is a lattice in $T_{\Ro,\Gao}$.

\begin{prop}
\label{prop_tangent_cut}
If $\Ro$ is a directed ribbon graph and $\Gao$ is admissible on $\Ro$, there is an exact sequence:
\begin{equation*}
0\longrightarrow \Z^{X_1\Gamma}\longrightarrow T_{\Ro}(\Z)\overset{T\cut_\Gamma}{\longrightarrow }T_{\Ro,\Gao}(\Z)\longrightarrow 0.
\end{equation*}
Moreover, the first non trivial map is defined by $\gamma\to \xi_\gamma$.
\end{prop}

\begin{proof}
It is the result of a long, exact sequence of relative integral homology, we have
\begin{equation*}
 0\to H_1(\Gamma, \Z)\to H_1(M^{cap}\backslash X_0R, X_2R, \Z)\to H_1(M^{cap}\backslash X_0R, X_2R \sqcup \Gamma, \Z)\to H_0(\Gamma, \Z)\to 0.
\end{equation*}
The kernel of
\begin{equation*}
    H_1(M^{cap}\backslash X_0R, X_2R \sqcup \Gamma, \Z)\to H_0(\Gamma, \Z),
\end{equation*}
is the space $T_{\Ro,\Gao}(\Z)$; the first non-trivial arrow is the inclusion, the second space is identified with the tangent space using $f_R$ and Poincare duality. By excision, the last space is identified with
\begin{equation*}
H_1(M^{cap}_\Gamma \backslash X_0R_\Gamma, X_2R_\Gamma, \Z), 
\end{equation*}
which is also identified with $T_{R_\Gamma}$. It remains to prove that the last arrow is the tangent of the cutting map. We have for all $\gamma\in \Sit(M_\Gamma)$
\begin{equation*}
l_\gamma(S_\Gamma)=l_\gamma(S).
\end{equation*}
And then
\begin{equation*}
\cut_\Gamma^* dl_\gamma =dl_\gamma.
\end{equation*}
Moreover, we have $f_R^*(dl_\gamma)=y_\gamma$. Then, for all cocycles $\langle \cut_\Gamma(x),\gamma \rangle=\langle x,\gamma \rangle$, we can now conclude because the only map that satisfes this property is the natural inclusion. 
\end{proof}

\paragraph{Gluings and bundles for directed surfaces:}
\label{gluing_coordinates}
In this part, we study in more detail the map $\cut_\Gamma$ in a particular case. Let $\Mo$ be a directed surface and $\Gao$ be a directed multi-curve on it. We consider the subset $\MF_{\Gao}(\Mo)$ of directed foliations on $\Mo$ transverse to $\Gamma$, which are represented by an Abelian differential. Such differential induces a direction on the curves of $\Gamma$, and we assume that it corresponds to $\Gao$. For all stable graphs $\Go$, let $\Tc(\Go)$ be the subset of $\prod_c \Tc(\Go(c))$:
\begin{equation*}
\Tc(\Go)={\prod}_{\G} \Tc(\Go(c)).
\end{equation*}
Let $\Gao$ be a directed multi-curve with a directed stable graph $\Go$, then we have the following proposition:
\begin{prop}
\label{prop_cutting_measure_bundle}
$\MF_{\Gao}(\Mo)$ is stable under the twist flow along each curve in $\Gamma$; moreover, the map
\begin{equation*}
\cut_\Gamma: \MF_{\Gao}(\Mo) \longrightarrow \Tc(\Go).
\end{equation*}
Is surjective, and it is an affine $\R^{\Gamma}$ bundle compatible with the integral structure.
\end{prop}

\begin{rem}
We can also perform surgeries on each stratum. Let  $\Mdo$ a directed, decorated surfaces; $\Gado\in \MS(\Mdo)$, and $\Gdo$ is the corresponding directed, decorated stable graph. Then we can consider $\MF_{\Gado}(\Mdo)$ the directed foliations $\lambdao$ such that $\lambda$ is transverse to $\Gamma$, and the direction induced on $\Gamma$ corresponds to $\Gao$. Moreover, the square root of $q_\Gamma(\lambda)$ defines a decoration on each connected component of $\Mo_{\Gao}$ and then on $\Gao$. We assume that this decoration coincides with $\Gado$. Then we still have
\begin{equation*}
\cut_\Gamma : \MF_{\Gado}(\Mdo)\to \Tc(\Gdo),
\end{equation*}
and the statement of the last proposition remains true.
\end{rem}

We prove the following lemma, which gives the existence of local twist coordinates and also gives proposition \ref{prop_cutting_measure_bundle}.
\begin{lem}
For each $S\in \Tc(\Go)$, it exists $V\subset \Tc(\Go)$ and $U\subset \MF_{\Gao}(\Mo)$ such that
\begin{itemize}
\item $V$ is an open neighborhood of $S$ invariant by dilatation, and $U=\cut_{\Gao}^{-1}(V)$.
\item For all $\gamma\in \Gamma$, we can find a map $t_\gamma : U \longrightarrow \R$ such that
\begin{equation*}
t_\gamma(\phi^t_{\gamma'})=t_\gamma + t \delta_{\gamma,\gamma'}
\end{equation*}
for all $\gamma'\in \Gamma$.
\item There is a piecewise linear isomorphism:
\begin{equation*}
U\longrightarrow V \times \R^{\Gamma},
\end{equation*}
which induces a bijection
\begin{equation*}
U_\Z \longrightarrow V_\Z \times \Z^\Gamma.
\end{equation*}
Where $U_\Z,V_\Z$ are the sets of integer points
\end{itemize}
\end{lem}


\newpage

\section{Acyclic decomposition:}
\subsection{Symplectic geometry on the space of metric ribbon graphs}
\label{subsection_symplectic}
There is a natural two-form on the combinatorial Teichmüller spaces defined by M. Kontsevitch in \cite{kontsevich1992intersection} and studied in more detail in \cite{zvonkine2002strebel} and \cite{andersen2020kontsevich}. Here we relate this two-form to the natural pairing on the anti-invariant cohomology.
\label{Kon_symp_form}
\paragraph{The anti-symmetric pairing on $K_R$:}
Let $R$ be a ribbon graph. As we see in paragraph \ref{paragraph_coho_ribbongraph} there are natural identifications between the tangent space and the cohomology of the ribbon graph:
\begin{equation*}
T_R\overset{f_R}{\simeq}H^1(R,X_0R,X_2R)~~~\text{and}~~~ K_R\overset{f_R}{\simeq}H^1(R,X_0R).
\end{equation*}
The space $H^1(R)$ can be identified with the anti-invariant cohomology of the two covers $\tilde{M}_{R}^{cap}$ (and the cohomology of $M_R^{cap}$ when the graph is directed). There is a natural anti-symmetric pairing given by the cup product on the cohomology of $\tilde{M}_{R}^{cap}$,
\begin{equation*}
\langle\omega_1,\omega_2\rangle = \frac{1}{2}\int_{\tilde{M}_{R}^{cap}} \omega_1\wedge\omega_2.
\end{equation*}
We can drop the $\frac{1}{2}$ and integrate over $M^{cap}_R$ in the directed case. In this way, we define an anti-symmetric pairing $\Omega_R$ on $H^1(R)$.
\begin{lem}
\label{lem_symplectic_H1(R)}
The space $(H^1(R),\Omega_R)$ is a symplectic vector space.
\end{lem}

There is a natural map from the relative cohomology to the absolute cohomology:
\begin{equation*}
K_R\simeq H^1(R,X_0R)\longrightarrow H^1(R),
\end{equation*}
We also denote $\Omega_R$,
the anti-symmetric pairing induced on $K_R$ by taking the pull back of $\Omega_R$ on $H^1(R)$.

\paragraph{Degeneration of the symplectic structure:}
\label{deg_symp}

The long exact sequence in the relative cohomology is useful to study the degeneration of the symplectic structure:
\begin{lem}
\label{lem_deg_symp_pair}
For each ribbon graph $R$, we have an exact sequence of relative cohomology:
\begin{equation*}
 \{0\} \rightarrow H^0(R)\rightarrow H^0(X_0R)\rightarrow H^1(R,X_0R)\rightarrow H^1(R) \rightarrow \{0\}.
\end{equation*}
Moreover, the dimension of $H^0(X_0R)$ is the number of vertices of $R$ with an even degree, and the dimension of $H^0(R)$ is the number of dirigible connected components of the graph $R$.
\end{lem}
\begin{proof}
We have a short, exact sequence:
\begin{equation*}
\{0\}\to C^*(R,X_0R)\to C^*(R)\to C^*(X_0R)\to \{0\},
\end{equation*}
it gives the long exact sequence of cohomology. We have $H^0(R,X_0R)=0$ because $C^0(R,X_0R)=0$; and $H^1(X_0R)=0$ because $C^1(X_0R)=0$. We also have:
\begin{equation*}
H^0(X_0R)\simeq C^0(X_0R)^-\simeq (\R^{X_0\Rt})^-.
\end{equation*}
The involution acts on $X_0\Rt$. For each $v\in X_0R$, either $v$ is odd, and then it has only one pre-image fixed by the involution. Either it is even and it has two pre-images exchanged by the involution. Then we see that the space of anti-invariant elements is isomorphic to $\R^{X_0^{even}R}$.
\end{proof}
We can see that the pairing is non-degenerate when the graph has only vertices of odd degrees. If $H_R$ is the image of $H^0(X_0R)$ and the ribbon graph $R$ is connected and dirigible, then, the dimension of $H_R$ is the number of vertices of the graph minus one.
\begin{rem}
The elements $\xi \in H_R$ are characterized by the following property:
\begin{equation*}
    dl_\lambda(\xi)=0~~~~~~\forall \lambda \in \MF_0(R).
\end{equation*}
Then twist flows along elements in the kernel foliation preserve lengths of admissible curves.
\end{rem}

\paragraph{Minimal graphs:}
From Lemma \ref{lem_deg_symp_pair},if $\Ro$ is directed and connected, the form $\Omega_R$ is non-degenerate iff the graph has only one vertex. In a dual way, such graphs are also unicellular maps; as we see later, they are building bricks to construct more complicated directed ribbon graphs.

\begin{Def}
\label{def_minimal}
A directed ribbon graph $\Ro$ is minimal iff it has only one vertex, i.e., iff $(K_R,\Omega_R)$ is a symplectic vector space.
\end{Def}

In figure \ref{fig_minimal}, we give minimal graphs for low degrees. In general, a minimal ribbon graph is not necessarily irreducible; it can have a nontrivial genus (see definition \ref{Def_irreducible} for an irreducible graph).

\begin{figure}
    \centering
    \includegraphics[width=0.8\linewidth]{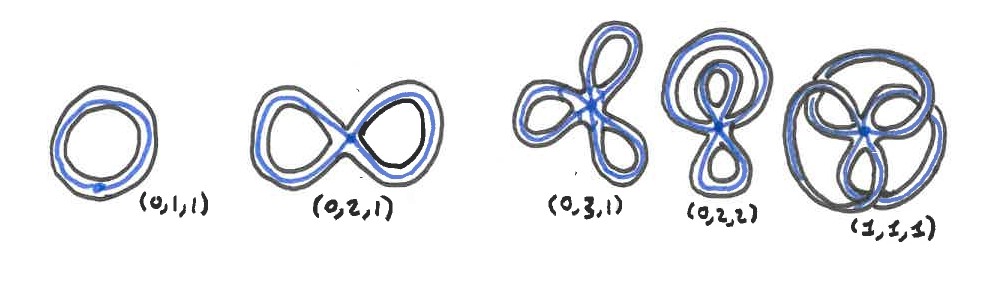}
    \caption{List of minimal graphs for low genus.}
    \label{fig_minimal}
\end{figure}
\paragraph{Hamiltonian of the horocyclic flow:}
\label{para_hamiltonian}
Here we give a sketch of the proof of the following theorem, which was also proved in \cite{andersen2017geometric} for the principal stratum.
\begin{thm}
\label{twist_ham_lenght}
Let $R$ be any ribbon graph and $\lambda \in \MF_0(R)$. We have the following identity on $K_R$:
\begin{equation*}
    \Omega_R(\xi_\lambda,.)=-dl_\lambda.
\end{equation*}
\end{thm}
\begin{proof}
    Let $H^1(\tilde{M}^{cap}_{R},X_0\tilde{R}\cup X_2\tilde{R},\R)^-$ as in section \ref{paragraph_coho_ribbongraph}. There is two projections: 
\begin{eqnarray*}
     p_R: H^1(\tilde{M}^{cap}_{R},X_0\tilde{R}\cup X_2\tilde{R},\R)^- &\to& H^1(\tilde{M}^{cap}_{R},X_0\tilde{R},\R)^-\\
    p_R^*: H^1(\tilde{M}^{cap}_{R},X_0\tilde{R}\cup X_2\tilde{R},\R)^- &\to& H^1(\tilde{M}^{cap}_{R},X_2\tilde{R},\R)^-
\end{eqnarray*}
    On each of these three spaces, there is a natural pairing given by the pullback of the pairing on $H^1(\tilde{M}^{cap}_{R},\R)^-$. This induces a paring between $H^1(\tilde{M}^{cap}_{R},X_0\tilde{R},\R)^-$ and $H^1(\tilde{M}^{cap}_{R},X_2\tilde{R},\R)^-$.  If $z,z'\in H^1(\tilde{M}^{cap}_{R},X_0\tilde{R}\cup X_2\tilde{R},\R)^-$ we have 
    \begin{equation*}
     \langle z , z' \rangle =  \langle p_R(z) , p_R(z') \rangle=  \langle p_R(z) , p_R^*(z') \rangle
    \end{equation*}
    Let $x(z)$ and $y(z')$ the coordinate on $H^1(\tilde{M}^{cap}_{R},X_0\tilde{R},\R)^-$ and $H^1(\tilde{M}^{cap}_{R},X_2\tilde{R},\R)^-$
    \begin{equation*}
        x_e(z)=\int_{[e]}z~~~~~~ y_e(z')=\int_{[e]}z'~~\forall e \in X_1R
    \end{equation*}
    We have 
    \begin{equation*}
        p_R(z)=\sum_e x_e(z)[e]^*~~~~~~ p_R~^*(z')=\sum_e y_e(z')[e^*]^*
    \end{equation*}
    Moreover, $\langle [e]^*,[e^*]^*\rangle=1$ and then 
    \begin{equation*}
        \langle p_R(z) , p_R^*(z') \rangle= \sum_e x_e(z)y_e(z').
    \end{equation*}
    Let $\lambda,\lambda'$ and $z(\lambda),z(\lambda')$ defined in paragraph \ref{zippered_dual}. We have $p_R(z(\lambda))=x(\lambda)=f_R(\xi_{\lambda})$ and $p_R^*(z(\lambda))=y(\lambda)=f_R^*(dl_{\lambda'})$, then we can obtain:
    \begin{equation*}
        \Omega_R(\xi_{\lambda},\xi_{\lambda'})= \langle z(\lambda) , z(\lambda') \rangle = \sum_e x_e(\lambda)y_e(\lambda') = dl_{\lambda'}(\xi_{\lambda}).
    \end{equation*}
   The map $\lambda'\rightarrow \xi_{\lambda'}$ is surjective on $K_R$; we can conclude that
\begin{equation*}
\Omega_R(\xi_\lambda,.)= - dl_\lambda(.).
\end{equation*}
\end{proof}
\paragraph{Computation of the pairing in coordinates:}

We give several expressions for the pairing. The elements $(x_e)_{e\in X_1R}$ and $(y_e)_{e\in X_1R}$ define $1-$forms on $W_R$; using them, it is possible to express the $2-$form $\overline{\Omega}_R$ :
\begin{equation*}
\overline{\Omega}_R=\frac{1}{2}\sum_{e\in X_1R} x_e \wedge y_e.
\end{equation*}
In terms of the $z$ coordinates, we also have the following expression:
\begin{equation*}
    \overline{\Omega}_R=\frac{1}{2}\sum_{e\in XR} z_{s_0e} \wedge z_{e},
\end{equation*}
which is in some sense the Thurston form due to similarity with the Thurston two form on the train track $\tau_R$ . There is a dual expression of this formula. The roles of $s_0$ and $s_2$ are in some sense symmetric, and we have
\begin{equation*}
\overline{\Omega}_R=\frac{1}{2}\sum_{e\in XR} z_e \wedge z_{s_2e}.
\end{equation*}
These two expressions come from the formula
\begin{equation*}
x_e\wedge y_e=z_e\wedge z_{s_0^{-1}e} + z_{s_1e}\wedge z_{s_0^{-1}s_1e}=z_{e}\wedge z_{s_2e} + z_{s_1e}\wedge z_{s_2s_1e}.
\end{equation*}

It is also possible to give expressions in terms of the forms $x$ and $y$. In \cite{kontsevich1992intersection}, M. Kontsevich introduces for each boundary $\beta$ a two-form on $K_R$ defined in the following way: Let an edge $e$ with $[e]_2=\beta$ and assume that the boundary contains $r$ edges,
\begin{equation*}
\omega_\beta=\sum_{0\le i<j<r}x_{s_2^ie}\wedge x_{s_2^je}=\sum_{1\le j\le r} z_{s_2^{i-1}e}\wedge z_{s_2^{i}e}=\sum_{e,[e]_2=\beta} z_e\wedge z_{s_2e}.
\end{equation*}
Then, by summing over the boundaries, we recover the second Thurston form and deduce
\begin{equation*}
\overline{\Omega}_R=\frac{1}{2}\sum_\beta \omega_\beta.
\end{equation*}
In a similar way, for each vertex $v$, we can fix an edge $e$ with $[e]_0=v$. If the vertex is of degree $r$, then we can set
\begin{equation*}
\hat{\omega}_v=\sum_{0\le i<j<r}(-1)^{i+j}y_{s_0^ie}\wedge y_{s_0^je}=-\sum_{1\le j\le r} z_{s_0^ie}\wedge z_{s_0^{i-1}e}
\end{equation*}
and then:
\begin{equation*}
\overline{\Omega}_R=\frac{-1}{2}\sum_\beta  \hat{\omega}_v.
\end{equation*}

\newpage
\subsection{Acyclic decomposition}
\label{subsection_acyclic_decomposition}
In this section, we state the main theorem of the paper (Theorems \ref{intro_thm_acycl_curve} and \ref{intro_thm_acycl_decomp} in the introduction).

\begin{figure}
\centering
\begin{subfigure}{0.3\textwidth}
    \includegraphics[width=\textwidth]{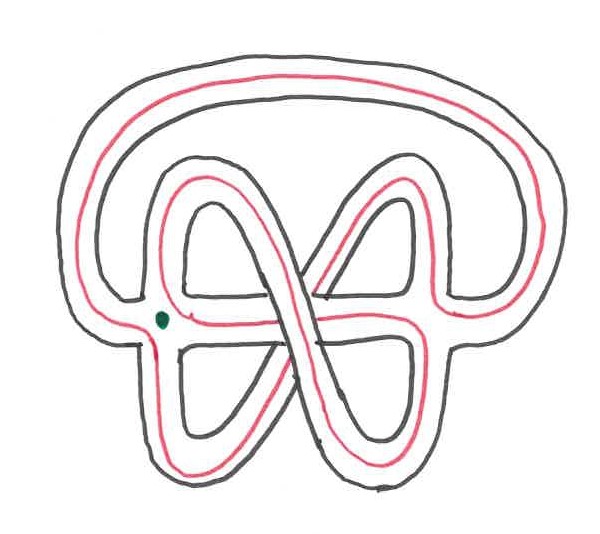}
\end{subfigure}
\begin{subfigure}{0.3\textwidth}
    \includegraphics[width=\textwidth]{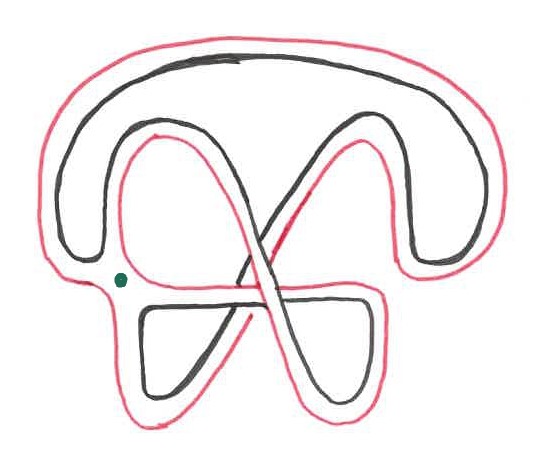}
\end{subfigure}
\begin{subfigure}{0.3\textwidth}
    \includegraphics[width=\textwidth]{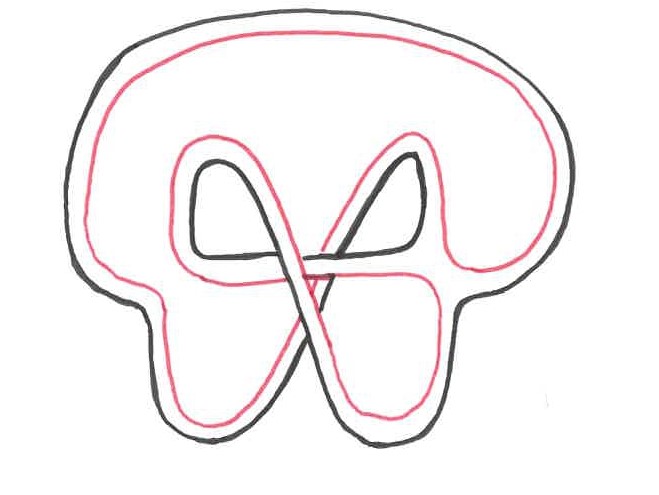}
\end{subfigure}
\caption{Acyclic decomposition of a directed graph of type $(1,1,1)$.}
\label{figure_acyclic_decomp_tore}
\end{figure}

\paragraph{Vertex surgery theorem:}

Let $\Ro$ be a directed ribbon graph. We say that an admissible multi-curve spares a vertex $v$ in $R$ from the rest of the graph if the component that contains $v$ in $R_\Gamma$ is minimal (see definition \ref{def_minimal}).

\begin{thm}
\label{thm_acycl_curve}
Let $\Ro$ be a directed metric ribbon graph with at least two vertices. For each vertex $v$, there exists a unique admissible primitive multi-curve $\Gamma_v^{+}$ such that:
\begin{itemize}
\item The directed stable graph $\Go_v$ of $\Gamma_v^+$ contains a component $c_v$ that spares $v$ from the rest of the graph.
\item All the curves in $\Gamma_v^+$ are boundaries of $c_v$.
\item $c_v$ is glued along negative boundaries only.
\end{itemize}
\end{thm}
 In other words, the last point means that the directed stable graph associated to $\Gamma_v^+$ is acyclic and the component $c_v$ is a maximal element for the partial order. The multi-curve $\Gamma_v^+$ satisfies several elementary properties.

\begin{itemize}
\item The multi-curve is functorial in the sense that if $g: (\Ro_1,v_1) \rightarrow (\Ro_2,v_2)$ is a morphism that preserves the direction, then $g\cdot \Gamma_{v_1}^+(\Ro_1)=\Gamma_{v_2}^+(\Ro_2)$.
\item There is a similar result for negative boundaries: $\Gamma_v^-(\Ro)$ is defined to be $\Gamma_v^+(-\Ro)$, where $-\Ro$ is obtained by reversing the direction of $\Ro$. If $\xi_v^{\pm}$ is the twist flow along $\Gamma_v^{\pm}(\Ro)$, then we have $\xi_v^{-}=-\xi_v^{+}$.
\item The tangent vectors of the twist flow $\xi_v^{+}$ are in $H_R$.
\end{itemize}

\paragraph{Acyclic decomposition:}
\label{para_acyclic_decomp}
We start with the following three definitions:

\begin{Def}
\begin{enumerate}
\item An acyclic decomposition of a directed ribbon graph $\Ro$ is an admissible multi-curve $\Gao$ such that the directed stable graph $\Go$ associated to $\Gao$ is acyclic.
\item An acyclic decomposition is maximal if it is not contained in another acyclic decomposition.
\item A linear order\footnote{We use this terminology about linear order to be coherent with the terminology on directed graphs and order relations in general.} on $\Ro$ is an enumeration of the vertices of the ribbon graph.
\end{enumerate}
\end{Def}
In this section, from Theorem \ref{thm_acycl_curve} we deduce Theorem \ref{thm_acycl_decomp} (Theorem \ref{intro_thm_acycl_decomp} in the introduction):
\begin{thm}
\label{thm_acycl_decomp}
Let $\Ro$ be a directed ribbon graph with a linear order; then, there is a unique admissible primitive multi-curve $\Gao$ such that:
\begin{enumerate}
\item The components of $\Ro_{\Gao}$ are minimal.
\item The directed stable graph $\Go$ associated to $\Gao$ is acyclic.
\item The linear order on the vertices induces a linear order on the acyclic stable graph.
\end{enumerate}
\end{thm}
With Lemma \ref{lem_maximal_minimal}, we can rephrase this result in the following way: Given a linear order on the ribbon graph $\Ro$, there is a unique maximal acyclic decomposition that is compatible with this linear order. We remark that different linear orders can produce the same decomposition; a given acyclic stable graph can have several linear orders.
\begin{lem}
\label{lem_maximal_minimal}
An acyclic decomposition $\Gao$ of a directed graph $\Ro$ is maximal iff all the components of $\Ro_{\Gao}$ are minimal (see definition \ref{def_minimal}).
\end{lem}
\begin{proof}
Assume that some components are not minimal. Pick one of them. Using Theorem \ref{thm_acycl_curve}, we can decompose this component along a multi-curve $\Gamma'$ by removing a vertex. Moreover, the directed stable graph associated to this new decomposition is acyclic. Using Proposition \ref{prop_gluing_acycl}, the stable graph associated with $\Gamma\sqcup \Gamma'$ is also acyclic, and then $\Gamma$ is not maximal. Assuming that all the components are minimal, it is easy to see that an admissible multi-curve on a minimal graph is necessarily associated with a directed stable graph with a non-trivial cycle, and then a finer decomposition will not be acyclic.
\end{proof}
We prove Theorem \ref{thm_acycl_decomp}.
\begin{proof}
We can proceed by induction on the number of vertices. It is, of course, trivial when the graph has only one vertex; the multi-curve is empty in this case. If we assume the property is true for ribbon graphs with $n$ vertices and let $\Ro$ be a directed graph with $n+1$ vertices and a linear order, Let $v_{n+1}$ be the vertex labeled by $n+1$. According to Theorem \ref{thm_acycl_curve}, we can extract $v_{n+1}$ by using $\Gamma_{v_{n+1}}^+$. We obtain a ribbon graph $\Ro_{n+1}$ that contains $v_{n+1}$ and a family of ribbon graphs glued to $\Ro_{n+1}$. The linear order on $\Ro$ induces a linear order on each of these graphs, and then, by assumption, we can find an acyclic decomposition of each graph compatible with the linear order. Then, using Proposition \ref{prop_gluing_acycl}, the concatenation of these decompositions is still acyclic, and the linear order on $\Ro$ induces a linear order on the directed stable graph. Uniqueness is a consequence of uniqueness in Theorem \ref{thm_acycl_curve}. Assuming that we have such decomposition, then the component that contains $v_{n+1}$ is necessarily a local maximum of the acyclic stable graph, and this component is necessarily glued along negative boundaries only. If $\Gamma_{n+1}$ are the curves in the decomposition that are in the boundary of this component, then $\Gamma_{n+1}$ satisfies the hypothesis of Theorem \ref{thm_acycl_curve} and then must be $\Gamma_{v_{n+1}}^+$. By induction, we obtain uniqueness.
\end{proof}
\begin{rem}[Acyclic versus Fenchel-Nielsen decomposition]
Acyclic decomposition and Fenchel-Nielsen decomposition are distinct notions. In the second case, the graph is not necessarily acyclic; in the first case, the components are not supposed to be irreducible. With Theorem \ref{thm_acycl_decomp}, we are not allowed to cut the graph along a curve that is a handle, some components of the decomposition can have non trivial genus.
\end{rem}
In the case of the sphere, the minimal ribbon graphs are irreducible, then we have the following corollary:

\begin{cor}
\label{cor_acyclic_decomp_sphere}
Let $\Ro$ be a directed ribbon graph on the sphere with a linear order, then there is a unique admissible multi-curve $\Gao$ such that:
\begin{enumerate}
\item The directed stable graph $\Go$ associated with $\Gao$ is acyclic.
\item The linear order on the vertices is compatible with the order on the graph.
\item All the components of $\Ro_{\Gao}$ are irreducible.
\end{enumerate}
\end{cor}


\paragraph{Proof of the Theorem \ref{thm_acycl_curve}:}
\label{proof_thm_1}
Let $\Ro=(R,\epsilon)$ be a directed ribbon graph, and let $v$ be a vertex. We construct the tangent vector associated with the twist flow along $\Gamma_v^{+}$ (see figure \ref{figure_vector_xiv}). Let $\xi_v^{+}(\Ro)$ be the vector in $T_R$ defined by
\begin{equation*}
\xi_v^{+}(\Ro)=-\sum_{e,[e]_0=v} \epsilon(e) \partial_{[e]_1}.
\end{equation*}
By using Proposition \ref{prop_x_admissible_curves}, there is a unique admissible multi-curve $\Gamma^+_v(\Ro)\in \MS_\Z(R)$ such that $\xi_{\Gamma^+_v(\Ro)}=\xi_v^{+}(\Ro)$. We prove that $\Gamma^+_v$ satisfies all the desired properties. We start by proving two lemmas.

\begin{figure}
    \centering
    \includegraphics[width=0.3\linewidth]{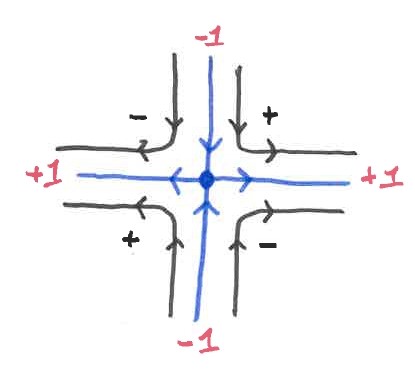}
    \caption{The vector $\xi_v^{+}$.}
    \label{figure_vector_xiv}
\end{figure}
\begin{lem}
\label{basis_HR}
The vector $\xi_v^{+}(\Ro)$ belongs to the image of $H_R$ (see paragraph \ref{deg_symp}) and the vectors $(\xi_v^+)_{v\in X_0R}$ span $H_R$. Moreover, if $\Ro$ is connected, the only relations between the vectors $(\xi_v^{+}(\Ro))_{v\in X_0R}$ are proportional to
\begin{equation*}
\sum_{v\in X_0R} \xi_v^{+}(\Ro) =0.
\end{equation*}
\end{lem}

\begin{proof}
We have the exact sequence for the relative homology:
\begin{equation*}
\{0\}\longleftarrow H_0(R) \longleftarrow H_0(X_0R)\overset{\partial}{\longleftarrow} H_1(R,X_0R) \longleftarrow H_1(R)
\end{equation*}
Moreover, the boundary map is just
\begin{equation*}
\partial [e]= [s_1e]_0-[e]_0.
\end{equation*}
Then, by duality, the boundary map in the cohomology $d : H^0(X_0R)\longrightarrow H^1(R,X_0R)$ is defined by
\begin{equation*}
\langle d[v], [e]\rangle = \langle[v], d[e]\rangle = \delta_{v,[s_1e]_0}-\delta_{v,[e]_0}.
\end{equation*}
By using the identification between $H^1(R,X_0R)$ and $K_R$, we obtain
\begin{equation*}
d[v]=\sum_{[e],\epsilon(e)=1}\langle d[v], [e]\rangle\partial_e = \xi_v^+.
\end{equation*}
Then, by using the exact sequence in relative cohomology (see Lemma \ref{lem_deg_symp_pair}, we can finish the proof.
\end{proof}

From Lemma \ref{basis_HR}, we can obtain the following proposition:

\begin{prop}
\label{prop_minimal_carac_v}
If $\Ro$ is connected and $v\in X_0R$, we have $\xi_v^+(\Ro)=0$ iff the graph $\Ro$ is minimal.
\end{prop}

The following lemma is also elementary but very useful. 

\begin{lem}
\label{lem_cut_v}
Let $\Gamma$ be an admissible curve, $v\in X_0R$, and $v'\in X_0R_\Gamma$ be the image of $v$ in the graph $R_\Gamma$. Then we have the following relation: 
\begin{equation*}
 T\cut_\Gamma( \xi_v^+(\Ro)) = \xi_{v'}^+(\Ro_{\Gao}).
\end{equation*}
\end{lem}

\begin{proof}[Lemma \ref{lem_cut_v}]
From part \ref{paragraph_coho_ribbongraph}, we have the identification 
\begin{equation*}
T_R\simeq H_1(M^{cap}_R\backslash X_0R,X_2R,\R).
\end{equation*}
The vector $\xi_v^+(\Ro)$ corresponds to a circle around the vertex $v$ in the homology. The tangent map of $\cut_\Gamma$ corresponds to the map in homology.
\begin{equation*}
H_1(M^{cap}_R\backslash X_0R,X_2R,\R)\longrightarrow H_1(M^{cap}_R\backslash X_0R,X_2R \sqcup \Gamma,\R)\simeq H_1(M^{cap}_{R_\Gamma}\backslash X_0R_\Gamma,X_2R_\Gamma,\R).
\end{equation*}
This map is the inclusion of the homology into the relative homology, and then it maps the circle around $v$ to the circle around $v'$.
\end{proof}

We can prove the first point in Theorem \ref{thm_acycl_curve}. Let $R_v^+$ be the component of $\Ro_{\Gamma_v^+}$ that contains $v$.

\begin{lem}
The graph $R_v^+$ is minimal.
\end{lem}

\begin{proof}
By using Lemma \ref{lem_cut_v}, we have the relation:
\begin{equation*}
T\cut_v(\xi_v^+(\Ro))=\xi_v^+(R_v^+).
\end{equation*}
But as we seen in Proposition \ref{prop_tangent_cut}, the twist flow along $\Gamma_v^+$ is in the kernel of $T\cut_{\Gamma_v^+}$, so we have $T\cut_v(\xi_v^+(\Ro))=0$. and then $\xi_v^+(R_v^+)=0$ by using Lemma \ref{lem_cut_v}. Moreover, by using Proposition \ref{prop_minimal_carac_v}, we can conclude that $R_v^+$ is minimal.
\end{proof}

Let $\Go_{v,+}$ be the stable graph associated with $\Gamma_v^+$. We decompose $\Gamma_v^+$ into three sets of curves $A_i~,~i=1...3$. Where:
\begin{itemize}
\item $A_1$ are the curves that join a boundary of $R_v^+$ to another other vertex of $\G_v^+$.
\item $A_2$ are the curves that are not in the boundary of $R_v^+$
\item $A_3$ are the curves that connects two boundaries of $R_v^+$
\end{itemize}
\begin{lem}
The sets $A_2,A_3$ are empties.
\end{lem}

\begin{proof}
Consider the graph $\Ro_{A_1}$ obtained after cutting $\Ro$ along the curves in $A_1$ we can write it as $\Ro_{v,1}\sqcup \Ro_{v,2}$, where $\Ro_{v,1}$ is the connected component that contains $v$ and $\Ro_{v,2}$ is the union of all the other components. By the Lemma \ref{lem_cut_v} we have
\begin{equation*}
T\cut_{A_1}(\xi_v^+(\Ro))=\xi_v^+(\Ro_{A_1})=\xi_v^+(\Ro_{v,1}).
\end{equation*}
The graph $\Ro_{v,1}$ is minimal because $(\Ro_{v,1})_{A_3}=\Ro_{v,+}$ then we have
\begin{equation*}
\xi_v^+\in \Ker~( T\cut_A).
\end{equation*}
By using Proposition \ref{prop_tangent_cut}, the kernel of $T\cut_{A_2}$ is generated by twist flows along $A$, and then
\begin{equation*}
\xi_v^+\in \bigoplus_{\gamma\in A_1} \Z~\xi_\gamma.
\end{equation*}
By using Proposition \ref{prop_tangent_cut} again, the tangent vectors of twist flow along disjoint curves are independent. Then the vectors $\xi_\gamma,\gamma\in \Gamma_v^+$ are free. By writing
\begin{equation*}
\xi_v^+=\xi_1+\xi_2+\xi_3,
\end{equation*}
with $\xi_i\in \text{Span}(\xi_\gamma,~\gamma\in A_i)$, we see that $\xi_2,\xi_3$ must vanish, which implies that $A_2,A_3$ must be empty.
\end{proof}

We give an algorithmic construction of the curve $\Gamma^+_v$:
\begin{itemize}
\item Start with an edge $e_0$ with $[e_0]_0=v$ and $\epsilon(e_0)=1$.
\item If $[s^+e_k]_0=v$ then $e_{k+1}=s^- e_k$.
\item Else, if $[s^+e_k]_0\neq v$, then $e_{k+1}=s^+e_k$.
\end{itemize}

After some time, the process ends at $e_0$, and the result is a combinatorial curve. If the curve does not contain all the directed edges $e$ with $\epsilon(e)=1$ and $[e]_0=v$, then we restart the procedure on another edge. In other words, to construct the curve, we take all the positive boundaries that meet $v$ and perform cutting and gluing at $v$, which are described in figure \ref{figure_symp_curve_constr}. At the end, the procedure creates the minimal representation of a family of curves $\tilde{\Gamma}_v^+$. From the construction and result of section \ref{subsection_admissible}, we can derive

\begin{itemize}
\item The intersection coordinates $y_{e}(\widetilde{\Gamma}_v^+)$ are in $\{0,1\}$,
\item Each curve in $\tilde{\Gamma}_v^+$ is simple, and $\tilde{\Gamma}_v^+$ is a multi-curve.
\item Each curve in $\tilde{\Gamma}_v^+$ is admissible.
\end{itemize}

The multi-curve $\tilde{\Gamma}_v^+$ is in $\widetilde{\MS}_\Z(R)$; all connected components are either in $\Si(M)$ or are homotopic to a negative boundary. Let $B^+$ be the positive boundaries of $\Ro$ that are adjacent to $v$ and $B^-$ be the negative boundaries of $\Ro$ that are adjacent to $v$ only. From the construction we can obtain the following lemma:

\begin{lem}
The curve constructed in this way is $\Gamma^+_v + \sum_{\beta\in B^-} \beta$, then we have on $T_R$
\begin{equation*}
dl_{\Gamma^+_v}=\sum_{\beta\in B^+} dl_\beta - \sum_{\beta\in B^-} dl_\beta.
\end{equation*}
\end{lem}

\begin{figure}
\centering
\includegraphics[height=4cm]{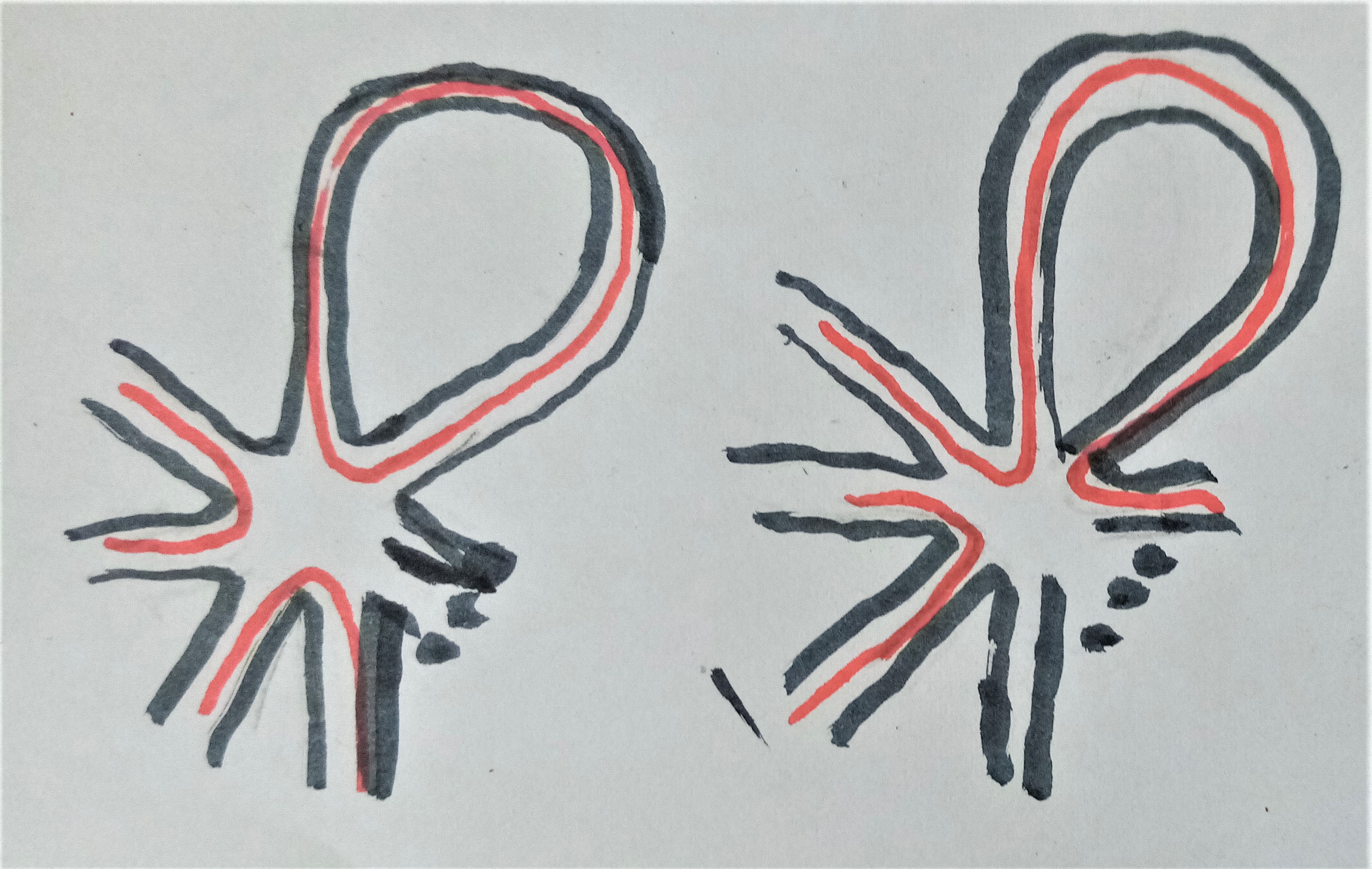}
\caption{Illustration of the construction.}
\label{figure_symp_curve_constr}
\end{figure}

We prove that $\Gamma_v^+$ satisfies the last property in Theorem \ref{thm_acycl_decomp}. Let $M_v=M_{R_v^+}$ be the component that contains $v$ and $M=M_R$. We have the relative sequence in cohomology:
\begin{equation*}
H^1(M,X_0R,\R) \longrightarrow H^1(M_v,\R) \longrightarrow H^2(M,M_v,\R) \longrightarrow 0,
\end{equation*}
where the first map is the projection $\cut_v$ of $\cut_{\Gamma_v^+}$
\begin{equation*}
T_R\longrightarrow T_{R_{\Gamma_v^+}} \longrightarrow T_{R_v^+}
\end{equation*}
Then the obstruction to the map being surjective is in $H^2(M,M_v,\R)$.

\begin{lem}
\label{lem_H^2(M,M_v,R)}
We have $H^2(M,M_v,\R)=\{0\}$.
\end{lem}
Let $R_v(c)$ be a component of $R_{\Gamma_v^+}$ that doesn't contain $v$, and $\Gamma_v^+(c)$ be the boundaries of $R_v(c)$ that are in $\Gamma_v^+$. Let $\epsilon_v$ be the direction on $\Gamma_v^+$ induced by $R_v^+$. Then we have the following fact:
\begin{lem}
\label{lem_constant_direct_component}
The direction $\epsilon_v$ is constant on $\Gamma_v^+(c)$.
\end{lem}

\begin{proof}
Assume it is not. As the curves in $\Gamma_v^+(c)$ are naturally oriented they define an element $[\Gamma_v^+(c)]\in H_1(M_R\backslash X_0R,\R)$. We see that we can find a simple curve in the homology such that
\begin{equation*}
\langle [\Gamma_v^+(c)],[\gamma']\rangle \neq 0;~~~\text{and}~~~ \langle [\Gamma_v^+(c')],[\gamma']\rangle = 0,~~~\text{if} ~~c\neq c'.
\end{equation*}
Then
\begin{equation*}
dl_{\Gamma_v^+}(\xi_{\gamma'})=\langle [\Gamma_v^+],[\gamma']\rangle\neq 0
\end{equation*}
but by assumption, $dl_{\Gamma_v^+}=0.$ on $K_R$, and then the direction must be constant.
\end{proof}

Then we prove Lemma \ref{lem_H^2(M,M_v,R)}:

\begin{proof}
The dimension of the space $H^2(M,M_v,\R)$ is equal to the number of components of $M_{\Gamma_v^+}$ such that all the boundaries of the component are in $\Gamma_v^+$. Then, using Lemma \ref{lem_constant_direct_component}, this is impossible because the direction is constant on a given component.
\end{proof}
From Lemma \ref{lem_H^2(M,M_v,R)}, we obtain

\begin{lem}
\label{relation_trv}
On $T_{R_v}$, we have the relation
\begin{equation*}
dl_{\Gamma_v^+}+dl_{B^-}-dl_{B^+}=0
\end{equation*}
\end{lem}

From Proposition \ref{prop_dimorientability}, the only relation on $T_{R_v}$ is given by the direction $\epsilon_v$ of the graph. Then, by writing
\begin{equation*}
dl_{\Gamma_v^+}=\sum_{\gamma\in \Gamma_v^+}  m_\gamma dl_\gamma.
\end{equation*}
The relation of Lemma \ref{relation_trv} is proportional to
\begin{equation*}
\sum_{\gamma\in \Gamma_v^+} \epsilon_v(\gamma) dl_\gamma + dl_{B^+}-dl_{B^-}
\end{equation*}
then we obtain that
\begin{equation*}
\epsilon_v(\gamma)=-m_\gamma
\end{equation*}
Then we get that $m_\gamma=1$ and $\epsilon_v(\gamma)=-1$. Then we see that the multi-curve is primitive and satisfies the third point in Theorem \ref{thm_acycl_curve}. \\ 

To prove the converse statement of Theorem \ref{thm_acycl_curve}, we give the following generalization of the last result: For each subset $I\subset X_0R$, let
\begin{equation*}
\xi_I^+=\sum_{v\in I}\xi_v^+,
\end{equation*}
and $\xi_I^-$ defined in a similar way. Then $\xi_I^+$ satisfies the following properties:
\begin{prop}
\label{acyclicI}
There is a primitive multi-curve $\Gamma_I^+$ such that
\begin{equation*}
\xi_I^+=\xi_{\Gamma_I^+}
\end{equation*}
Moreover, the multi-curve satisfies
\begin{itemize}
\item We have $\Ro_{\Gamma_I^+}=R^{\circ,+}_I\sqcup R^{\circ,-}_{I^c}$ where $X_0 R_I^+=I$,
\item All the curves in $\Gamma_I^+$ rely on $R^{\circ,+}_I$ to $R^{\circ,-}_{I^c}$,
\item $R^{\circ,+}_I$ is glued along negative boundaries.
\end{itemize}
\end{prop}

We can prove the converse statement of Theorem \ref{thm_acycl_curve}.

\begin{prop}
If $\Gamma$ satisfies the properties of Theorem \ref{thm_acycl_curve}, it is $\Gamma_v^+$.
\end{prop}
\begin{rem}
By using the same argument, it is possible to prove that $\Gamma_I^+$ is the unique primitive multi-curve that fills the conditions of Proposition \ref{acyclicI}.
\end{rem}
\begin{proof}
Let $\Gamma$ be a multi-curve that fills the condition of Theorem \ref{thm_acycl_curve}. Then we have the relation on $T_R$.
\begin{equation*}
dl_\Gamma=dl_{B^+}-dl_{B^-}
\end{equation*}
Which means that $\xi_\Gamma$ belongs to $H_R$, and then from Lemma \ref{basis_HR}
\begin{equation*}
\xi_\Gamma=\sum_{u\neq v} \lambda_u \xi_u.
\end{equation*}
Let $\Ro_{\Gao}(c)$ be a connected component of $\Ro_{\Gao}$ that doesn't contain $v$; $\Gao_c$ are the curves in $\Gao$ that are in the boundary of $\Ro_{\Gao}(c)$; and $B_+(c),B_-(c)$ are the other boundaries of $\Ro_{\Gao}(c)$. We have the relation on $T_R$.
\begin{equation*}
    dl_{\Gao_c}= dl_{B_-(c)}-dl_{B_+(c)}
\end{equation*}
and then $\xi_{\Gao(c)}$ is in $H_R$. Moreover, $\xi_{\Gao(c)}$ is in the kernel of $T\cut_{\Gao_c}$, the graph $\Ro_{\Gao(c)}$ has only two connected components, and then the space $\Ker (T\cut_{\Gamma_c})\cap H_R$ is a one-dimensional vector space generated by $\xi_{X_0R_\Gamma(c)}^+$. Then
\begin{equation*}
\xi_{\Gamma(c)}=\lambda_c \xi_{X_0R_\Gamma(c)}^+.
\end{equation*}
If $\Gamma(c)$ is primitive, then $\lambda_c\in \{\pm 1\}$.
We have $-\xi_{X_0R_\Gamma(c)}^+=\xi_{X_0R_\Gamma(c)}^-$ and $R_{\Gamma(c)}$ is glued along positive boundaries, then we conclude from Proposition \ref{acyclicI} that $\lambda_c=-1$. To conclude, we have
\begin{equation*}
\xi_\Gamma=\sum_c \xi_{\Gamma(c)} = -\sum_c \xi_{X_0R_\Gamma(c)}^+
\end{equation*}
because the curves are disjoint. From Lemma \ref{basis_HR}, we have
\begin{equation*}
-\sum_c \xi_{X_0R_\Gamma(c)}^+ =\xi_v^+,
\end{equation*}
And then we conclude the proof: $\xi_\Gamma=\xi_v^+$.
\end{proof}

\paragraph{Extracting a pair of pants on a directed surface:}
\label{acycl_pant}

In the generic case, when there is only vertices of degree $4$, the only minimal ribbon graphs are topological pairs of pants. In this case, minimal and irreducible ribbon graphs coincide, and then Theorem \ref{thm_acycl_decomp} gives a particular family of canonical Fenchel-Nielsen decompositions of a ribbon graph. Let $\Ro$ be a generic directed ribbon graph, an embedded bounded pair of pants \footnote{The terminology is motivated by the third condition; it implies that the length of such a curve is necessarily bounded by a function of the lengths of the boundaries. This implies by Proposition \ref{prop_coord_x} that there is only a finite number of bounded embedded pairs of pants.} in $\Ro$ is an admissible curve $\Gao$ such that:

\begin{itemize}
\item There is a component $c_v$\footnote{These three conditions also imply that the choice of the marked component $c_v$ is canonical. Moreover, there is only one directed ribbon graph on a pair of pants and it contains a unique vertex of order four. Then a bounded pair of pants on a generic directed metric ribbon graph spares a vertex from the rest of the graph.} of $\Go$, which is a pair of pants.
\item All the curves in $\Gamma$ are in the boundary of $c_v$.
\item $c_v$ is glued along it's negative boundaries.
\end{itemize}

\begin{thm}
\label{thm_bounded_pants}
For each generic directed metric ribbon graph $S$ and each vertex $v$, there exists a unique bounded pair of pants $\Gamma^+_v$ that spares $v$ from the rest of the surface.
\end{thm}

A reformulation of Theorem \ref{thm_acycl_decomp} gives the following result:

\begin{cor}
\label{prop_bounded_pants_decomp}
Let $\Ro$ be a generic directed ribbon graph with a linear order. There is a unique directed multi-curve $\Gao$ such that
\begin{enumerate}
\item $\Gao$ is maximal, i.e., components of $\Ro_{\Gao}$ are pairs of pants.
\item The directed stable graph $\Go$ associated with $\Gao$ is acyclic.
\item The linear order on $\Ro$ induces a linear order on $\Go$.
\end{enumerate}
\end{cor}

\newpage

\section{Recursion for volumes of moduli space of directed metric ribbon graphs:}
In this part, we give the proof of Theorem \ref{intro_thm_recurrence_1} given in the introduction. We use Theorem \ref{thm_acycl_curve} and integration formulas to compute integrals over moduli spaces of directed metric ribbon graphs; these formulas were proved in \cite{andersen2020kontsevich} in the case of trivalent graphs and were inspired by the pioneering work of M. Mirzakhani in \cite{mirzakhani2007simple}.

\subsection{Statistics for multi-curves and desintegration of measures:}

 We cover the case of generic directed metric ribbon graphs, but the results are also valid for other strata of the moduli spaces of directed metric ribbon graphs. Let $\Mo\in \bordo$ and $\Gao \in \MS(\Mo)$ we set:
\begin{equation*}
\Tcs_{\Gao}(\Mo)=\{(S,\Gamma)~~|~~ S\in \Tcs(\Mo), \Gamma \in \MF_0(S) ,~~\text{and}~~\Gao=\Gao_S\}.
\end{equation*}
It is the space of generic directed metric ribbon graphs $S$ such that $\Gamma$ is admissible on $S$ and the direction on it, induced by $S$, corresponds to $\Gao$.
 Let $\Go$ be the stable graph of $\Gao$, the lenght of curves in $\Gamma$ defines a map $L_\Gamma: \Tcs_{\Gao}(\Mo)\to \Lambda_{\Go}$. There is an action of $\Stab(\Gao)$ on $ \Tcs_{\Gao}(\Mo)$, and we denote:
\begin{equation*}
\Mcs_{\Go}(\Mo) = \Tcs_{\Gao}(\Mo)/\Stab(\Gao).
\end{equation*}
 The map $L_{\Gamma}$ defines a map $L_{\Go}$ on $\Mcs_{\Go}(\Mo)$ but with values on $\Lambda_{\Go}/\Aut(\Go)$. There is a canonical map:
\begin{equation*}
\pi_{\Go}~:~\Mcs_{\Go}(\Mo) \longrightarrow \Mcs(\Mo).
\end{equation*}
The fiber $\pi_{\Go}^{-1}(S)$ over $S$ is the set $\mathcal{O}_{\Gao}(S)$ of admissible multi-curves on $S$ that are in the orbit of $\Gao$, or equivalently, admissible multi-curves with a directed stable graph given by $\Go$. The map $\pi_{\Go}$ is a covering over each cell, and the space $\Mcs_{\Go}(\Mo)$ is equipped by the pullback of the measure on $\Mcs(\Mo)$.

As in \cite{andersen2020kontsevich}, \cite{andersen2023topological} we consider statistics for the distribution of the length of curves in a multi-curve. If $ F: \Lambda_{\Go}/\Aut(\Go) \to \R_+$ is a continuous function. We define $N_{\Go}F(S)$ as the sum of $F(L_{\Gamma}(S))$ on all admissible multi-curves with stable graph $\Go$:
\begin{equation*}
N_{\Go}F(S)=\sum_{\Gamma\in \pi^{-1}_{\Go}(S)} F(L_{\Gamma}(S)).
\end{equation*}
This is, by definition, the push forward of $F\circ L_{\Gamma}$ under $\pi_{\Go}$. The function $ N_{\Go}F$ is well defined on the moduli space $\Mc(\Mo)$ because of the relation
\begin{equation*}
 F(L_{g\cdot \Gamma}(g\cdot S))=F(L_\Gamma(S)),
\end{equation*}
and the map $g: \MS(R)\longrightarrow \MS(g\cdot R)$ preserves the direction (for $g\in \Mod(M)$). Then it satisfies the following relation, which is the formula for a push forward along a covering:
\begin{equation*}
\int_{\Mcs(\Mo,L)}N_{\Go}F(S)d\mu_{\Mo}(L)=\int_{\Mcs_{\Go}(\Mo,L)}F(L_\Gamma(S))d\mu_{\Mo}(L).
\end{equation*}
The object of this section is to express the RHS of the formula.

\begin{prop}
\label{prop_integral_formula_2}
The integral of $N_{\Go}F(S)$ satisfies the Mirzakhani integral formula \cite{mirzakhani2007simple}:
\begin{equation*}
\int_{\Mc(\Mo,L)}N_{\Go}F(S)d\mu_{\Mo}(L)=\frac{1}{\#\Aut(\Go)}\int_{x\in \Lambda_{\Go}(L)}F(x) \prod_{c\in X_0\G} V_{\Go(c)}(L_c(x)) \prod_{\gamma \in X_1\Go} l_\gamma(x) d\sigma_{\Go}(L).
\end{equation*}
\end{prop}
 Where $L_c : \Lambda_{\Go}\to \Lambda_{\Go(c)}$ is the projection. To prove this formula, we need to increase the space $\Mcs_{\Go}(\Mo,L)$ to obtain complete twist flow and decompose the measures. We construct this space in the next paragraph using spaces of foliations.

\paragraph{Covering and decomposition of the measures:}
\label{para_decomp_measure}
Let $\Mo$ be a directed surface and $\Go$ be a directed stable graph in $\st(\Mo)$. There is a natural bundle over $\Mc(\Go)$ that corresponds to all the possible gluing's of surfaces in $\Mc(\Go)$. If $\Gao$ is a multi-curve that represents $\Go$, we consider the space $\MF_{\Gao}(\Mo)$ of directed foliations transverse to $\Gao$ (see paragraph \ref{gluing_coordinates}). This space carries an action of $\Stab(\Gao)$ the stabilizer of $\Gao$ under the action of the mapping class group $\Mod(M)$. It is possible to consider the quotient.
\begin{equation*}
 B\Mc(\Go)= \MF_{\Gao}(\Mo)/\Stab(\Gao).
\end{equation*}
It does not depend on the choice of $\Gao$, only of $\Go$. The reasons for this choice are the following:
\begin{itemize}
\item The space $\MF_{\Gao}(\Mo)$ contains, in a natural way $\Tc_{\Gao}(\Mo)$, by using the map $\So \to \lambdao_{\So}$ (see Proposition \ref{prop_ribbon_fol}),
\item It is possible to cut an element of $\MF_{\Gao}(\Mo)$ along curves in $\Gamma$, and we obtain an element of $\Tc(\Go)$, this induces a map:
\begin{equation*}
\cut_{\Go}: B\Mc(\Go) \longrightarrow \Mc(\Go).
\end{equation*}
\item Moreover, the twist flows along the curves in $\Gamma$ are complete in $\MF_\Gamma(\Mo)$, which was not the case of $\Tc_{\Gao}(\Mo)$, and this induces an orbifold torus action on $B\Mc(\Go)$.
\end{itemize}
This space is a piecewise linear orbifold with a natural measure $d\tilde{\mu}_{\Go}$ normalized by the set of integer points in the tangent space. Using the results of Proposition \ref{prop_cutting_measure_bundle}, the map
\begin{equation*}
 \cut_{\Go} : B\Mc(\Go)\longrightarrow \Mc(\Go),
\end{equation*}
is an $\R^{\#X_1\G}$ bundle, and for $\gamma \in \Gamma$ the Dehn twist $\delta_\gamma$ acts on the twist parameter by $t_\gamma(\delta_\gamma(S))=t_\gamma(S)+l_\gamma(S)$. And then $B\Tc(\Go)$ is a toric bundle over $\Mc(\Go)$.
\begin{lem}
\label{lem_decomp_measure_l}
The measures $d\tilde{\mu}_{\Go}$ and $d\mu_{\Go}$ satisfy the relation:
\begin{equation*}
(\cut_{\Go})_* d\tilde{\mu}_{\Go} = \prod_{\gamma\in X_1\G} l_\gamma ~~ d\mu_{\Go}
\end{equation*}
\end{lem}

\begin{proof}
This lemma is a consequence of Propositions \ref{prop_tangent_cut}, \ref{prop_cutting_measure_bundle} and Lemma \ref{lem_decompo_measure}. From the two first propositions, we have
\begin{equation*}
(\cut_{\Go})_* d\tilde{\mu}_{\Go}=\text{Vol}(\cut_{\Go}^{-1}(S))d\mu_{\Go}.
\end{equation*}
From the second, the map 
\begin{equation*}
 \cut_{\Go} : B\Mc(\Go)\longrightarrow \Mc(\Go),
\end{equation*}
is a $\R^{\#X_1\G}$ bundle, and for $\gamma \in \Gamma$ the Dehn twist $\delta_\gamma$ acts on the twist parameter by $t_\gamma(\delta_\gamma(S))=t_\gamma(S)+l_\gamma(S)$. And then $B\Tc(\Go)$ is a toric bundle over $\Mc(\Go)$. The fiber over a point is $\prod_\gamma \R/l_\gamma\Z$ with the Haar measure $\prod_\gamma dt_\gamma$. We conclude that
\begin{equation*}
\text{Vol}(\cut_{\Go}^{-1}(S))=\prod_\gamma l_\gamma(S).
\end{equation*}
\end{proof}
    For any stable graph and any positive function $F: \Lambda_{\Go}/\Aut(\Go) \longrightarrow \R_+$. We can consider the integral:
\begin{equation*}
\Vc_{\Go}(F)(L) = \int_{B\Mc(\Go,L)}F(L_{\Go}(S)) d\tilde{\mu}_{\Go}(L).
\end{equation*}
\begin{lem}
\label{lem_integral_formula_1}
For all $\Go$ stable graphs, the function $\Vc_{\Go}(F)(L)$ is given by:
\begin{equation*}
\Vc_{\Go}(F)(L) = \frac{1}{\#\Aut(\Go)|} \int_{x\in\Lambda_{\Go}(L)}F(x) \prod_{c\in X_0\G} V_{\Go(c)}(L_c(x)) \prod_{\gamma\in X_1\Go} l_\gamma(x) d\sigma_{\Go}(L).
\end{equation*}
\end{lem}
\begin{proof}
To prove Lemma \ref{lem_integral_formula_1}, we use  Lemma \ref{lem_decomp_measure_l} and Lemma \ref{lem_decompo_measure}.
\end{proof}

\paragraph{Proof of Proposition \ref{prop_integral_formula_2} and discussion:}
There is a canonical map:
\begin{equation*}
\Mcs_{\Go}(\Mo)\longrightarrow B\Mcs(\Go).
\end{equation*}
This map is not surjective, but Lemma \ref{lem_zero_measure} allows us to avoid this problem.
\begin{lem}
\label{lem_zero_measure}
The subset $\Mcs_{\Go}(\Mo)$ is of full measure in $B\Mcs(\Go)$, and this is also true for $\Mcs_{\Go}(\Mo,L)$ in $B\Mcs(\Go,L)$ for all $L\in \Lambda_{\Mo}$.
\end{lem}
\begin{proof}
We don't give a detailed proof of this lemma; it is based on the fact that foliations in $B\Mcs(\Go,L)\backslash \Mcs_{\Go}(\Mo)$ contain saddle connections. In the directed case, such saddle connections can be generic. But this phenomenon cannot happen if the foliation is transverse to a multi-curve. Then the only configurations of saddle connections that are possible are not generic, and the problem falls in some codimension-one sub-spaces. There is only a countable number of such sub-spaces, and then problems are in a null set\footnote{In the case of the acyclic gluing that we use, the situation is indeed much simpler because the gluing's always create multi-arcs; the space $\MF_{\Gao}(\Mo)$ is included in $\MA_\R(\Mo)$, because gluing cannot create a cycle.}. We refer to \cite{andersen2020kontsevich} for a more detailed proof.
\end{proof}
\begin{proof}[Proposition \ref{prop_integral_formula_2}]
    Using Lemma \ref{lem_zero_measure} and Lemma \ref{lem_integral_formula_1}, we obtain the relation:
\begin{equation*}
\Vc_{\Go}(F)(L)=\int_{\Mcs(\Mo,L)}N_{\Go}F(S)d\mu_{\Mo}(L).
\end{equation*}
Which concludes the proof of Proposition \ref{prop_integral_formula_2}.
\end{proof}

We end this part by a discussion on the case of other strata in the moduli spaces of directed metric ribbon graphs. For $\Mdo$ a decorated directed surface and $\Gdo\in \st(\Mdo)$, we can still consider the statistics of admissible multi-curves with a decoration that coincides with $\Gdo$. In this case we still have the integral formula:
\begin{equation*}
\int_{\Mc(\Mdo,L)}N_{\Gdo}F(S)d\mu_{\Mdo}(L)=\frac{1}{\#\Aut(\Gdo)}\int_{x\in \Lambda_{\Go}(L)}F(x) \prod_c V_{\Gdo(c)}(L_c(x)) \prod_\gamma l_\gamma(x) d\sigma_{\Go}(L).
\end{equation*}
And the decomposition of measures remains valid in this case.

\subsection{Recursion for volumes of moduli spaces}
\paragraph{A degenerated geometric recursion formula:}
We denote $P(\Mo)$ the set of directed stable graphs on $\Mo$ that correspond to bounded pairs of pants (see Theorem \ref{thm_bounded_pants}). There are four families of such stable graphs, which are represented in figure \ref{fig_pant_gluing} and are also given in the introduction:
\begin{enumerate}
 \item We remove a pants that contains two positive boundaries $i,j$.
 \item We remove a pants that contains a positive boundary $i$ and a negative boundary $j$.
 \item We remove a pants with one positive boundary $i$ that is connected to the surface by two negative boundaries, and do not disconnect the surface.
 \item We remove a pants with one positive boundary $i$, which is connected to the surface by two negative boundaries, and disconnect the surface into two connected components.
\end{enumerate}
Then we can rewrite Theorem \ref{thm_bounded_pants} in the following way: it is a kind of degeneration of the geometric recursion formula \cite{andersen2017geometric}. This formulation is straightforward to integrate over the moduli space $\Mc(\Mo)$ by using Proposition \ref{prop_integral_formula_2}.

\begin{prop}
\label{prop_mirzmacshane_1}
For all $S\in \Mcs_{g,n^+,n^-}(L)$, we have:
\begin{equation*}
2g-2+n^++n^-=
\sum_{\Go \in P(\Mo)} N_{\Go}(1)(S).
\end{equation*}
\end{prop}

\begin{proof}
Theorem \ref{thm_bounded_pants} can be stated in the following way: for each four-valent directed ribbon graph with at least two vertices, there is a bijection between the set of admissible multi-curves with stable graph in $P(\Mo)$ and the set of vertices of the graph. There is $2g-2+n^++n^-$ vertices, and then we obtain the formula because the right-hand side is the number of admissible primitive multi-curves with a directed stable graph in $ P(\Mo)$.
\end{proof}

\begin{figure}
    \centering
    \includegraphics[width=0.5\linewidth]{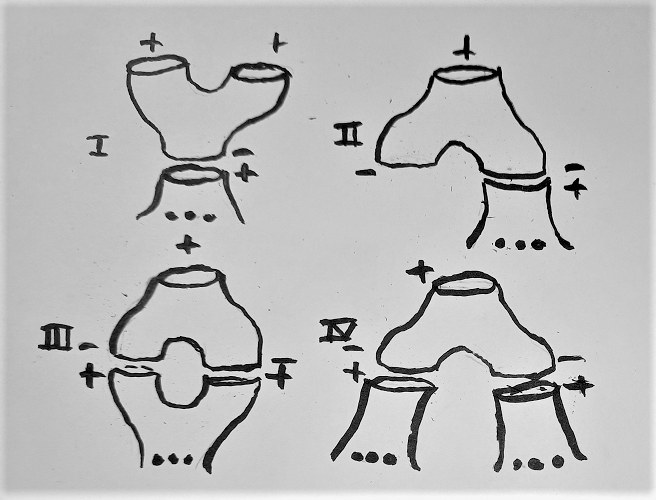}
    \caption{Different gluings that appear in the recursion.}
    \label{fig_pant_gluing}
\end{figure}

\paragraph{Recursion for volumes:}
In this part we prove Theorem \ref{intro_thm_recurrence_1} given in the introduction. We use the integral formula of Proposition \ref{prop_integral_formula_2} to integrate the formula obtained in Proposition \ref{prop_mirzmacshane_1}.
\begin{thm}
\label{thm_recurrence_1}
For all values of the boundary lengths $L=(L^+,L^-)$, the functions $V_{g,n^+,n^-}$ satisfy the recursion:
\begin{eqnarray*}
(2g-2+n)V_{g,n^+,n^-}(L^+|L^-)&=& \frac{1}{2}\sum_{i\neq j} (L_i^+~+L_j^+)~V_{g,n^+-1,n^-}(L_i^+ +L_j^+,L^+_{\{i,j\}^c}|L^-)\\
&+&\sum_{i}\sum_{j} [L_i^+-L_j^-]_+~V_{g,n^+,n^--1}([L_i^+-L_j^-]_+,L^+_{\{i\}^c}|L^-_{\{j\}^c})\\
&+&\frac{1}{2}\sum_{i} \int_0^{L_i^+} V_{g-1,n^++1,n^-}(x,L_i^+-x,L^+_{\{i\}^c}|L^-)~x(L_i^+-x)~ dx\\
&+&\frac{1}{2}\sum_{i}\sum_{\underset{I_1^{\pm}\sqcup I_2^{\pm}=I^{\pm}}{g_1+g_2=g}}^{'}  x_1 x_2 V_{g_1,n_1^++1,n^-_1}(x_1,L^+_{I_1^+}|L^-_{I_1^-})~ V_{g_2,n_2^++1,n^-_2}(x_2,L^+_{I_2^+}|L^-_{I_2^-});
\end{eqnarray*}
where we use the notation\footnote{Notations relative to indices are recalled in paragraph \ref{paragraph_notation_indices}.}
\begin{equation*}
x_l = |L_{I^{-}_l}^{-}|_1-|L_{I^{+}_l}^{+}|.
\end{equation*}
\end{thm}
The index $'$ above the sum means that we sum over all the stable pairs $(g_i,n^+_i,n^-_i)$ with $2g_i+n^+_i+n^-_i-2>0$.
\begin{proof}
To obtain Theorem \ref{thm_recurrence_1}, we multiply the formula of Proposition \ref{prop_mirzmacshane_1} by  $d\mu_{\Mo}(L)$ and integrate over the moduli space $\Mc_{g,n^+,n^-}(L)$. By applying Proposition \ref{prop_integral_formula_2}, we obtain the identity:
\begin{equation*}
(2g-2+n)V_{\Mo}= \sum_{\Go \in P(\Mo)} \Vc_{\Go}(1).
\end{equation*}
The covering of bounded pairs of pants splits into several coverings, which correspond to the four types of gluings with all the possible choices of boundaries. Now, from the results of the last section, we have
\begin{equation*}
\Vc_{\Go}(1)(L)=\frac{1}{\#\Aut(\Go)}\int_{x\in \Lambda_{\Go}(L)} \prod_{c} V_{\Go(c)}(L_c(x)) \prod_\gamma l_\gamma(x) d\sigma_{\Go}(L).
\end{equation*}
Then there is a distinguished component $c_0$ in $X_0\Go$ which is a pair of pants glued along its negative boundaries, which is either of type $(0,2,1)$ or $(0,1,2)$. The volumes $V_{0,2,1},V_{0,1,2}$  are constant, equal to one. The reason is that, there is only one directed graph on a directed pair of pants. Then this term disappears in the formulas. To finish the proof, it remains to compute the domain of integration in each case. Except in the case of type $III$ the directed stable graphs are trees. As we see in section \ref{paragraph_trees} in these cases the domain of integration  $\Lambda_{\Go}(L)$ is a point, then the integration is trivial, and we can obtain in this way the lines $1,2,4$ of the formula in Theorem \ref{thm_recurrence_1}. For the third line, we have $\Lambda_{\Go}(L)=\{(l_1,l_2)\in \Rp^2|l_1+l_2=L\}$.
\end{proof}

\paragraph{Case of marked points:}

If $\Mo$ has $m$ marked points that are not labeled, as we explain, a generic directed ribbon graph is then a quadrivalent graph with bivalent vertices at the marked points. In the case where the boundaries are labeled, we denote
\begin{equation}
V_{g,n^+,n^-,m}(L^+|L^-),
\end{equation}
the volume of the moduli space. We also denote $E:\Lambda_{n^+,n^-}\to \R$ the function 
\begin{equation*}
    E(L)=| L^+|_1=|L^-|_1.
\end{equation*}
Then, using the Theorem \ref{thm_acycl_curve} applied to the bivalent vertices, we can derive the following proposition: 
\begin{prop}
\label{prop_recurssion_marked_points}
The functions $V_{g,n^+,n^-,m}$ satisfy
\begin{equation*}
m V_{g,n^+,n^-,m}= E~V_{g,n^+,n^-,m-1},
\end{equation*}
then
\begin{equation*}
 V_{g,n^+,n^-,m} = \frac{E^m}{m!} V_{g,n^+,n^-}.
\end{equation*}
\end{prop}
\begin{proof}
This is also a consequence of Theorem \ref{thm_acycl_curve}. We can apply it to vertices of degree two. In this case, there is only one minimal directed ribbon graph; it corresponds to a ribbon graph on a cylinder with one vertex of degree two and only one edge. The recursion is the following:
\begin{equation*}
mV_{g,n^+,n^-,m} = \sum_i L_i^+ V_{g,n^+,n^-,m-1}.
\end{equation*}
Which gives a proof of Proposition \ref{prop_recurssion_marked_points}.
\end{proof}

\subsection{Graphical expansion and structure of the volumes}
\paragraph{Graphical expansion:}
\label{paragraph_graphical_expension}
As in the case of the topological recursion, the formula  in Theorem \ref{thm_recurrence_1} admits a graphical expansion. It can be obtained by iterating the recursion. Proposition \ref{prop_bounded_pants_decomp} gives canonical maximal acyclic multi-curves that decompose the ribbon graph, in the case of a generic directed ribbon graph, the decomposition is a pants decomposition.

\begin{prop}
\label{prop_graph_expension}
The volume $V_{g,n^+,n^-}$ satisfies:
\begin{equation*}
(2g-2+n)!~V_{g,n^+,n^-}(L)=\sum_{\Go}\frac{n_{\Go}}{\#\Aut(\Go)}\int_{x\in\Lambda_{\Go}(L)}\prod_\gamma l_\gamma(x) d\sigma_{\Go}(L)=\sum_{\Go\in \acyclos}n_{\Go} V_{\Go}(L).
\end{equation*}
The sum over all the directed acyclic pants decompositions of genus $g$ and with $n^+$ positives and $n^-$ negatives labeled boundaries and $n=n^++n^-$, $n_{\Go}$ is the number of linear orders on the acyclic stable graph.
\end{prop}

\begin{proof}
We can use the Theorem \ref{prop_bounded_pants_decomp}, there is a correspondence between enumeration of the vertices and multi-curves in $\MS(\Ro)$ with an acyclic stable graph and a linear order. If $\Ro$ is four-valent, a maximal acyclic curve is a pant decomposition? and then for each $S\in \Mcs(\Mo)$ we have:
\begin{equation*}
(2g-2+n^++n^-)!= \sum_{\Go\in \acyclos} n_{\Go} N_{\Go}(1)(S).
\end{equation*}
Then, by taking the integral over $\Mc(\Mo,L)$ and applying Proposition \ref{prop_integral_formula_2}, we obtain the formula:
\begin{equation*}
(2g-2+n^++n^-)!~V_{g,n^+,n^-}=\sum_{\Go} n_{\Go} V_{\Go}.
\end{equation*}
\end{proof}

\paragraph{Piece-wise polynomiality:}

In this part, we obtain result on piece-wise polynomiality for the function $V_{g,n^+,n^-}$. We could use the recursion, but we use results on acyclic directed stable graphs obtained in Theorem \ref{thm_polynomial_ZGo}. The structure of the volumes is similar to the one of double Hurwitz numbers studied in \cite{hahn2018wallcrossing}.

\begin{thm}
\label{thm_polynomial_Vgnn}
The volume $V_{g,n^+,n^-}(L^+|L^-)$ is an element of $\mathcal{P}_{n^+,n^-}$. It is a continuous piece-wise polynomial of degree $4g-3+n^++n^-$.
\end{thm}

\begin{proof}
 We have already got all the ingredients for the proof. From Theorem \ref{thm_polynomial_ZGo}, we know that the volumes $V_{\Go}$ are continuous piece-wise polynomials in $\mathcal{P}_{n^+,n^-}$ homogeneous of degree $4g-3+n^++n^-$. Moreover, by using Proposition \ref{prop_graph_expension}, we deduce Theorem \ref{thm_polynomial_ZGo} by linearity.
\end{proof}

\subsection{General formula for higher volumes}

Assumes that $\Mdo=(\Mo,\nu)$ is a decorated directed surface. We consider the volumes:
\begin{equation*}
V_{\Mdo}(L)=\int_{\Mcs(\Mdo,L)}d\mu_{\Mdo}(L).
\end{equation*}
In this part, we state the recursion in full generality using  Theorem \ref{thm_acycl_curve}. We plan to make it more explicit in future work. We define $\text{Bi}_i(\Mdo)$ as the set of decorated acyclic stable graphs $\Gdo$ with a height
\begin{equation*}
h~:~X_0\G\longrightarrow \{0,1\}.
\end{equation*}
A height is a function that preserves the order of the acyclic graph. We assume that $h$ satisfies the following conditions:
\begin{itemize}
\item $\Gdo\in \st(\Mdo)$.
\item $h$ is strictly increasing for the order relation on the graph.
\item There is a unique component $c_1$ with $l(c_1)=1$, $\Gdo(c_1)$ is minimal and $\nu(c_1)=(i)$.
\end{itemize}
The last condition says that there is a unique component at the top of the graph, it is minimal with a vertex of degree $2i+2$. It is glued along its negative boundaries (because it is at the top of an acyclic graph). Then, we can rewrite Theorem \ref{thm_acycl_curve} in the following way: For each $\So=(\Ro,m)$ and each $v\in X_0R$ of degree $2i+2$, there is a unique admissible curve $\Gado$ such that $\Gdo_{\Gado}\in \Bil_i(\Mdo)$. And $\Gado$ spares $v$ from the rest of the surface. Then we can derive the following corollary that generalizes Proposition \ref{prop_mirzmacshane_1}.

\begin{prop}
\label{prop_mirzmacshane_general}
For each $\Mdo$ and each $i$, we have for all $S\in \Mcs(\Mdo)$:
\begin{equation*}
\nu(i)=\sum_{\Gdo\in\text{Bi}_{i}(\Mdo)}N_{\Gdo}(1).
\end{equation*}
\end{prop}

And from this, using integration over the moduli space, we can obtain the following Theorem.

\begin{thm}
\label{thm_recurssion_general}
The volumes satisfy the recursion:
\begin{equation*}
\nu(i)V_{\Mdo}(L^+|L^-)=\sum_{\Gdo\in \text{Bi}_{i}(\Mdo)} \frac{1}{\#\Aut(\Gdo)} \int_{x\in\Lambda_{\Go}(L)} V_{\Gdo(1)}(L_{1}(x))V_{\Gdo(0)}(L_0(x))\prod_\gamma l_\gamma(x) d\sigma_{\Go}(L).
 \end{equation*}
\end{thm}

In this formula, $V_{\Gdo(1)}$ corresponds to the volume associated to the minimal component and $V_{\Gdo(0)}$ to the volume of all the other components. The formula gives a recursion that computes $V_{\Mdo}$ by removing vertices of degree $2i+2$, it allows to compute the functions $V_{\Mdo}$ using the volumes of minimal graphs only. The possible terms that appear in the formula for $i=2$ are given in figure \ref{fig:gluing_6}; in the case of $i=1$ we obtain all the ways to remove a directed pant, i.e., Theorem \ref{thm_recurrence_1}; and for $i=0$, we remove a cylinder with one marked point, which is Proposition \ref{prop_recurssion_marked_points}. As before, we can also obtain an expansion in terms of acyclic stable graphs. Let $\acyclodmax(\Mdo)$ be the set of decorated acyclic stable graphs on $\Mdo$, such that all components are minimal. This is also the set of maximal acyclic decompositions of $\Mdo$. An before $n_{\Go}$ denotes the number of linear orders. Let
\begin{equation*}
V_{\Gdo}(L)=\frac{1}{\#\Aut(\Gdo)}\int_{x\in\Lambda_{\Go}} \prod_c V_{\Gdo(c)}(L_c(x))\prod_\gamma l_\gamma d\sigma_{\Go}(L).
\end{equation*}
Then we have the following proposition.
\begin{prop}
We have the following relation
\begin{equation*}
n(\nu)! V_{\Mdo}(L)=\sum_{\Gdo\in \acyclodmax(\Mdo)}n_{\Go} V_{\Gdo}(L).
\end{equation*}
\end{prop}

\begin{figure}
\centering
\includegraphics[width=15cm]{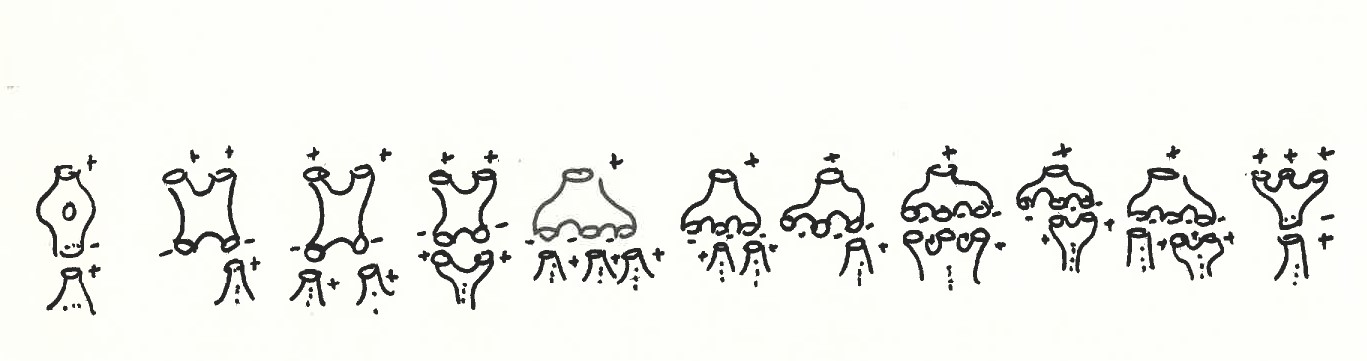}
\caption{Different gluings that appear when we remove a vertex of degree $6$.}
\label{fig:gluing_6}
\end{figure}

\newpage

\section{Surfaces with one negative boundary component:}

In this section, we study in more detail the case of surfaces with only one negative boundary. In this case, the recursion of Theorem \ref{thm_recurrence_1} takes a much simpler form, and it is possible to relate it to Cut-and-Join equations.

\subsection{Recursion for the volumes}
\label{subsection_one_bound}
First of all, if $n^-=1$, the first projection allows us to identify $\Lambda_{n^+,1}\simeq (\Rp)^{n_+}$,
 then it is possible to drop the negative boundaries and write $V_{g,n^+,1}(L^+|L^-)$ as a function of $L^+$. We write:
\begin{equation*}
\Vo_{g,n^+}(L)=V_{g,n^+,1}(L||L|_1),
\end{equation*}
where $\Vo_{g,n^+}(L)$ is a function on $\Rp^{n^+}$. We can see that the recursion of Theorem \ref{thm_recurrence_1} preserves the family of surfaces with one negative boundary component. Extracting a bounded pair of pants on a surface with only one negative boundary creates only surfaces of this type. Moreover, these gluings are necessarily non-separating and cannot be of type $II$. Then the recursion takes the following form:

\begin{thm}
\label{thm_recursion_1_bord}
Functions $\Vo_{g,n}$ are homogeneous polynomials of degree $4g-2+n$ and satisfy the following recursion:
\begin{eqnarray}
\label{formula_recurence_1_bord}
(2g+n-1)\Vo_{g,n}(L)&=& \frac{1}{2}\sum_{i\neq j} (L_i~+L_j)~\Vo_{g,n-1}(L_i +L_j,L_{\{i,j\}^c})\\
&+&\frac{1}{2}\sum_{i} \int_0^{l_i} \Vo_{g-1,n+1}(x,L_i-x,L_{\{i\}^c})~x(L_i-x)~ dx
\end{eqnarray}
\end{thm}
Using this theorem, we can deduce $\Vo_{g,n}(L)$ from the case $\Vo_{0,2}(L_1,L_2)=1$.
\begin{proof}
To prove Theorem \ref{thm_recursion_1_bord}, we remark that when we apply Theorem \ref{thm_bounded_pants} to a surface with only one boundary, surgeries that are allowed are necessarily of type $I$ or $III$; in other cases, a component of the surface must contain only positive boundary components, which is impossible. Reciprocally performing gluings of type I or III preserve the sub-family of surfaces with only one negative boundary component, then we can deduce the Formula \ref{formula_recurence_1_bord}. The form of the recursion preserves the space of polynomials, the first line increases the degree by one and preserves the homogeneity, the second increase the degree by $3$. In both cases, we have 
\begin{equation*}
(4g-2+n-1)+1=4g-2+n ~~~\text{and}~~~(4(g-1)-2+n+1) + 3 = 4g-2+n.
\end{equation*}
Then, by induction, we obtain that $\Vo_{g,n}$ is homogeneous of degree $4g-2+n$.
\end{proof}
\paragraph{String and dilaton equations:}
Functions $\Vo_{g,n}$ satisfy two series of equations that are similar to string and dilaton equations.
\begin{prop}
\label{prop_dilaton_F}
The volume $\Vo_{g,n+1}$ satisfies the following two relations:
\begin{itemize}
    \item \textit{String} equation: $~.\Vo_{g,n+1}(0,L)=|L|_1 \Vo_{g,n}(L)$
    \item \textit{Dilaton} equation: $~\frac{\partial \Vo_{g,n+1}}{\partial L_1}(0,L)= (2g+n-1) \Vo_{g,n}(L).$
\end{itemize}
\begin{equation*}
\frac{\partial \Vo_{g,n+1}}{\partial L_1}(0,L)= (2g+n-1) \Vo_{g,n}(L).
\end{equation*}
\end{prop}
\begin{proof}
    This \textit{string} equation is obtained by computing the volume of the space of a ribbon graph with only one edge in the first boundary component $(\alpha_1^+=1 )$. By looking at the contribution of ribbon graphs with at most two edges in the first boundary component, it is possible to compute the first order of $V_{g,n+1}$ at $L_1=0$, and obtain the \textit{dilaton} equation.
\end{proof}

\paragraph{Computation of $\Vo_{0,n}$:}
A particular case is one of the spheres. Rewriting the last formula, we obtain the following corrolary:
\begin{cor}
The volumes $\Vo_{0,n}(L)$ satisfy the recursion:
\begin{equation}
\label{formula_recurssion_1_bord_sphere}
(n-1) \Vo_{0,n}(L) = \sum_{i<j}(L_i+L_j)\Vo_{0,n-1}(L_i+L_j,L_{\{i,j\}}).
\end{equation}
And then they are given by: $\Vo_{0,n}(L)= |L|_1^{n-2}$.
\end{cor}
Then, in this case, $\Vo_{0,n}$ is one of the simplest homogeneous polynomials. 
\begin{proof}
The recursion is a consequence of Theorem \ref{thm_recursion_1_bord}; in the case of the sphere, gluings of type $III$ are not allowed, and then we obtain Formula \ref{formula_recurssion_1_bord_sphere}. To obtain the explicit formula, we can see that the recursion determines $\Vo_{0,n}(L)$ for all $n$ with the initial condition $\Vo_{0,2}=1$. Then, by a direct computation, we can see that the functions $|L|^{n-2}$ satisfy the recursion and also the initial condition.
\end{proof}
\begin{rem}
   For genus $0$, surfaces with only one positive boundary are also preserved by the recursion along positive boundaries. We have:
\begin{equation*}
(n-1) \Vo_{0,1,n}(L_1^+|L^-)= \frac{1}{2}\sum_{I_1^-,I_2^-} |L_{I_1}^-|_1 |L_{I_2}^-|_1\Vo_{0,1,n_1}(|L_{I_1}^-|_1|L_{I_1}^-)\Vo_{0,1,n_2}(|L_{I_2}^-|_1|L_{I_2}^-).
\end{equation*}
Using the relation $\Vo_{0,1,n}(|L|_1|L)=\Vo_{0,n,1}(L||L|_1)=\Vo_{0,n}(L)$, it is possible to obtain another recurrence relation for the polynomials $\Vo_{0,n}(L)$:
\begin{equation*}
(n-1) \Vo_{0,n}(L) = \frac{1}{2}\sum_{I_1,I_2}|L_{I_1}|_1 |L_{I_2}|_1\Vo_{0,n_1+1}(L_{I_1})\Vo_{0,n_1+1}(L_{I_2}).
\end{equation*}
\end{rem}

\paragraph{ Cut-and-Join equation:}

In this section, we investigate the recursion obtained for coefficients of the polynomials $\Vo_{g,n}(L)$. We can write
\begin{equation*}
\Vo_{g,n}(L)=\sum_{\alphab} c_{g,n}(\alphab)L^{\alphab},~~~\text{with}~~~L^{\alphab}=\prod_i L_i^{\alpha_i},
\end{equation*}
where $\alphab=(\alpha_1,...,\alpha_n)$
\begin{equation*}
d(\alphab)=4g-2+n~~~\text{and}~~~n(\alpha)=n.
\end{equation*}
Then it is possible to drop indices $(g,n)$ in the notation. By using Formula \ref{formula_recurence_1_bord}, we derive the following relation for the coefficients:
\begin{prop}
\label{prop_coef_c_1_bord}
The coefficients $(c(\alphab))_{\alphab}$ satisfy the following recursion:
\begin{eqnarray}
\label{formula_recurssion_c}
(2g-1+n)c(\alphab)&=& \frac{1}{2}\sum_{i\neq j}\frac{(\alpha_i+\alpha_j)!}{\alpha_i!\alpha_j!} c(\alpha_i+\alpha_j-1,\alphab_{\{i,j\}^c})\\
&+&\frac{1}{2}\sum_{i,x_1+x_2=\alpha_i-3} \frac{(x_1+1)!(x_2+1)!}{\alpha_i!} c(x_1,x_2,\alphab_{\{i\}^c}).
\end{eqnarray}
\end{prop}
The coefficients $c(\alphab)_{\alphab}$ satisfy an important symmetry; they are invariant under permutations. Then, as for the intersection numbers of the tautological class over the moduli space, it is possible to see $c$ as a function on the set of generalized partitions; we write $c(\mu)$ where $\mu=(\mu(0),\mu(1),...)$. Then we can consider the following formal series with infinitely many variables $\tb=(t_0,t_1,t_2,...)$
\begin{equation*}
\Zo(q,\tb)=\sum_\mu \frac{q^{\frac{d(\mu)+n(\mu)}{2}}\prod_i (i!) ^{\mu(i)}t_i^{\mu(i)}}{\prod_i \mu(i)!} c(\mu).
\end{equation*}
From Proposition \ref{prop_coef_c_1_bord}, we obtain the following result:
\begin{thm}
\label{cor_cutandjoin_1_bord}
The series $\Zo(q,\tb)$ satisfies the following equation:
\begin{equation}
\label{formula_cut_and_join_Zo}
\frac{\partial \Zo}{\partial q} = \frac{1}{2}\sum_{i,j}(i+j)t_it_j \partial_{i+j-1} \Zo+ \frac{1}{2}\sum_{i,j}(i+1)(j+1)t_{i+j+3} \partial_i\partial_j \Zo + \frac{t_0^2}{2},
\end{equation}
with $\Zo(0,\tb)=0 $.
\end{thm}
\begin{rem}[Dependance in $q$]
    The variables $q$ and $\tb$ are not independent; we have the relation
\begin{equation*}
\Zo(q,t)=\Zo(1,t(q))=\Zo(t(q)),~~~\text{with}~~~t_i(q)=q^{\frac{i+1}{2}}t_i.
\end{equation*}
And then, by taking the derivative in $q$ and evaluating at $q = 1$, we have:
\begin{equation*}
\frac{\partial \Zo}{\partial q}= \sum_i \frac{i+1}{2} t_i \partial_i \Zo.
\end{equation*}
Then we can obtain the following reformulation of Formula \ref{formula_cut_and_join_Zo}:
\begin{equation*}
\sum_i (i+1) t_i \partial_i\Zo = \sum_{i,j}(i+j)t_it_j \partial_{i+j-1} \Zo + \sum_{i+j}(i+1)(j+1)t_{i+j-3} \partial_i\partial_j \Zo + \frac{t_0^2}{2}.
\end{equation*}
\end{rem}

\subsection{Dual problem and Hurwitz number:}

Let ${\Ro}^{\nu}_{g,n^+,n^-}(\alphab^+|\alphab^-)$ be the number of isomorphism classes of directed ribbon graphs $\Ro$ of type $(g,n^+,n^-)$ such that:
\begin{itemize}
\item The orders of the vertices are prescribed by $\nu$ ($\nu_{\Ro}=\nu$).
\item The perimeter of the positive (resp. negative) boundary components is given by $\alphab^+$ (resp. $\alphab^-$).
\end{itemize}
We weight each ribbon graph $\Ro$ by $\frac{1}{\#\Aut(\Ro)}$ and we simply denote $\Ro_{g,n^+,n^-}(\alphab^+|\alphab^-)$ the number of generic directed ribbon graphs (with only quadrivalent vertices). On the other hand, we denote $h_{g,n^+,n^-}^\nu(\alphab^+|\alphab^-)$ the Hurwitz number of coverings over the sphere ramified over three points $(x_0,x_+,x_-)$ and such that:
\begin{itemize}
\item There is $\nu(i)$ ramification of order $i$ over $x_0$.
\item There are $n^-$ ramifications
over $x_-$ and $n^+$ over $x_+$ and they are labeled and of order $\alphab_i^{\pm}-1$.
\end{itemize}
As before, we weight a cover by the inverse of the size of its automorphism group. These covers are called dessins d'enfants and have been studied at many places, $h_{g,n^+,n^-}(\alphab^+|\alphab^-)$ correpsond to covers with simple ramifications over $x_0$.
\begin{lem}
\label{lem_hurwitz_graph_orient}
We have the following equality:
\begin{equation*}
{\Ro}^{\nu}_{g,n^+,n^-}(\alphab^+|\alphab^-)=h_{g,n^+,n^-}^{\nu}(\alphab^+|\alphab^-).
\end{equation*}
\end{lem}
We give a pictorial explanation in figure \ref{figure_oriented_dessin_denfants}.
\begin{figure}
\centering
\includegraphics[height=7cm]{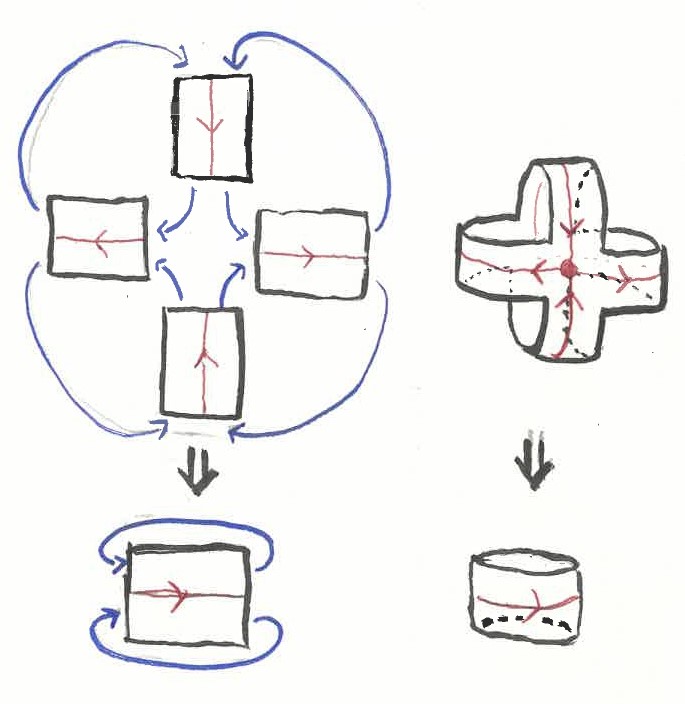}
\caption{From directed ribbon graphs to covers.}
\label{figure_oriented_dessin_denfants}
\end{figure}

\begin{proof}
Let $\Ro$ be a directed ribbon graph according to construction \ref{zip_rect_RG} $\Ro$ is made by gluing rectangles $\mathbf{R}_e^{\bullet}$, where $e\in X^+\Ro$ is a positive half edge. From Proposition  \ref{prop_coord_x} we have a canonical one form $\omega_R$ given locally by $dz$ on each $\mathbf{R}_e$. If we fix a vertex $v\in X_0R$ of the graph, the period map:
\begin{equation*}
z\longrightarrow \int_v^z \omega_R,
\end{equation*}
is well defined on the universal cover of $M_R^{\bullet}$ and the image of the fundamental group is contained in $\Z\subset \C$, then by taking the quotient we have a well defined map:
\begin{equation*}
M_R^{\bullet} \longrightarrow \C/\Z.
\end{equation*}
The target is an infinite cylinder (sphere with two marked points). Under this map, $\omega_R$ is given by the pullback of $\omega_0=dz$ on $\C/\Z$. When we pull back a differential $z^ldz$ by the map $\phi =z^k$, we obtain $\phi^*w=kz^{(l+1)k-1}dz$; moreover, if $l=-1$, we also have $\Res_0 \phi^*w =k\Res_0 w$. According to this, a vertex of degree $2i+2$ is mapped to $0\in \C/\Z$ and corresponds to a zero of degree $i$ of $\omega_R$. Then we see that the map $\phi_R$ is ramified of degree $i$ at this point. We can see $\C/\Z$ as a sphere with two removed points at $x_+,x_-$ and $\omega_0$ as two simple poles at these points; under our assumption, the positive boundary components are mapped to the pole at $x_+$ (resp. $x_-$). Their preimages are simple poles of $w_R$, and the absolute value of the residue at a pole corresponds to the length of the boundary component; it is also equal to the number of edges that it contains. On the other hand, for all coverings ramified over $x_0,x_-,x_+$, it is possible to obtain a ribbon graph on the surface by looking at the pre-image of the circle based at $x_0$; the orientation on the circle induces an orientation on the graph (see figure \ref{figure_oriented_dessin_denfants}).
\end{proof}

Then directed generic ribbon graphs with one negative boundary correspond to dessins d'enfants with a maximal ramification over $x^-$ and $2g-2+n$ simple ramifications over $0$. Let $h_{g,n^+,1}(\alphab^+)$ be the corresponding Hurwitz number. We assume that the ramifications over the first point are labeled. The last lemma and an explicit computation of the volumes in this case allow us to write the following formula:

\begin{cor}
\label{cor_Vo_hurwitz}
The volumes $\Vo_{g,n}$ are polynomials that are naturally related to Hurwitz numbers in the following way:
\begin{equation*}
\Vo_{g,n}(x_1,....,x_n) = \sum_{\alphab} h_{g,n,1}(\alphab)\prod_i \frac{x_i^{\alpha_i-1}}{(\alpha_i-1)!}.
\end{equation*}
\end{cor}

\begin{proof}
To prove Corrolary \ref{cor_Vo_hurwitz}, we compute the volume $V_{\Ro}(L)$ associated with a directed ribbon graph with a single negative boundary component. This is an integral over some affine subspace in $\Met(R)$. The linear relations on $\Met(R)$ are given by $L_i^{+}(m)=L_i$, for each $i=1,...,n^+$. As the graph is directed, the dual is bipartite, and then an edge of the graph appears in exactly one of these equations with a weight equal to $1$. In other words, there is an identification:
\begin{equation*}
\Met(\Ro,L,|L|_1) =\prod_{i} \{(m_e)_{[e]_2=\beta^+_i}|\sum_e m_e=L_{i}\}.
\end{equation*}
Then the volume is equal to the volume of a product of simplices. Each factor is equipped with the affine measure and
\begin{equation*}
\int_{\sum_1^n y_j=x}d\sigma= \frac{x^{n-1}}{(n-1)!}.
\end{equation*}
Then, the affine volume of $ \Met(\Ro,L,|L|)$ is
\begin{equation*}
\prod_i \frac{L_i^{\alpha_i^{+}(\Ro)-1}}{(\alpha_i^{+}(\Ro)-1)!},
\end{equation*}
where $\alpha_i^{+}(\Ro)$ is the number of edges in the boundary $\beta_i^+$. By summing the contribution of all directed generic ribbon graphs, we obtain:
\begin{equation*}
\Vo_{g,n}= \sum_{\Ro} V_{\Ro}=\sum_{\Ro} \frac{1}{\#\Aut(\Ro)} \frac{L_i^{\alpha_i^+(\Ro)-1}}{(\alpha_i^+(\Ro)-1)!}.
\end{equation*}
The coefficient in front of $ \prod_i \frac{L_i^{\alpha_i-1}}{(\alpha_i-1)!}$ is the number of quadrivalent ribbon graphs with $\alpha_i$ edges on the $i-$ieme positive boundary components, counted with automorphisms. By using Lemma \ref{lem_hurwitz_graph_orient}, we can conclude the proof of Corrolary \ref{cor_Vo_hurwitz}
\end{proof}

\newpage

\appendix

\section{Vocabulary and notations:}
\paragraph{Multi-graphs:}
\label{paragraph_graph}
In this text, we consider multi-graphs (see \cite{bollobas2012graph}), we generally call them graphs by abuse of language. A multi-graph $G$ can be defined by data $(X_0G,XG,(XG(c))_{c\in X_0G},s_1)$. Where $XG$ is the set of half edges, $X_0G$ is the set of vertices, $(XG(c))_{c\in X_0G}$ is a partition of $XG$, indexed by the vertices, and $s_1$ is an involution that encodes how to glue two half edges together. A multi-graph can have multiple edges, loops, and also legs, which are the fixed points of $s_1$. The quotient $X_1G=XG/\langle s_1\rangle$ corresponds to the set of edges of the graph. We denote $\partial G$ the legs and 
 $X^{int}G$ is the set of internal edges, which are the orbits of order $2$. For each $e\in XG$ we denote $[e]_0$ and $[e]_1$ the projections on $X_0G$ and $X_1G$. Two graphs $G,G'$ are isomorphic iff there is a bijection $\phi: XG\rightarrow XG'$ that preserves the two partitions and satisfies $s_1'\circ \phi=\phi\circ s_1$. We always assume that an automorphism fix the legs of the graph.\\

If $G$ is a graph and $E$ a subset of $X_1^{int}G$, we can define $G_{\langle E\rangle}$ the quotient, by identifying two vertices that are joined by an edge in $E$ and removing these edges. \footnote{Formally, let $\tilde{E}\subset XG$ the subset such that $s_1(\tilde{E})=\tilde{E}$ and $\tilde{E}/\langle s_1\rangle=E$. We consider the equivalence relation $\sim_E$ on $X_0G$ generated by the following symmetric relation:
\begin{equation*}
    c_1\sim_E' c_2 \Leftrightarrow \exists~
 e_1,e_2\in \tilde{E}~\text{such as}~ e_i \in XG(c_i)~~\text{and}~~e_1=s_1(e_2).
\end{equation*}
Then, we denote $X_0G_{\langle E\rangle }=X_0G/\sim_E$ the quotient. For each $c\in X_0G_{\langle E\rangle }$ we consider the subset  $XG_{\langle E\rangle }(c)=\sqcup_{c'\sim_{ E } c} XG(c')\backslash \tilde{E}$, and we take the restriction of $s_1$ to $XG_{\langle E\rangle}=XG \backslash \tilde{E}$.}\\

For $E\subset X_1G$ we can also define the graph $G_E$ by removing the edges in $E$ and deleting the vertices with no edges or half edges.

\paragraph{Measure on affine spaces:}
\label{paragraph_measure_convex}
We generally consider a vector space $V$ of dimension $n$, and  a convex polytope $X\subset V$. A polytope is a subspace defined by a finite number of linear inequalities:
\begin{equation*}
    X=\bigcap_{i=1}^r\{x\in~V~,~l_i(x)\ge b_i\}.
\end{equation*}
With $l_i\in V^*$ and $b_i\in\R$ for all $i$, a polytope is open if the inequalities are strict. We recommend \cite{barvinok2008integer}  for an introduction to the subject. We often have a lattice $V(\Z)$ in $V$ that we call ``integer points''. Then we can define the Lebesgue measure\footnote{If $X$ is relatively compact in $V$. The volume of $X$ is related to the asymptotic behavior of the number of integral points in $t\cdot X$ when $t$ tends to $\infty$} on $V$ normalized by $V(\Z)$. This measure is given by the pullback of the usual Lebesgue measure under any linear isomorphism $\phi~:~V\to \R^n$, such that $\phi(V(\Z))=\Z^n$.\\ 

We use fibrations, which are locally given by linear maps between convex polytopes. The following lemma is useful at many places in the text and is a basic result of linear algebra.
\begin{lem}
\label{lem_decompo_measure}
Assume that $V_i$ for $i=1,2,3$ are vector spaces with lattices of integer points $V_i(\Z)$. Let $d\sigma_i$ the Lebesgue measure be normalized by $V_i(\Z)$. If $A~:~V_2\longrightarrow V_1$ is a linear map, which induces an exact sequence:
    \begin{equation*}
        \{0\}\longrightarrow V_3(\Z)\longrightarrow V_2(\Z) \longrightarrow V_1(\Z)\longrightarrow \{0\}.
    \end{equation*}
Then the measure $d\sigma_3$ is the conditional measure of $d\sigma_2$ with respect to $d\sigma_1$
\end{lem}
This lemma implies the following fact: if $X_2\subset V_2$ and $X_1\subset V_1$ are polytopes, and $A: X_2\longrightarrow X_1$ is linear such that:
\begin{itemize}
    \item $A(V_2(\Z))=V_1(\Z)$,
    \item $\ker(A)\cap V_2(\Z)$ is a lattice in $\ker(A)$.
\end{itemize}
For $y\in X_1$, let $X_2(y)=A^{-1}(\{y\})$ be the fiber on $y$, and $d\sigma_3(y)$ be the measure on $X_2(y)$ normalized by $\ker(A)\cap V_2(\Z)$. Then for all integrable functions $f$ we have
\begin{equation*}
\label{formule_decomposition_general}
    \int_{X_2} fd\sigma_2=\int_{X_1}\int_{X_2(y)}fd\sigma_3(y)d\sigma_1.
\end{equation*}
This formula also works for maps between spaces that are cells complexes, such that cells are polytopes.

\paragraph{Notations on indices:}
\label{paragraph_notation_indices}
We also use several notations relative to indices. Let $I$ be a
totally ordered set and $L\in \R^{I}$ be a vector. For each subset $J\subset I$ we use the notation $L_{J}=(L_i)_{i\in J}$, ordered according to the order induced on $J$ by $I$ (in many cases, we are considering symmetric functions, so the order is sometimes irrelevant). When there is no possible confusion on $I$, for each $J\subset I$, we denote $L_{J^c}$ the vector $L_{I\backslash J}$. The notation 
\begin{equation*}
    |L|_1=\sum_i |L_i|,
\end{equation*}
is also used. For each $n^+,n^-\in \Npp$ we consider
\begin{equation}
\label{formula_Lambda_1}
    \Lambda_{n^+,n^-}=\{(L^+,L^-)\in \Rp^{n^+}\times\Rp^{n^-}~,~|L^+|_1=|L^-|_1 \},
\end{equation}
and denote $d\sigma_{n^+,n^-}$ the measure on $\Lambda_{n^+,n^-}$, proportional to the Lebesgue measure and normalized by the set of integer points. We have two projections:
\begin{equation*}
L_{\pm}~:~\Lambda_{n^+,n^-}\longrightarrow \Rp^{n^+}.
\end{equation*}
We also denote $E:\Lambda_{n^+,n^-}\to \Rp$ the function 
\begin{equation*}
    E(L)=|L^+|_1=|L^-|_1.
\end{equation*}

\section{Case of non-dirigible ribbon graphs:}

When we consider non-dirigible ribbon graphs with at least one vertex of even degree, we still have degeneration of the symplectic structure, and we can also find canonical curves that spare this vertex from the rest of the surface.

\begin{thm}
If $R$ is any ribbon graph and $v$ a vertex of even degree, there are exactly two admissible multi-curves $\Gamma^{\pm}_v$ such that:
\begin{itemize}
\item $\Gamma^{\pm}_v$ spares $v$ from the rest of the surface.
\item The component $R_v^\pm$ is dirigible and admits a direction such that it is glued along it's negative boundaries.
\item All curves in $\Gamma^{\pm}_v$ are boundaries of $R_v^\pm$.
\end{itemize}
\end{thm}

This theorem contains the case of the theorem \ref{thm_acycl_curve}.

\begin{rem}
In this case, the group of automorphisms of the surface can eventually exchange the two multi-curves $\Gamma^{\pm}_v$. This happens, for instance, for the torus with one boundary.
\end{rem}

\newpage

\bibliography{biblio}
\bibliographystyle{alpha}
\end{document}